\documentclass{amsart}
\usepackage{verbatim}
\usepackage{amssymb}
\usepackage{amsmath}
\usepackage{graphicx}
\usepackage{amscd}
\usepackage{mathtools}
\usepackage[numbers,sort,square]{natbib}
\usepackage[T1]{fontenc}
\usepackage[utf8]{inputenc}

\usepackage[colorlinks,linkcolor={blue},citecolor={blue},urlcolor={red},]{hyperref}

\usepackage{enumerate}

\theoremstyle{plain}
\newtheorem{theorem}{Theorem}

\newtheorem{proposition}[theorem]{Proposition}
\newtheorem{lemma}[theorem]{Lemma}
\newtheorem{corollary}[theorem]{Corollary}

\newtheorem{definition}[theorem]{Definition}

\theoremstyle{definition}

\newtheorem{remark}[theorem]{Remark}

\numberwithin{equation}{section}
\numberwithin{theorem}{section}
\allowdisplaybreaks

\usepackage[colorinlistoftodos,prependcaption,color=yellow,textsize=tiny]{todonotes}

\def\N{{\mathbb N}}

\def\R{{\mathbb R}}

\newcommand{\eps}{\varepsilon}

\newcommand{\dd}[0]{\mathrm{d}}
\newcommand{\ud}[0]{\,\mathrm{d}}

\newcommand{\vertiii}[1]{{\left\vert\kern-0.25ex\left\vert\kern-0.25ex\left\vert #1
    \right\vert\kern-0.25ex\right\vert\kern-0.25ex\right\vert}}

\begin{document}

\title[Burkholder-Davis-Gundy inequalities in UMD Banach spaces]
{Burkholder-Davis-Gundy inequalities\\
in UMD Banach spaces}

\author{Ivan Yaroslavtsev}
\address{Max Planck Institute for Mathematics in the Sciences\\
Inselstra{\ss}e 22\\
04103 Leipzig\\
Germany}
\address{Delft Institute of Applied Mathematics\\
Delft University of Technology \\ P.O. Box 5031\\ 2600 GA Delft\\The
Netherlands}
\email{yaroslavtsev.i.s@yandex.ru}

\begin{abstract}
In this paper we prove Burkholder-Davis-Gundy inequalities for a general martingale $M$ 
with values in a UMD Banach space $X$. Assuming that $M_0=0$, we show that the following two-sided inequality holds for all $1\leq p<\infty$:
\begin{align}\label{eq:main}\tag{{$\star$}}
 \mathbb E \sup_{0\leq s\leq t} \|M_s\|^p \eqsim_{p, X} \mathbb E  \gamma([\![M]\!]_t)^p ,\;\;\; t\geq 0.
\end{align}
Here $ \gamma([\![M]\!]_t) $ is the $L^2$-norm of the unique Gaussian measure on $X$ having  
$[\![M]\!]_t(x^*,y^*):= [\langle M,x^*\rangle, \langle M,y^*\rangle]_t$ 
as its covariance bilinear form. This extends
to general UMD spaces a recent result by Veraar and the author, where a pointwise version of 
\eqref{eq:main} was proved for UMD Banach functions spaces $X$.

We show that for continuous martingales, \eqref{eq:main} holds for all $0<p<\infty$, and that for  
purely discontinuous martingales the right-hand side of \eqref{eq:main} 
can be expressed more explicitly in terms of the jumps of $M$. For martingales with independent increments, \eqref{eq:main} is shown to hold more generally in reflexive Banach spaces $X$ with finite cotype.
In the converse direction, we show that the validity of \eqref{eq:main} for arbitrary martingales 
implies the UMD property for $X$.

As an application we prove various It\^o isomorphisms for vector-valued stochastic integrals with respect to general martingales, which extends earlier results by van Neerven, Veraar, and Weis for vector-valued stochastic integrals with respect to a Brownian motion. We also provide It\^o isomorphisms for vector-valued stochastic integrals with respect to compensated Poisson and general random measures.
\end{abstract}



\maketitle

\tableofcontents

\section{Introduction}

In the celebrated paper \cite{BDG}, Burkholder, Davis, and Gundy proved that if $M = (M_t)_{t\ge 0}$ is a real-valued martingale satisfying $M_0=0$, then for all $1\leq p<\infty$ and $t\ge 0$ one has the two-sided
inequality 
\begin{equation}\label{eq:BDGreal-valyuedintro}
 \mathbb E \sup_{0\leq s\leq t} |M_s|^p \eqsim_p \mathbb E [M]_t^{\frac p2},
\end{equation}
where $[M]$ is the quadratic variation of $M$, i.e., 
\begin{equation}\label{eq:defquadvarintro}
 [M]_t  := \mathbb P-
 \lim_{\textrm{mesh}(\pi)\to 0}\sum_{n=1}^N |M(t_n)-M(t_{n-1})|^2,
\end{equation}
where the limit in probability is taken over partitions $\pi = \{0=t_0 < \ldots < t_N = t\}$
whose mesh approaches $0$. 
Later, Burkholder \cite{Burk91,Burk87} and Kallenberg and Sztencel \cite{KalS91} extended \eqref{eq:BDGreal-valyuedintro} to Hilbert space-valued martingales (see also \cite{MarRo16}). 
They showed that if $M$ is a martingale with values in a Hilbert space $H$ satisfying $M_0=0$, then for all $1\leq p<\infty$
and $t\geq 0$ one has
\begin{equation}\label{eq:BDGHS-valyuedintro}
 \mathbb E \sup_{0\leq s\leq t} \|M_s\|^p \eqsim_p \mathbb E [M]_t^{\frac p2},
\end{equation}
where the quadratic variation $[M]$ is defined as in \eqref{eq:defquadvarintro} with absolute values replaced by norms in $H$. 
A further result along these lines was obtained recently by Veraar and the author
\cite{VY18}, who showed that if $M$ is an $L^p$-bounded martingale, $1<p<\infty$, with $M_0=0$, 
that takes values in a UMD Banach function space $X$ over a measure space $(S,\Sigma,\mu)$ (see Section \ref{subsec:discreteBDG} and \ref{subsec:UMDBanachfs} for the definition), 
then for all $t\geq 0$:
\begin{equation}\label{eq:BDGBFSINTRO}
  \mathbb E \sup_{0\leq s\leq t} \|M_s(\sigma)\|^p \eqsim_{p, X} \mathbb E \bigl\|[M(\sigma)]_t^{\frac 12}\bigr\|^{ p},
\end{equation}
where the quadratic variation $[M(\sigma)]_t$ is considered pointwise in $\sigma \in S$. 
Although this inequality seems to be particularly useful from a practical point of view, it does not give any hint how to work with a general Banach space since not every (UMD) Banach space has a Banach function space structure (e.g.\ noncommutative $L^q$-spaces).

Notice that \eqref{eq:BDGHS-valyuedintro}-type inequalities obtained for general Banach spaces could be of big interest in the area of mathematical physics for the following two reasons. First, vector-valued stochastic analysis is closely tied to vector-valued harmonic analysis; in particular, inequalities of the form \eqref{eq:BDGHS-valyuedintro} could yield sharp bounds for {\em Fourier multipliers} (i.e.\ operators of the form $f\mapsto \mathcal F^{-1}(m\mathcal F f)$, where $\mathcal F$ is the Fourier transform, $\mathcal F^{-1}$ is its inverse, and $m$ is a bounded function). Such operators acting on $L^p$, Sobolev, H\"older, and Besov spaces naturally appear in PDE theory while working with the frequency space (see e.g.\ \cite{Kak70,BChD11,HNVW1,HNVW3,LinPhD,KL04}). A notable example of such an interaction was demonstrated by Bourgain \cite{Bour83} and Burkholder \cite{Burk81} in the case of the Hilbert transform (see also \cite{OY18,Y17FourUMD,Os12}). 

Second, as we will show in Section \ref{subsec:appandmis}, \eqref{eq:BDGHS-valyuedintro} (and its Banach space-valued analogue \eqref{eq:thmBDGgeneralUMDintro}) provides us with sharp bounds for Banach space-valued stochastic integrals with respect to a general martingale. This in turn might be helpful for showing solution existence and uniqueness together with basic $L^p$ estimates for SPDEs containing nongaussian noise regularly exploited in models in physics and economics (such as $\alpha$-stable or general L\'evy processes, see e.g.\ \cite{Fa63,HPZh95}). There is a rich set of instruments (see e.g.\ those for stochastic evolution equations with Wiener noise explored  by van Neerven, Veraar, and Weis in \cite{NVW1}) which could help one to convert Burkholder-Davis-Gundy inequalities and stochastic integral estimates into the corresponding assertions needed. We refer the reader to Section \ref{subsec:appandmis} and \cite{NVW,DPZ,NVW1,Kr97,Kry,GK1,VY16} for further details on stochastic integration in infinite dimensions and its applications in SPDEs.

\medskip

In connection with all of the above the following natural question is rising up.
{\em Given a Banach space $X$. Is there an analogue of \eqref{eq:BDGHS-valyuedintro} for a general $X$-valued local martingale $M$ and how then should the right-hand side of \eqref{eq:BDGHS-valyuedintro} look like?}
In the current article we present the following complete solution to this problem for local martingales $M$ with values in a UMD Banach space $X$.

\begin{theorem}\label{thm:BDGgeneralUMDintro}
 Let $X$ be a UMD Banach space. Then for any local martingale $M:\mathbb R_+\times \Omega \to X$ with $M_0=0$ and any $t\geq 0$ the covariation bilinear form $[\![M]\!]_t$ is well-defined and bounded almost surely, and for all $1\leq p<\infty$ we have
 \begin{equation}\label{eq:thmBDGgeneralUMDintro}
   \mathbb E \sup_{0\leq s\leq t}\|M_s\|^p \eqsim_{p,X} \mathbb E  \gamma([\![M]\!]_t)^p .
 \end{equation}
\end{theorem}

Here $\gamma(V)$, where $V:X^*\times X^* \to \mathbb R$ is a given nonnegative symmetric bilinear form, is the $L^2$-norm of an $X$-valued Gaussian random variable $\xi$ with 
$$
\mathbb E \langle \xi, x^*\rangle^2 = V(x^*, x^*), \;\;\;x^*\in X^*.
$$ 
We call $\gamma(V)$ the {\em Gaussian characteristic} of $V$ (see Section \ref{sec:Gaussiancharacteristic}).

Let us explain briefly the main steps of the proof of Theorem \ref{thm:BDGgeneralUMDintro}. This discussion will also clarify the meaning 
of the term on the right-hand side, which is equivalent 
to the right-hand side of \eqref{eq:BDGHS-valyuedintro} if $X$ is a Hilbert space, and of \eqref{eq:BDGBFSINTRO} (up to a multiplicative constant) 
if $X$ is a UMD Banach function space.

In Section \ref{subsec:discreteBDG} we start by proving the discrete-time version of Theorem \ref{thm:BDGgeneralUMDintro}, which takes the following simple form
\begin{equation}\label{eq:genBDGdiscreteintro}
  \mathbb E \sup_{1\le m\le N} \Bigl\|\sum_{n=1}^m d_n\Bigr\|^p \eqsim_{p,X} \mathbb E \Bigl(\mathbb E_\gamma  \Big\|\sum_{n=1}^N \gamma_n d_n\Bigr\|^2 \Bigr)^{\frac p2},
\end{equation}
where $(d_n)_{n= 1}^N$ is an $X$-valued martingale difference sequence and $(\gamma_n)_{n= 1}^N$
is a sequence of independent standard Gaussian random variables defined on a probability space $(\Omega_\gamma,\mathbb  P_\gamma)$. 
\eqref{eq:genBDGdiscreteintro} follows from a decoupling inequality due to Garling \cite{Gar85} and a martingale transform inequality due to Burkholder \cite{Burk86} (each of which holds if and only if $X$ has the UMD property) together with the equivalence of Rademacher and Gaussian random sums with values in spaces with finite cotype due to Maurey and Pisier (see \cite{MP76}).

Theorem \ref{thm:BDGgeneralUMDintro} is derived from \eqref{eq:genBDGdiscreteintro} by finite-dimensional approximation and discretization. This is a rather intricate procedure and depends on 
some elementary, but nevertheless important properties of a Gaussian characteristic $\gamma(\cdot)$. 
In particular in Section \ref{sec:Gaussiancharacteristic} we show that for a finite dimensional Banach space $X$ there exists a proper continuous extension of the Gaussian characteristic to all (not necessarily nonnegative) symmetric bilinear forms $V:X^* \times X^* \to \mathbb R$, with the bound 
$$
(\gamma(V))^2\lesssim_X  \sup_{\|x^*\|\leq 1}V(x^*, x^*).
$$ 

Next, in Section \ref{sec:BDG-cont}, under the assumptions of Theorem \ref{thm:BDGgeneralUMDintro} we show that $M$ has a well-defined {\em covariation bilinear form}, i.e.\ for each $t\geq 0$ and for almost all $\omega\in \Omega$ there exists a symmetric bilinear form $[\![M]\!]_t(\omega):X^* \times X^* \to \mathbb R$ such that for all $x^*, y^*\in X^*$ one has
\[
 [\![M]\!]_t(x^*, y^*) = [\langle M, x^*\rangle,\langle M, y^*\rangle]_t \;\;\; \text{a.s.}
\]
Existence of such a covariance bilinear form in the nonhilbertian setting used to be an open problem since 1970's (see e.g.\ Meyer \cite[p.\ 448]{Mey77} and M\'{e}tivier \cite[p.\ 156]{MetSemi}; see also \cite{GP74,VY16,SC02,BDMKR}). In Section \ref{sec:BDG-cont} we show that such a covariation exists in the UMD case. Moreover, in Proposition \ref{prop:cadlag[[M]]} we show that the process $[\![M]\!]$ has an {\em increasing adapted  c\`adl\`ag} version.

Next we prove that the bilinear form $[\![M]\!]_t(\omega)$ has a finite Gaussian characteristic $\gamma([\![M]\!]_t)$ for almost all $\omega\in \Omega$.
After these preparations we prove Theorem \ref{thm:BDGgeneralUMDintro}. We also show that the UMD property is necessary for the conclusion of the theorem to hold true (see Subsection \ref{subsec:necofUMDgeneralrhs}).

\smallskip
In Section \ref{sec:ramifications} we develop three ramifications of our main result:
\begin{itemize}
 \item if $M$ is continuous, the conclusion of Theorem \ref{thm:BDGgeneralUMDintro}
holds for all $0< p< \infty$. 
 \item if $M$ is purely discontinuous, the theorem can be reformulated in terms of the jumps of $M$. 
\item if $M$ has independent increments, the UMD assumption on $X$ can be weakened to reflexivity and finite cotype.
\end{itemize}
The first two cases are  particularly important in view of the fact that any UMD space-valued local martingale has a unique {\em Meyer-Yoeurp decomposition} into a sum of a continuous local martingale and a purely discontinuous local martingale (see \cite{Y17MartDec,Y17GMY}).

\smallskip
A reasonable part of the paper, namely Section \ref{subsec:appandmis}, is devoted to applications of Theorem \ref{thm:BDGgeneralUMDintro} and results related to Theorem \ref{thm:BDGgeneralUMDintro}. Let us outline some of them. In Subsection \ref{subsec:applicItoisomgenmart} we develop a
theory of vector-valued stochastic integration.
Our starting point is a result of 
van Neerven, Veraar, and Weis \cite{NVW}. They proved that if $W_H$ is a cylindrical Brownian motion in a Hilbert space $H$ and $\Phi:\mathbb R_+\times \Omega \to \mathcal L(H, X)$ is an elementary predictable process, then for all $0<p<\infty$ and $t\ge 0$ one has the two-sided inequality
\begin{equation}\label{eq:BDGNVWintro}
 \mathbb E \sup_{0\leq s\leq t} \Bigl\|\int_0^s \Phi \ud W_H\Bigr\|^p \eqsim_{p, X} \mathbb E \|\Phi\|_{\gamma(L^2([0,t]; H),X)}^p.
\end{equation}
Here $\|\Phi\|_{\gamma(L^2([0,t]; H),X)}$ is the {\em $\gamma$-radonifying norm} of $\Phi$ 
as an operator from a Hilbert space $L^2([0,t]; H)$ into $X$
(see \eqref{eq:defofgammanormsnove} for the definition); this norm coincides with the Hilbert-Schmidt norm 
given $X$ is a Hilbert space. This result was extended  to continuous local martingales in \cite{VY16,Ver}. 

Theorem \ref{thm:BDGgeneralUMDintro} directly implies \eqref{eq:BDGNVWintro}.
More generally, if $M = \int \Phi \ud \widetilde M$ for some $H$-valued martingale $\widetilde M$ and elementary predictable process $\Phi:\mathbb R_+\times \Omega \to \mathcal L(H, X)$, then
it follows from Theorem \ref{thm:BDGgeneralUMDintro} that for all $1\leq p<\infty$ and $t\ge 0$ one has
\begin{equation}\label{eq:stochintwrtgenmartINTRO}
 \mathbb E \sup_{0\leq s\leq t} \Bigl\|\int_0^s\Phi\ud \widetilde M\Bigr\|^p \eqsim_{p,X} \mathbb E \|\Phi q_{\widetilde M}^{1/2}\|^p_{\gamma(L^2(0,t;[\widetilde M]), X)}.
\end{equation}
Here $q_{\widetilde M}$ is the quadratic variation derivative of $\widetilde M$
and $\gamma(L^2(0,t;[\widetilde M]), X)$ is a suitable space of $\gamma$-radonifying operator associated
with $\widetilde M$ (see Subsection \ref{subsec:applicItoisomgenmart} for details).
This represents a significant improvement of \eqref{eq:BDGNVWintro}.

In Subsection \ref{subsec:ItoisomPrmandgrm} we apply our results to vector-valued stochastic integrals with respect to a compensated Poisson random measure $\widetilde N$. We show that if $N$ is a Poisson random measure on $\mathbb R_+ \times J$ for some measurable space $(J, \mathcal J)$, $\nu$ is its compensator, $\widetilde N := N - \nu$ is the corresponding compensated Poisson random measure, then for any UMD Banach space $X$, for any elementary predictable $F:J\times \mathbb R_+ \times \Omega \to X$, and for any $1\leq p<\infty$ one has that
\begin{equation}\label{eq:ItoisomforPoissonINTRO}
    \mathbb E \sup_{0\leq s\leq t} \Bigl\|\int_{J\times [0,s]} F\ud \widetilde N\Bigr\|^p \eqsim_{p, X} \mathbb E \|F\|^p_{\gamma(L^2(J\times [0,t]; N), X)},\;\;\;\; t\geq 0.
\end{equation}
We also show that \eqref{eq:ItoisomforPoissonINTRO} holds if one considers a general quasi-left continuous random measure $\mu$ instead of $N$. 

In Subsection \ref{sec:martdom} we prove the following {\em martingale domination inequality}: for all local martingales $M$ and $N$ with values in a UMD Banach space $X$ such that 
$$
\|N_0\|\leq \|M_0\|\;\;\;\text{a.s.,}$$
and
$$[\langle N, x^*\rangle]_{\infty} \leq [\langle M, x^*\rangle]_{\infty} \ \ \hbox{almost surely, for all $x^*\in X^*$},$$
for all $1\le p<\infty $ we have that
\begin{equation*}
  \mathbb E \sup_{t\geq 0} \|N_t\|^p \lesssim_{p,X} \mathbb E \sup_{t\geq 0} \|M_t\|^p.
\end{equation*}
This extends {\em weak differential subordination} $L^p$-estimates obtained in \cite{Y17FourUMD,Y17MartDec} (which used to be known to hold only for $1<p<\infty$, see \cite{Y17FourUMD,Y17MartDec,OY18}).

Finally, in Section \ref{subsec:UMDBanachfs}, we prove that for any UMD Banach function space $X$ over a measure space $(S, \Sigma, \mu)$, that any $X$-valued local martingale $M$ has a pointwise local martingale version $M(\sigma)$, $\sigma\in S$, such that if $1\leq p<\infty$, then for $\mu$-almost all $\sigma\in S$ one has
\[
 \mathbb E \sup_{0\leq s\leq t} \|M_s(\sigma)\|^p \eqsim_{p, X} \mathbb E \bigl\|[M(\sigma)]_t^{\frac 12}\bigr\|^{ p}
\]
for all $t\geq 0$, which extends \eqref{eq:BDGBFSINTRO} to the case $p=1$ and general local martingales. 

\smallskip

In conclusion we wish to notice that it remains open whether one can find a {\em predictable} right-hand side in \eqref{eq:thmBDGgeneralUMDintro}: so far such a predictable right-hand side was explored only in the real-valued case and in the case $X = L^q(S)$, $1<q<\infty$, see {\em Burkholder-Novikov-Rosenthal inequalities} in the forthcoming paper \cite{DMY18}. This problem might be resolved via using recently discovered {\em decoupled tangent martingales}, see \cite{Y19}.

\section{Burkholder-Davis-Gundy inequalities: the discrete time case}\label{subsec:discreteBDG}

Let us show discrete Burkholder-Davis-Gundy inequalities.  First we will provide the reader with the definitions of UMD Banach spaces and $\gamma$-radonifying operators.
A Banach space $X$ is called a {\em UMD space} if for some (equivalently, for all)
$p \in (1,\infty)$ there exists a constant $\beta>0$ such that
for every $n \geq 1$, every martingale
difference sequence $(d_j)^n_{j=1}$ in $L^p(\Omega; X)$, and every $\{-1,1\}$-valued sequence
$(\varepsilon_j)^n_{j=1}$
we have
\[
\Bigl(\mathbb E \Bigl\| \sum^n_{j=1} \varepsilon_j d_j\Bigr\|^p\Bigr )^{\frac 1p}
\leq \beta \Bigl(\mathbb E \Bigl \| \sum^n_{j=1}d_j\Bigr\|^p\Bigr )^{\frac 1p}.
\]
The least admissible constant $\beta$ is denoted by $\beta_{p,X}$ and is called the {\em UMD constant}. It is well known (see \cite[Chapter 4]{HNVW1}) that $\beta_{p, X}\geq p^*-1$ and that $\beta_{p, H} = p^*-1$ for a Hilbert space $H$. 
 We refer the reader to \cite{Burk01,HNVW1,Rubio86,Pis16,LVY18,HNVW2,GM-SS} for details.
 
 Let $H$ be a separable Hilbert space, $X$ be a Banach space, $T\in \mathcal L(H, X)$. Then $T$ is called {\em $\gamma$-radonifying} if
\begin{equation}\label{eq:defofgammanormsnove}
 \|T\|_{\gamma(H,X)} := \Bigl(\mathbb E \Bigl\|\sum_{n=1}^{\infty} \gamma_n Th_n\Bigr\|^2\Bigr)^{\frac 12} <\infty,
\end{equation}
where $(h_n)_{n\geq 1}$ is an orthonormal basis of $H$, and $(\gamma_n)_{n\geq 1}$ is a sequence of standard Gaussian random variables (otherwise we set $\|T\|_{\gamma(H,X)}:=\infty$). Note that $\|T\|_{\gamma(H,X)}$ {\em does not depend} on the choice of $(h_n)_{n\geq 1}$ (see \cite[Section 9.2]{HNVW2} and \cite{Ngamma} for details). Often we will call $\|T\|_{\gamma(H, X)}$ the {\em $\gamma$-norm} of $T$.
$\gamma$-norms are exceptionally important in analysis as they are easily computable and enjoy a number of useful properties such as the ideal property, $\gamma$-multiplier theorems, Fubini-type theorems, etc., see \cite{HNVW2,Ngamma}.

Now we are able state and prove discrete UMD-valued Burkholder-Davis-Gundy inequalities.

\begin{theorem}\label{thm:BDGgendiscmart}
 Let $X$ be a UMD Banach space, $(d_n)_{n\geq 1}$ be an $X$-valued martingale difference sequence. Then for any $1\leq p<\infty$
 \begin{equation}\label{eq:discBDGUMDBanach}
   \mathbb E \sup_{m\geq 1}\Bigl\|\sum_{n=1}^{m} d_n\Bigr\|^p \eqsim_{p, X} \mathbb E \|(d_n)_{n=1}^{\infty}\|_{\gamma(\ell^2, X)}^p.
 \end{equation}
\end{theorem}

For the proof we will need {\em Rademacher random variables}.

\begin{definition}
 A real-valued random variable $r$ is called {\em Rademacher} if $\mathbb P(r=1) = \mathbb P(r=-1) = 1/2$.
\end{definition}

\begin{proof}[Proof of Theorem \ref{thm:BDGgendiscmart}]
Without loss of generality we may assume that there exists $N\geq 1$ such that $d_n=0$ for all $n>N$. Let $(r_n)_{n\geq 1}$ be a sequence of independent Rademacher random variables, $(\gamma_n)_{n\geq 1}$ be a sequence of independent standard Gaussian random variables. Then
 \begin{align}\label{eq:Lpnormofdiscreteintermsofgammanorm}
  \mathbb E \sup_{m\geq 1}\Bigl\|\sum_{n=1}^{m} d_n\Bigr\|^p &\stackrel{(i)}\eqsim_{p, X} \mathbb E   \mathbb E_{r}\sup_{m\geq 1} \Bigl\|\sum_{n=1}^{N} r_n d_n\Bigr\|^p \stackrel{(ii)}\eqsim_{p}\mathbb E\mathbb E_{r} \Bigl\|\sum_{n=1}^{N} r_n d_n\Bigr\|^p\nonumber\\
  &\stackrel{(iii)}\eqsim_{p,X}\mathbb E \mathbb E_{\gamma}\Bigl\|\sum_{n=1}^{N} \gamma_n d_n\Bigr\|^p \stackrel{(iv)}\eqsim_{p}\mathbb E \Bigl(\mathbb E_{\gamma} \Bigl\|\sum_{n=1}^{N} \gamma_n d_n\Bigr\|^2\Bigr)^{\frac p2}\\
  &= \mathbb E \|(d_n)_{n=1}^{\infty}\|_{\gamma(\ell^2, X)}^p,\nonumber
 \end{align}
where $(i)$ follows from \cite[(8.22)]{Burk86}, $(ii)$ holds by \cite[Proposition 6.1.12]{HNVW2}, $(iii)$ follows from \cite[Corollary 7.2.10 and Proposition 7.3.15]{HNVW2}, and $(iv)$ follows from \cite[Proposition 6.3.1]{HNVW2}.
\end{proof}

\begin{remark}\label{rem:deponconstants}
Note that if we collect all the constants in \eqref{eq:Lpnormofdiscreteintermsofgammanorm}, then the final constant will depend only on $p$ and $\beta_{2, X}$ (or $\beta_{q, X}$ for any fixed $1<q<\infty$).
\end{remark}

\begin{remark}\label{rem:c_ppXandC_p'Xareindepofp}
 If we collect all the constants in \eqref{eq:Lpnormofdiscreteintermsofgammanorm} then one can see that those constants behave well as $p\to 1$, i.e.\ for any $1<r<\infty$ there exist positive $C_{r, X}$ and $c_{r,X}$ such that for any $1\leq p\leq r$
 \[
   c_{r, X} \mathbb E \|(d_n)_{n=1}^{\infty}\|_{\gamma(\ell^2, X)}^p \leq \mathbb E \sup_{m\geq 1}\Bigl\|\sum_{n=1}^{m} d_n\Bigr\|^p \leq C_{r, X} \mathbb E \|(d_n)_{n=1}^{\infty}\|_{\gamma(\ell^2, X)}^p.
 \]
\end{remark}

\begin{remark}
 Fix $1<p<\infty$ and a UMD Banach space $X$. By Doob's maximal inequality \eqref{eq:DoobsineqXBanach} and Theorem \ref{thm:BDGgendiscmart} we have that
 \[
 \mathbb E \Bigl\|\sum_{n=1}^{\infty} d_n\Bigr\|^p \eqsim_{p}  \mathbb E \sup_{m\geq 1}\Bigl\|\sum_{n=1}^{m} d_n\Bigr\|^p \eqsim_{p, X}\mathbb E \|(d_n)_{n=1}^{\infty}\|_{\gamma(\ell^2, X)}^p.
 \]
Let us find the constants in the equivalence 
$$
\mathbb E \Bigl\|\sum_{n=1}^{\infty} d_n\Bigr\|^p \eqsim_{p, X}\mathbb E \|(d_n)_{n=1}^{\infty}\|_{\gamma(\ell^2, X)}^p.
$$
Since $X$ is UMD, it has a finite cotype $q$ (see \cite[Definition 7.1.1. and Proposition 7.3.15]{HNVW2}), and therefore by modifying \eqref{eq:Lpnormofdiscreteintermsofgammanorm} (using decoupling inequalities \cite[p.\ 282]{HNVW1} instead of \cite[(8.22)]{Burk86} and \cite[Proposition 6.1.12]{HNVW2}) one can show that
 \begin{multline*}
  \frac{1}{\beta_{p, X}c_{p, X}} \Bigl(\mathbb E \|(d_n)_{n=1}^{\infty}\|_{\gamma(\ell^2, X)}^p\Bigr)^{\frac 1p} \leq \Bigl(\mathbb E\Bigl\|\sum_{n=1}^{m} d_n\Bigr\|^p\Bigr)^{\frac 1p}\\
  \leq 2 \beta_{p, X} \kappa_{p,2} \Bigl(\mathbb E \|(d_n)_{n=1}^{\infty}\|_{\gamma(\ell^2, X)}^p\Bigr)^{\frac 1p},
 \end{multline*}
 where $c_{p, X}$ depends on $p$, the cotype of $X$, and the Gaussian cotype constant of $X$ (see \cite[Proposition 7.3.15]{HNVW2}), while $\kappa_{p,q}$ is the Kahane-Khinchin constant (see \cite[Section 6.2]{HNVW2}).
\end{remark}

\begin{remark}\label{rem:phidisccase}
 Theorem \ref{thm:BDGgendiscmart} can be extended to general convex functions. Indeed, let $X$ be a UMD Banach space, $\phi:\mathbb R_+ \to \mathbb R_+$ be a convex increasing function such that $\phi(0)=0$ and 
 \begin{equation}\label{phi2l<cphil}
  \phi(2\lambda) \leq c\phi(\lambda),\;\;\;\lambda\geq 0,
 \end{equation}
for some fixed $c>0$. Then from a~standard good-$\lambda$ inequality argument due to Burkholder (see \cite[Remark 8.3]{Burk86}, \cite[Lemma 7.1]{Bur73}, and \cite[pp.\ 1000--1001]{Burk81})  we imply that
\begin{align}\label{eq:discBDGforgenconvfunction}
  \mathbb E \phi\Bigl( \sup_{m\geq 1}\Bigl\|\sum_{n=1}^{m} d_n\Bigr\|\Bigr) &\stackrel{(i)}\eqsim_{\phi, X} \mathbb E   \mathbb E_{r}\phi\Bigl(\sup_{m\geq 1} \Bigl\|\sum_{n=1}^{N} r_n d_n\Bigr\|\Bigr) \stackrel{(ii)}\eqsim_{\phi}\mathbb E\mathbb E_{r} \phi\Bigl(\Bigl\|\sum_{n=1}^{N} r_n d_n\Bigr\|\Bigr)\nonumber\\
  &\stackrel{(iii)}\eqsim_{\phi,X}\mathbb E \mathbb E_{\gamma}\phi\Bigl(\Bigl\|\sum_{n=1}^{N} \gamma_n d_n\Bigr\|\Bigr) \stackrel{(iv)}\eqsim_{\phi}\mathbb E \phi\Bigl(\mathbb E_{\gamma} \Bigl\|\sum_{n=1}^{N} \gamma_n d_n\Bigr\|\Bigr)\\
  &\stackrel{(v)}\eqsim_{\phi}\mathbb E \phi\Bigl(\Bigl(\mathbb E_{\gamma} \Bigl\|\sum_{n=1}^{N} \gamma_n d_n\Bigr\|^2\Bigr)^{\frac 12}\Bigr)= \mathbb E\phi\Bigl( \|(d_n)_{n=1}^{\infty}\|_{\gamma(\ell^2, X)}\Bigr),\nonumber
 \end{align}
 where $(i)$ and $(iii)$ follow from good-$\lambda$ inequalities \cite[(8.22)]{Burk86}, $(ii)$ follows from \cite[Proposition 6.1.12]{HNVW2}, $(iv)$ holds by \cite[Corollary 2.7.9]{dlPG}, Doob's maximal inequality \eqref{eq:DoobsineqXBanach}, and \eqref{phi2l<cphil}, and $(v)$ follows from \eqref{phi2l<cphil} and Kahane-Khinchin inequalities \cite[Theorem 6.2.6]{HNVW2}. Note that as in Remark \ref{rem:deponconstants} the final constant in \eqref{eq:discBDGforgenconvfunction} will depend only on $\phi$ and $\beta_{2, X}$ (or $\beta_{q, X}$ for any fixed $1<q<\infty$).
\end{remark}

In the following theorem we show that $X$ having the UMD property is necessary for Theorem \ref{thm:BDGgendiscmart} to hold.

\begin{theorem}\label{thm:necofUMDfordiscBDGtango}
 Let $X$ be a Banach space and $1\leq p<\infty$ be such that \eqref{eq:discBDGUMDBanach} holds for any martingale difference sequence $(d_n)_{n\geq 1}$. Then $X$ is UMD.
\end{theorem}

\begin{proof}
Note that for any set $(x_n)_{n=1}^N$ of elements of $X$ and for any $[-1,1]$-valued sequence $(\eps_n)_{n=1}^N$ we have that $\|(\eps_n x_n)_{n=1}^N\|_{\gamma(\ell^2_N, X)} \leq \|(x_n)_{n=1}^N\|_{\gamma(\ell^2_N, X)}$ by the ideal property (see \cite[Theorem 9.1.10]{HNVW2}). Therefore if \eqref{eq:discBDGUMDBanach} holds for any $X$-valued martingale difference sequence $(d_n)_{n\geq 1}$, then we have that for any $[-1,1]$-valued sequence $(\eps_n)_{n\geq 1}$
\begin{equation}\label{eq:discBDGimpliesUMDhahaha}
  \mathbb E \sup_{m\geq 1}\Bigl\|\sum_{n=1}^m \eps_n d_n\Bigr\|^p \lesssim_{p, X}\mathbb E \sup_{m\geq 1}\Bigl\|\sum_{n=1}^m d_n\Bigr\|^p.
\end{equation}
If $p>1$, then \eqref{eq:discBDGimpliesUMDhahaha} together with \eqref{eq:DoobsineqXBanach} implies the UMD property. If $p=1$, then \eqref{eq:discBDGimpliesUMDhahaha} for $p=1$ implies \eqref{eq:discBDGimpliesUMDhahaha} for any $p>1$ (see \cite[Theorem 3.5.4]{HNVW1}), and hence it again implies UMD.
\end{proof}

Now we turn to the continuous-time case. It turns out that in this case the right-hand side of \eqref{eq:discBDGUMDBanach} transforms to a so-called {\em Gaussian characteristic} of a certain bilinear form generated by a quadratic variation of the corresponding martingale. Therefore before proving our main result (Theorem \ref{thm:BDGgeneralUMD}) we will need to outline some basic properties of a Gaussian characteristic (see Section \ref{sec:Gaussiancharacteristic}). We will also need some preliminaries concerning continuous-time Banach space-valued martingales (see Section \ref{sec:prelim}).

\section{Gaussian characteristics}\label{sec:Gaussiancharacteristic}
The current section is devoted to the definition and some basic properties of one of the main objects of the paper -- a Gaussian characteristic of a bilinear form. Many of the statements here might seem to be obvious for the reader. Nevertheless we need to show them before reaching our main Theorem \ref{thm:BDGgeneralUMD}.

\subsection{Basic definitions}\label{subsec:gausscharactbasicdefinitions}

Let us first recall some basic facts on Gaussian measures.
Let $X$ be a Banach space. An $X$-valued random variable $\xi$ is called {\em Gaussian} if $\langle \xi, x^*\rangle$ has a Gaussian distribution for all $x^*\in X^*$. Gaussian random variables enjoy a number of useful properties (see \cite{BogGaus,Kuo}). We will need the following {\em Gaussian covariance domination inequality} (see \cite[Corollary 3.3.7]{BogGaus} and \cite[Theorem 6.1.25]{HNVW2} for the case $\phi = \|\cdot\|^p$).

\begin{lemma}\label{lem:Gauscovinequality}
 Let $X$ be a Banach space, $\xi,\eta$ be centered $X$-valued Gaussian random variables. Assume that $\mathbb E \langle \eta, x^*\rangle^2 \leq \mathbb E \langle \xi, x^*\rangle^2 $ for all $x^* \in X^*$. Then $\mathbb E \phi(\eta) \leq \mathbb E \phi(\xi)$ for any convex symmetric continuous function $\phi:X \to\mathbb R_+$.
\end{lemma}

\smallskip

Let $X$ be a Banach space. We denote the linear space of all continuous $\mathbb R$-valued bilinear forms on $X\times X$ by $X^*\otimes X^*$. Note that this linear space can be endowed with the following natural norm:
\begin{equation}\label{eq:normofbilinearfromdefvoina}
  \|V\|:= \sup_{x\in X, \|x\|\leq 1} |V(x,x)|,
\end{equation}
where the latter expression is finite due to bilinearity and continuity of $V$. A bilinear form $V$ is called {\em nonnegative} if $V(x, x)\geq 0$ for all $x\in X$, and $V$ is called {\em symmetric} if $V(x, y) = V(y, x)$ for all $x, y\in X$.

\smallskip

Let $X$ be a Banach space, $\xi$ be a centered $X$-valued Gaussian random variable. Then $\xi$ has a {\em covariance bilinear form} $V:X^*\times X^*\to \mathbb R$ such that
\[
 V(x^*, y^*) = \mathbb E \langle \xi, x^*\rangle\langle \xi, y^*\rangle,\;\;\; x^*, y^*\in X.
\]
Notice that a covariance bilinear form is always continuous, symmetric, and nonnegative. It is worth noticing that one usually considers a {\em covariance operator} $Q:X^* \to X^{**}$ defined by
\[
 \langle Q x^*, y^*\rangle = \mathbb E \langle \xi, x^*\rangle\langle \xi, y^*\rangle,\;\;\; x^*, y^*\in X.
\]
But since there exists a simple one-to-one correspondence between bilinear forms and $\mathcal L(X^*, X^{**})$, we will work with covariance bilinear forms instead.
We refer the reader to \cite{BogGaus,DPZ,GvN,vN98} for details.

\smallskip

Let $V:X^*\times X^*\to \mathbb R$ be a symmetric continuous nonnegative bilinear form. Then $V$ is said to have a finite {\em Gaussian characteristic} $\gamma(V)$ if there exists a centered $X$-valued Gaussian random variable $\xi$ such that $V$ is the covariance bilinear form of $\xi$.
Then we set $\gamma(V) := (\mathbb E \|\xi\|^2)^{\frac 12}$ (this value is finite due to the Fernique theorem, see \cite[Theorem 2.8.5]{BogGaus}). Otherwise we set $\gamma(V) = \infty$. Note that then for all $x^*, y^*\in X^*$ one has the following control of continuity of $V$:
\begin{equation}\label{eq:contofVduetogammanormtikinula}
 \begin{multlined}
  |V(x^*, x^*)^{\frac{1}{2}} - V(y^*, y^*)^{\frac{1}{2}}| = (\mathbb E |\langle \xi, x^*\rangle|^2)^{\frac 12} - (\mathbb E |\langle \xi, y^*\rangle|^2)^{\frac 12}\\
 \leq (\mathbb E |\langle \xi, x^* - y^*\rangle|^2)^{\frac 12} \leq  (\mathbb E \|\xi\|^2)^{\frac 12}\|x^*-y^*\| = \|x^*-y^*\|\gamma(V).
 \end{multlined}
\end{equation}

\begin{remark}
 Note that for any $V$ with $\gamma(V)<\infty$ the distribution of the corresponding centered $X$-valued Gaussian random variable $\xi$ is {\em uniquely determined} (see \cite[Chapter 2]{BogGaus}).
\end{remark}

\begin{remark}\label{rem:fdinmpliesgamma(V)<infty}
 Note that if $X$ is finite dimensional, then $\gamma(V)<\infty$ for any nonnegative symmetric bilinear form $V$. Indeed, in this case $X$ is isomorphic to a finite dimensional Hilbert space $H$, so there exists an eigenbasis $(h_n)_{n=1}^d$ making $V$ diagonal, and then the corresponding Gaussian random variable will be equal to $\xi := \sum_{n=1}^d V(h_n, h_n) \gamma_n h_n$, where $(\gamma_n)_{n=1}^d$ are independent standard Gaussian.
\end{remark}

\subsection{Basic properties of $\gamma(\cdot)$}

Later we will need the following technical lemmas.

\begin{lemma}\label{lem:TfromVizveniiiilsya}
 Let $X$ be a reflexive (separable) Banach space, $V:X^* \times X^* \to \mathbb R$ be a symmetric continuous nonnegative bilinear form. Then there exist a (separable) Hilbert space $H$ and $T\in \mathcal L(H,X)$ such that 
 $$
 V(x^*, y^*) = \langle T^* x^*, T^*y^*\rangle,\;\;\; x^*,y^*\in X^*.
 $$
\end{lemma}

\begin{proof}
See \cite[pp.\ 57-58]{BN} or \cite[p.\ 154]{Kuo}.
\end{proof}

The following lemma connects Gaussian characteristics and $\gamma$-norms (see \eqref{eq:defofgammanormsnove}) and it can be found e.g.\ in \cite[Theorem 7.4]{Ngamma} or in \cite{BN,NW1}.

\begin{lemma}\label{lem:gammaofV=gammaofTfits}
 Let $X$ be a separable Banach space, $H$ be a separable Hilbert space, $T\in \mathcal L(H, X)$, $V:X^*\times X^* \to \mathbb R$ be  a symmetric continuous nonnegative bilinear form such that $V(x^*, y^*) = \langle T^* x^*, T^*y^*\rangle$ for all $x^*,y^*\in X^*$. Then $\gamma(V) =\|T\|_{\gamma(H,X)}$.
\end{lemma}

\begin{remark}
 Fix a Hilbert space $H$ and a Banach space $X$. Note that even though by the lemma above there exists a natural embedding of $\gamma$-radonifying operators from $\mathcal L(H, X)$ to the space of symmetric nonnegative bilinear forms on $X^*\times X^*$, this embedding is neither {\em injective} nor {\em linear}. This also explains why we need to use bilinear forms with finite Gaussian characteristics instead of $\gamma$-radonifying operators: in the proof of our main result -- Theorem \ref{thm:BDGgeneralUMD} -- we will need various statements (like triangular inequalities and convergence theorems) for {\em bilinear forms}, not {\em operators}.
\end{remark}

Now we will prove some statements about approximation of nonnegative symmetric bilinear forms by finite dimensional ones in $\gamma(\cdot)$.

\begin{lemma}\label{lem:V_0isonsubsetofVsoitsgammaislettznakomo}
 Let $X$ be a reflexive Banach space, $Y \subset X^*$ be a finite dimensional subspace. Let $P:Y \hookrightarrow X^*$ be an inclusion operator. Let $V:X^* \times X^* \to \mathbb R$ and $V_0:Y\times Y \to \mathbb R$ be symmetric continuous nonnegative bilinear forms such that $V_0(x_0^*, y_0^*) = V(Px_0^*, Py_0^*)$ for all $x_0^*, y_0^*\in Y$. Then $\gamma(V_0)$ is well-defined and $\gamma(V_0) \leq \gamma(V)$.
\end{lemma}

\begin{proof}
 First of all notice that $\gamma(V_0)$ is well-defined since $Y$ is finite dimensional, hence reflexive, and thus has a predual space coinciding with its dual. Without loss of generality assume that $\|V\|_{\gamma}<\infty$. Let $\xi_V$ be a centered $X$-valued Gaussian random variable with $V$ as the covariance bilinear form. Define $\xi_{V_0}:= P^*\xi_V$ (note that $Y^*\hookrightarrow X$ due to the Hahn-Banach theorem). Then for all $x_0^*, y_0^* \in X_0^*$
 \[
  \mathbb E \langle \xi_{V_0}, x_0^*\rangle\langle\xi_{V_0}, y_0^*\rangle= \mathbb E \langle \xi_{V}, Px_0^*\rangle\langle\xi_{V}, Py_0^*\rangle = V(Px_0^*, Py_0^*) = V_0(x_0^*, y_0^*),
 \]
so $V_0$ is the covariance bilinear form of $\xi_{V_0}$ and since $\|P^*\| = \|P\|=1$
\begin{equation}\label{eq:inlemV_0isonsubsetofVsoitsgammaisllastguy}
 \gamma(V_0) = (\mathbb E \|\xi_{V_0}\|^2)^{\frac 12}  = (\mathbb E \|P^*\xi_{V}\|^2)^{\frac 12} \leq (\mathbb E \|\xi_{V}\|^2)^{\frac 12} = \gamma(V).
\end{equation}
\end{proof}

\begin{proposition}\label{prop:convergenceofrestrictedV_m}
 Let $X$ be a separable reflexive Banach space, $V:X^* \times X^* \to \mathbb R$ be a symmetric continuous nonnegative bilinear form. Let $Y_1\subset Y_2\subset \ldots \subset Y_m\subset \ldots$ be a sequence of finite dimensional subspaces of $X^*$ with $\overline{\cup_m Y_m}=X^*$. Then for each $m\geq 1$ a symmetric continuous nonnegative bilinear form $V_m = V|_{Y_m\times Y_m}$ is well-defined and $\gamma(V_m) \to \gamma(V)$ as $m\to \infty$.
\end{proposition}

\begin{proof}
 First of all notice that $V_m$'s are well-defined since each of the  $Y_m$ is finite dimensional, hence reflexive, and thus has a predual space coinciding with its dual (which we will call $X_m$ and which can even be embedded into $X$ due to the Hahn-Banach theorem). Let $P_m:Y_m \hookrightarrow X^*$ be the inclusion operator (thus is particular $\|P_m\|\leq 1$). Let a Hilbert space $H$ and an operator $T\in \mathcal L(H,X)$ be as constructed in Lemma \ref{lem:TfromVizveniiiilsya}. Let $(h_n)_{n\geq 1}$ be an orthonormal basis of $H$, and $(\gamma_n)_{n\geq 1}$ be a sequence of standard Gaussian random variables. For each $N\geq 1$ define a centered Gaussian random variable $\xi_N :=\sum_{n=1}^{N} \gamma_n Th_n$. Then for each $m\geq 1$ the centered Gaussian random variable $\sum_{n=1}^{\infty} \gamma_n P_m^*Th_n$ is well-defined (since $P_m^*T$ has a finite rank, and every finite rank operator has a finite $\gamma$-norm, see \cite[Section 9.2]{HNVW2}), and for any $x^*\in Y_m$ we have that 
 \begin{align*}
   V_m(x^*, x^*) = V(x^*, x^*) = \| T^* x^*\|= \|T^* P_m x^*\|=\mathbb E \Bigl\langle \sum_{n=1}^{\infty} \gamma_n P_m^* Th_n, x^*\Bigr\rangle^2,
 \end{align*}
so $V_m$ is the covariance bilinear form of $\sum_{n=1}^{\infty} \gamma_n P_m^* Th_n$, and
\[
 \gamma(V_m) = \Bigl(\mathbb E \Bigl\|\sum_{n=1}^{\infty} \gamma_n  P_m^*Th_n\Bigr\|^2\Bigr)^{\frac 12} = \Bigl(\mathbb E \Bigl\|P_m^*\sum_{n=1}^{\infty} \gamma_n  Th_n\Bigr\|^2\Bigr)^{\frac 12}.
\]
The latter expression converges to $\gamma(V)$ by Lemma \ref{lem:gammaofV=gammaofTfits} and due to the fact that $\|P^*_m x\|\to \|x\|$ monotonically for each $x\in X$ as $m\to \infty$.
\end{proof}

The next lemma provides the Gaussian characteristic with the triangular inequality.

\begin{lemma}\label{lem:||V+W||gammaleq||V||gamma+||W||gamma}
 Let $X$ be a reflexive Banach space, $V, W:X^*\times X^*$ be symmetric continuous nonnegative bilinear forms. Then $\gamma(V+W) \leq \gamma(V) +\gamma(W)$.
\end{lemma}

\begin{proof}
 If $\max\{\gamma(V) , \gamma(W)\} = \infty$ then the lemma is obvious. Let $\gamma(V) , \gamma(W) < \infty$. Let $\xi_V$ and $\xi_W$ be $X$-valued centered Gaussian random variables corresponding to $V$ and $W$ respectively. Without loss of generality we can set $\xi_V$ and $\xi_W$ independent. Let $\xi_{V+W} = \xi_V + \xi_W$. Then $\xi_{V+W}$ is an  $X$-valued centered Gaussian random variable (see \cite{BogGaus}) and for any $x^*\in X^*$ due to the independence of $\xi_V$ and $\xi_W$
 \[
  \mathbb E \langle \xi_{V+W}, x^*\rangle^2 = \mathbb E \langle \xi_{V} + \xi_W, x^*\rangle^2 = \mathbb E \langle \xi_{V}, x^*\rangle^2 + \mathbb E \langle \xi_W, x^*\rangle^2 = (V+W)(x^*, x^*).
 \]
So $\xi_{V+W}$ has $V+W$ as the covariation bilinear form, and therefore
\[
 \gamma(V+W) = (\mathbb E \|\xi_{V+W}\|^2)^{\frac 12} \leq (\mathbb E \|\xi_{V}\|^2)^{\frac 12} + (\mathbb E \|\xi_{W}\|^2)^{\frac 12} = \gamma(V) + \gamma(W). 
\]
\end{proof}

Now we discuss such important properties of $\gamma(\cdot)$ as monotonicity and monotone continuity.

\begin{lemma}\label{lem:VgeqWthenVgammageqWgamma}
 Let $X$ be a separable Banach space, $V, W:X^*\times X^* \to \mathbb R$ be symmetric continuous nonnegative bilinear forms such that
 $W(x^*, x^*) \leq V(x^*, x^*)$ for all $x^* \in X^*$. Then $\gamma(W) \leq \gamma(V)$.
\end{lemma}

\begin{proof}
 The lemma follows from Lemma \ref{lem:gammaofV=gammaofTfits} and \cite[Theorem 9.4.1]{HNVW2}.
\end{proof}

\begin{lemma}\label{lem:V_nmonto0thenthesamefor0}
  Let $X$ be a separable reflexive Banach space, $Y\subset X^*$ be a dense subset, $(V_n)_{n\geq 1}$ be symmetric continuous nonnegative bilinear forms on $X^* \times X^*$ such that $V_n(x^*, x^*) \to 0$ for any $x^*\in Y$ monotonically as $n\to \infty$. Assume additionally that $\gamma(V_n) <\infty$ for some $n\geq 1$. Then $\gamma(V_n)\to 0$ monotonically as $n\to \infty$.
\end{lemma}

\begin{proof}
Without loss of generality assume that $\gamma(V_1)<\infty$. Note that by Lemma \ref{lem:VgeqWthenVgammageqWgamma} the sequence $(\gamma(V_n))_{n\geq 1}$ is monotone and bounded by $\gamma(V_1)$.
First of all notice that $V_n(x^*, x^*) \to 0$ for any $x^*\in X^*$ monotonically as $n\to \infty$. Indeed, fix $x^*\in X^*$. For any $\eps>0$ fix $x^*_{\eps} \in Y$ such that $\|x^*-x^*_{\eps}\| < \eps$. Then $(V_n(x^*_{\eps},x^*_{\eps}))_{n\geq 1}$ vanishes monotonically, and 
$$
|V_n(x^*,x^*)^{1/2} - V_n(x^*_{\eps},x^*_{\eps})^{1/2}| \leq \|x^* -x^*_{\eps}\| \gamma(V_n) \leq \eps \gamma(V_1),
$$ 
by \eqref{eq:contofVduetogammanormtikinula}. Thus $(V_n(x^*,x^*))_{n\geq 1}$ vanishes monotonically if we let $\eps \to 0$.

By Lemma \ref{lem:TfromVizveniiiilsya} we may assume that there exists a separable Hilbert space $H$ and a sequence of operators $(T_n)_{n\geq 1}$ from  $H$ to $X$ such that $V_n(x^*, x^*) = \|T_n^* x^*\|^2$ for all $x^*\in X^*$ (note that we are working with one Hilbert space since all the separable Hilbert spaces are isometrically isomorphic). Let $T\in \mathcal L(H,X)$ be the zero operator. Then $T_n^* x^* \to T^*x^* = 0$ as $n\to \infty$ for all $x^*\in X^*$, and hence by \cite[Theorem 9.4.2]{HNVW2}, Lemma \ref{lem:gammaofV=gammaofTfits}, and the fact that $\|T_nx^*\| \leq \|T_1x^*\|$ for all $x^*\in X^*$
\[
 \lim_{n\to \infty}\gamma(V_n) =\lim_{n\to \infty}\|T_n\|_{\gamma(H,X)} = \|T\|_{\gamma(H,X)} = 0.
\]
\end{proof}

The following lemma follows for Lemma \ref{lem:||V+W||gammaleq||V||gamma+||W||gamma} and \ref{lem:V_nmonto0thenthesamefor0}.

\begin{lemma}\label{lem:V_nmontoVthenthesameforgamma(V)}
 Let $X$ be a separable reflexive Banach space, $Y\subset X^*$ be a dense subset, $V$, $(V_n)_{n\geq 1}$ be symmetric continuous nonnegative bilinear forms on $X^* \times X^*$ such that $V_n(x^*, x^*) \nearrow V(x^*, x^*)$ for any $x^*\in Y$ monotonically as $n\to \infty$. Then $\gamma(V_n)\nearrow \gamma(V)$ monotonically as $n\to \infty$.
\end{lemma}

\subsection{$\gamma(\cdot)$ and $\gamma(\cdot)^2$ are not norms}\label{subsec:gammaisnotanirmtuchikakludi}
 Notice that $\gamma(\cdot)$ is not a norm. Indeed, it is easy to see that $\gamma(\alpha V) = \sqrt {\alpha} \gamma(V)$ for any $\alpha\geq 0$ and any nonnegative symmetric bilinear form $V$: if we fix any $X$-valued Gaussian random variable $\xi$ having $V$ as its covariance bilinear form, then $\sqrt{\alpha} \xi$ has ${\alpha} \gamma(V)$ as its covariance bilinear form.

It is a natural question whether $\gamma(\cdot)^2$ satisfies the triangle inequality and hence has the norm properties. It is easy to check the triangle inequality if $X$ is Hilbert: indeed, for any $V$ and $W$
 \[
   \gamma(V+W)^2 = \mathbb E \|\xi_{V+W}\|^2 = \mathbb E \|\xi_{V}\|^2 + \mathbb E \|\xi_{W}\|^2 + 2\mathbb E \langle\xi_V, \xi_W\rangle = \gamma(V)^2 +\gamma(W)^2,
 \]
 where $\xi_V$, $\xi_W$, and $\xi_{V+W}$ are as in the latter proof. 
 
 It turns out that if such a triangular inequality holds for some Banach space $X$, then this Banach space must have a {\em Gaussian type $2$} (see \cite[Subsection 7.1.d]{HNVW2}). Indeed, let $X$ be such that for all nonnegative symmetric bilinear forms $V$ and $W$ on $X^*\times X^*$
 \begin{equation}\label{eq:trianglineqforgamma2}
  \gamma(V+W)^2 \leq \gamma(V)^2 + \gamma(W)^2.
 \end{equation}
Fix $(x_i)_{i=1}^n\subset X$ and a sequence of independent standard Gaussian random variables $(\xi_i)_{i=1}^n$. For each $i=1,\ldots,n$ define a symmetric bilinear form $V_i:X^* \times X^* \to \mathbb R$ as $V_i(x^*, y^*) := \langle x_i, x^*\rangle \cdot \langle x_i, y^*\rangle$. Let $V = V_1+ \cdots + V_n$. Then by \eqref{eq:trianglineqforgamma2} and the induction argument
\begin{align*}
 \mathbb E \Bigl\|\sum_{i=1}^n \xi_i x_i\Bigr\|^2 \stackrel{(*)}= \gamma(V)^2 \leq \sum_{i=1}^n \gamma(V_i)^2 \stackrel{(**)}= \sum_{i=1}^n \mathbb E \|\xi_i x_i\|^2 = \sum_{i=1}^n  \|x_i\|^2,
\end{align*}
where $(*)$ follows from the fact that $\sum_{i=1}^n \xi_i x_i$ is a centered Gaussian random variable the fact that for all $x^*, y^*\in X^*$
\[
 \mathbb E \Bigl\langle\sum_{i=1}^n \xi_i x_i, x^*\Bigr\rangle \cdot \Bigl\langle\sum_{i=1}^n \xi_i x_i, y^*\Bigr\rangle = \sum_{i=1}^n \langle x_i, x^* \rangle\cdot \langle x_i, y^*\rangle = V(x^*, y^*),
\]
while $(**)$ follows analogously by exploiting the fact that $\xi_i x_i$ is a centered Gaussian random variable with the covariance bilinear form $V_i$. Therefore by \cite[Definition 7.1.17]{HNVW2}, $X$ has a Gaussian type 2 with the corresponding Gaussian type constant $\tau_{2,X}^{\gamma}=1$. In the following proposition we show that this condition yields that $X$ is Hilbert, and thus we conclude that $\gamma(\cdot)^2$ defines a norm if and only if $X$ is a Hilbert space.

\begin{proposition}\label{prop:Gaustype2=1=Hilbert}
 Let $X$ be a Banach space such that its Gaussian type $2$ constant equals $1$. Then $X$ is Hilbert.
\end{proposition}

\begin{proof}
Due to the parallelogram identity it is sufficient to show that every two dimensional space of $X$ is Hilbert; consequently, without loss of generality we can assume that $X$ is two dimensional. We need to show that the unit ball of $X$ is an ellipse as any ellipse corresponds to an inner product (see e.g.\ \cite{Day47}). Let $B\in X \simeq \mathbb R^2$ be the unit ball of $X$. Then by \cite[Theorem 1]{San96} there exists an ellipse $E \in X$ containing $B$ such that $\partial B$ and $\partial E$ intersect in at least two pairs of points. Let us denote these pairs by $(x_1, -x_1)$ and $(x_2, -x_2)$. Notice that both $x_1$ and $x_2$ are nonzero and are not collinear. Let $\vertiii{\cdot}$ be the norm associated to $E$. Then
\begin{equation}\label{||||||dominatedby||||}
 \vertiii{x} \leq \|x\|,\;\;\; x\in X
\end{equation}
as $B\subset E$, and $\vertiii{x_1} = \|x_1\|=\vertiii{x_2} = \|x_2\| =1$ (as both points are in $\partial B \cap \partial E$). Note that $X$ endowed with $\vertiii{\cdot}$ is a Hilbert space by \cite{Day47}, thus it has an inner product $\langle \cdot, \cdot \rangle_E$. Let $\gamma_1$ and $\gamma_2$ be independent standard Gaussian random variables. Then we have that
\begin{equation}\label{eq:forGaustype2constimpliesHilbert}
\begin{split}
 2 &= \vertiii{x_1}^2 +  \vertiii{x_2}^2 =\mathbb E \gamma_1^2 \vertiii{x_1}^2 +  \mathbb E \gamma_2^2 \vertiii{x_2}^2  + \mathbb E  2 \gamma_1\gamma_2 \langle x_1, x_2 \rangle_E\\
 &= \mathbb E \vertiii{\gamma_1 x_1 + \gamma_2x_2}^2 \stackrel{(*)}\leq \mathbb E \|\gamma_1 x_1 + \gamma_2x_2\|^2 \stackrel{(**)}\leq \|x_1\|^2 + \|x_2\|^2 = 2,
\end{split}
\end{equation}
where $(*)$ holds by \eqref{||||||dominatedby||||}, and $(**)$ holds since $\tau_{2, X}^1 = 1$ (see \cite[Definition 7.1.17]{HNVW2}). Therefore we have that every inequality in the estimate above is actually an equality, and hence $\mathbb E \vertiii{\gamma_1 x_1 + \gamma_2x_2}^2 = \mathbb E \|\gamma_1 x_1 + \gamma_2x_2\|^2$. Thus by \eqref{||||||dominatedby||||} $\vertiii{\gamma_1 x_1 + \gamma_2x_2} = \|\gamma_1 x_1 + \gamma_2x_2\|$ a.s., and as $x_1$ and $x_2$ are not collinear and $X$ is two dimensional, $\gamma_1 x_1 + \gamma_2x_2$ has a nonzero distribution density on the whole $X$, so we have that $\vertiii{x} = \|x\|$ for a.e.\ $x\in X$ (and by continuity for any $x\in X$), and the desired follows. 
\end{proof}

\begin{remark}
Assume that $X$ has a Gaussian {\em cotype} $2$ constant equals $1$. Then the same proof will yield that $X$ is Hilbert,  but now one needs to find an ellipse $E$ {\em inside} $B$ such that $\partial B$ and $\partial E$ intersect in at least two pairs of points. In order to find such an ellipse it is sufficient to find an ellipse $E'\in X^*$ containing the unit ball $B' \subset X^*$ such that $\partial B'$ and $\partial E'$ intersect in at least two pairs of points, and then set $B$ to be the unit ball of a space $Y^*$, where $Y$ is a Hilbert space having $E'$ as its unit ball. Then \eqref{eq:forGaustype2constimpliesHilbert} will hold true but with $\geq$ instead of $\leq$.
\end{remark}

\subsection{Finite dimensional case} Even though a Gaussian characteristic is well-defined only for some nonnegative symmetric forms, it can be extended in a proper continuous way to all the symmetric forms given $X$ is finite dimensional. Let $X$ be a finite dimensional Banach space. Notice that in this case $\gamma(V)<\infty$ for any nonnegative symmetric bilinear form $V$ (see Remark  \ref{rem:fdinmpliesgamma(V)<infty}). Let us define $\gamma(V)$ for a general symmetric $V\in X^{**}\otimes X^{**} = X\otimes X$ in the following way:
 \begin{equation}\label{eq:gammaofVgenerelVdef}
     \gamma(V) := \inf\{\gamma(V^+) + \gamma(V^-): V^+ ,V^- \text{ are nonnegative and } V\!=\!V^+\!-\!V^-\}.
 \end{equation}
 Notice that $\gamma(V)$ is well-defined and finite for any symmetric $V$. Indeed, by a well known linear algebra fact (see e.g.\ \cite[Theorem 6.6 and 6.10]{SR13}) any symmetric bilinear form $V$ has an eigenbasis $(x_n^*)_{n=1}^d$ of $X^*$ that {\em diagonalizes} $V$, i.e.\ there exists $(\lambda_n)_{n= 1}^d \in \mathbb R$ such that for all $(a_n)_{n= 1}^d, (b_n)_{n= 1}^d \in \mathbb R$ we have that for $x^* = \sum_{n=1}^d a_n x_n^*$ and $y^* = \sum_{n=1}^d b_n x_n^*$
 \[
  V(x^*, y^*) = \sum_{n=1}^d \sum_{m=1}^d a_n b_m V(x_n^*, x_m^*) = \sum_{n=1}^d \lambda_n a_n b_n.
 \]
 Therefore it is sufficient to define
 \[
  V^+(x^*, y^*) := \sum_{n=1}^d \mathbf 1_{\lambda_n\geq 0} \lambda_n  a_n b_n,\;\;\;
  V^-(x^*, y^*) := \sum_{n=1}^d \mathbf 1_{\lambda_n< 0} (-\lambda_n)  a_n b_n
 \]
and then $\gamma (V) \leq \gamma(V^+) + \gamma(V^-)<\infty$ due to the fact that $V^+$ and $V^-$ are nonnegative and by Remark \ref{rem:fdinmpliesgamma(V)<infty}. (In fact, one can check that $\gamma(V) = \gamma(V^+) + \gamma(V^-)$, but we will not need this later, so we leave this fact without a proof).

 Now we will develop some basic and {\em elementary} (but nonetheless  important) properties of such a general $\gamma(\cdot)$.
 
\begin{lemma}
 Let $V:X^* \times X^* \to \mathbb R$ be a {\em nonnegative} symmetric bilinear form. Then $\gamma(V)$ defined by \eqref{eq:gammaofVgenerelVdef} coincides with $\gamma(V)$ defined in Subsection \ref{subsec:gausscharactbasicdefinitions}. In other words, these definitions {\em agree} given $V$ is nonnegative.
\end{lemma}

\begin{proof}
 Fix nonnegative $V^+$ and $V^-$ such that $V = V^+ - V^-$. Then $\gamma(V^+) + \gamma(V^-) = \gamma(V + V^-) + \gamma(V^-) \geq \gamma(V) + \gamma(V^-) \geq \gamma(V)$ by Lemma \ref{lem:VgeqWthenVgammageqWgamma}, so $\gamma(V)$ does not change.
\end{proof}

\begin{lemma}\label{lem:triangleineqforgammageneral}
 Let $V, W:X^*\times X^* \to \mathbb R$ be symmetric bilinear forms. Then $\gamma(V) - \gamma(W) \leq \gamma(V-W)$.
\end{lemma}

\begin{proof}
 Denote $V-W$ by $U$. Fix $\eps>0$. Then there exist symmetric nonnegative bilinear forms $W^+,W^-,U^+,U^-$ such that $W= W^+-W^-$, $U=U^+-U^-$, and
\[
 \gamma(W) \geq \gamma(W^+) + \gamma(W^-) - \eps,
\]
\[
 \gamma(U) \geq \gamma(U^+) + \gamma(U^-) - \eps.
\]
Then since $V = U + W$ by \eqref{eq:gammaofVgenerelVdef} and Lemma \ref{lem:||V+W||gammaleq||V||gamma+||W||gamma}
\begin{align*}
 \gamma(V) - \gamma(W) &= \gamma((W^++U^+)-(W^-+U^-)) - \gamma(W^+-W^-)\\
 &\leq \gamma(W^++U^+)+ \gamma(W^-+U^-) - \gamma(W^+) - \gamma(W^-) + \eps\\
 &\leq \gamma(U^+)+ \gamma(U^-)  + \eps \leq \gamma(U) + 2\eps,
\end{align*}
and by sending $\eps\to 0$ we conclude the desired.
\end{proof}

\begin{lemma}\label{lem:gammaalphaV=alpha^12gammaV}
 Let $V:X^*\times X^* \to \mathbb R$ be a symmetric bilinear form. Then $\gamma(V) = \gamma(-V)$ and $\gamma(\alpha V) = \sqrt \alpha \gamma(V)$ for any $\alpha\geq 0$.
\end{lemma}

\begin{proof}
 The first part follows directly from \eqref{eq:gammaofVgenerelVdef}. For the second part we have that due to \eqref{eq:gammaofVgenerelVdef} it is enough to justify $\gamma(\alpha V) = \sqrt \alpha \gamma(V)$ only for nonnegative $V$, which was done in Subsection \ref{subsec:gammaisnotanirmtuchikakludi}.
\end{proof}

\begin{proposition}\label{prop:continuity+squreboundednessofgammaidukuryu}
 The function $\gamma(\cdot)$ defined by \eqref{eq:gammaofVgenerelVdef} is continuous on the linear space of all symmetric bilinear forms endowed with $\|\cdot\|$ defined by \eqref{eq:normofbilinearfromdefvoina}. Moreover, $\gamma(V)^2 \lesssim_X \|V\|$ for any symmetric bilinear form $V:X^* \times X^* \to \mathbb R$.
\end{proposition}

\begin{proof}
 Due to Lemma \ref{lem:triangleineqforgammageneral} and \ref{lem:gammaalphaV=alpha^12gammaV} it is sufficient to show that $\gamma(\cdot)$ is bounded on the unit ball with respect to the norm $\|\cdot\|$ in order to prove the first part of the proposition. Let us show this boundedness. Let $U$ be a fixed symmetric nonnegative element of $X\otimes X$ such that $U+V$ is nonnegative and such that $U(x^*, x^*)\geq V(x^*, x^*)$ for any symmetric $V$ with $\|V\|\leq 1$ (since $X$ is finite dimensional, one can take $U(x^*):= c\vertiii{x^*}^2$ for some Euclidean norm $\vertiii{\cdot}$ on $X^*$ and some big enough constant $c>0$). Fix a symmetric $V:X^* \times X^* \to \mathbb R$ with $\|V\|\leq 1$. Then $V= (U+V) - U$, and by \eqref{eq:gammaofVgenerelVdef}
\[
 \gamma(V) \leq \gamma(U+V) + \gamma(U) = \gamma(2U) + \gamma(U),
\]
which does not depend on $V$.

Let us show the second part. Due to the latter consideration there exists a constant $C_X$ depending only on $X$ such that $\gamma(V)\leq C_X $ if $\|V\|\leq 1$. Therefore by Lemma \ref{lem:gammaalphaV=alpha^12gammaV} we have that for a general symmetric $V$
\[
 \gamma(V)^2 = \|V\| \gamma(V/\|V\|)^2 \leq C_X^2 \|V\|.
\]
\end{proof}

Later we will also need the following elementary lemma.

\begin{lemma}\label{lem:boundedbessofVbyagoodnorms}
There exists vectors $(x_i^*)_{i=1}^n$ in $X^*$ such that
  \begin{equation}\label{eq:defofvertiiiforbilineaforms}
  \vertiii{V}:= \sum_{i=1}^n |V(x^*_i, x^*_i)|
 \end{equation}
 defines a norm on the space of all symmetric bilinear forms on $X^*\times X^*$. In particular we have that  $\|V\| \eqsim_X \vertiii{V}$ for any symmetric bilinear form $V:X^* \times X^*\to \mathbb R$.
\end{lemma}

We will demonstrate here the proof for the convenience of the reader.

\begin{proof}
First notice that $\vertiii{\cdot}$ clearly satisfies the triangular inequality. Let us show that there exists a set $(x_i^*)_{i=1}^n$ such that $\vertiii{V}=0$ implies $V=0$. Let $(y_i^*)_{i=1}^d$ be a basis of $X^*$. Then there exist $i, j \in \{1,\ldots,d\}$ such that 
 $$
 0\neq V(y_i^*, y_j^*)  = (V(y_i^* + y_j^*, y_i^* + y_j^*) - V(y_i^* - y_j^*, y_i^* - y_j^*))/4
 $$ 
 (otherwise $V = 0$). This means that for these $i$ and $j$
 $$
 |V(y_i^* + y_j^*, y_i^* + y_j^*)| + |V(y_i^* - y_j^*, y_i^* - y_j^*)| \neq 0,
 $$
 so in particular
 \[
  \sum_{i=1}^d\sum_{j=1}^d |V(y_i^* + y_j^*, y_i^* + y_j^*)| + |V(y_i^* - y_j^*, y_i^* - y_j^*)| \neq 0.
 \]
 It remains to notice that the latter sum has the form \eqref{eq:defofvertiiiforbilineaforms} for a proper choice of $(x_i^*)_{i=1}^n$ independent of $V$.
 
  In order to show the last part of the lemma we need to notice that the space of symmetric bilinear forms is finite dimensional if $X$ is so, so all the norms on the linear space of  symmetric bilinear forms are equivalent, and therefore $\|V\| \eqsim_X \vertiii{V}$ for any symmetric bilinear form $V:X^* \times X^*\to \mathbb R$. 
\end{proof}

\section{Preliminaries}\label{sec:prelim}

We continue with some preliminaries concerning continuous-time martingales.

\subsection{Banach space-valued martingales}\label{subsec:Bsvm}

Let $(\Omega,\mathcal F, \mathbb P)$ be a probability space with a filtration $\mathbb F = (\mathcal F_t)_{t\geq 0}$ which satisfies the usual conditions. Then $\mathbb F$ is right-continuous (see \cite{Kal,JS} for details).

Let $X$ be a Banach space. An adapted process $M:\mathbb R_+ \times \Omega \to X$ is called a {\em martingale} if $M_t\in L^1(\Omega; X)$ and $\mathbb E (M_t|\mathcal F_s) = M_s$ for all $0\leq s\leq t$ (we refer the reader to \cite{HNVW1} for the details on vector-valued integration and vector-valued conditional expectation). It is well known that in the real-valued case any martingale is {\em c\`adl\`ag} (i.e.\ has a version which is right-continuous and that has limits from the left-hand side). The same holds for a general $X$-valued martingale $M$ as well (see \cite{Y17FourUMD,VerPhD}), so one can define $\Delta M_{\tau} := M_{\tau} - \lim_{\eps \searrow 0} M_{0 \vee (\tau - \eps)}$ on $\{\tau<\infty\}$ for any stopping time $\tau$.

Let $1\leq p\leq \infty$. A martingale $M:\mathbb R_+ \times \Omega \to X$ is called an {\em $L^p$-bounded martingale} if $M_t \in L^p(\Omega; X)$ for each $t\geq 0$ and there exists a limit $M_{\infty} := \lim_{t\to \infty} M_t\in L^p(\Omega; X)$ in $L^p(\Omega; X)$-sense. Since $\|\cdot\|:X \to \mathbb R_+$ is a convex function, and $M$ is a martingale, $\|M\|$ is a submartingale by Jensen's inequality, and hence by Doob's inequality (see e.g.\ \cite[Theorem 1.3.8(i)]{KS}) we have that for all $1<p\leq \infty$
\begin{equation}\label{eq:DoobsineqXBanach}
 \mathbb E \sup_{0\leq s\leq t} \|M_s\|^p \eqsim_p \mathbb E \|M_t\|^p,\;\;\; t\geq 0.
\end{equation}

\subsection{Quadratic variation}\label{subsec:qv+q_M}

Let $H$ be a Hilbert space, $M:\mathbb R_+ \times \Omega \to H$ be a local martingale. We define a {\em quadratic variation} of $M$ in the following way:
\begin{equation}\label{eq:defquadvar}
 [M]_t  := \mathbb P-\lim_{{\rm mesh}\to 0}\sum_{n=1}^N \|M(t_n)-M(t_{n-1})\|^2,
\end{equation}
where the limit in probability is taken over partitions $0= t_0 < \ldots < t_N = t$. Note that $[M]$ exists and is nondecreasing a.s.
The reader can find more on quadratic variations in \cite{MetSemi,MP,VY16} for the vector-valued setting, and in \cite{Kal,Prot,MP} for the real-valued setting.

As it was shown in \cite[Proposition 1]{Mey77} (see also \cite[Theorem 2.13]{Roz90} and \cite[Example 3.19]{VY16} for the continuous case) that for any $H$-valued martingale $M$ there exists an adapted process $q_M:\mathbb R_+ \times \Omega \to \mathcal L(H)$ which we will call a {\em quadratic variation derivative}, such that the trace of $q_M$ does not exceed $1$ on $\mathbb R_+ \times \Omega$, $q_M$ is self-adjoint nonnegative on $\mathbb R_+ \times \Omega$, and for any $h,g\in H$ a.s.\
\begin{equation}\label{eq:whyq_Misinportant5element}
  [\langle M, h\rangle,\langle M, g\rangle]_t =\int_0^t \langle q_M^{1/2}(s)h, q_M^{1/2}(s) g \rangle \ud [M]_s,\;\;\; t\geq 0.
\end{equation}

\smallskip

For any martingales $M, N:\mathbb R_+ \times \Omega \to H$ we can define a {\em covariation} $[M,N]:\mathbb R_+ \times \Omega \to \mathbb R$ as $[M,N] := \frac{1}{4}([M+N]-[M-N])$.
Since $M$ and $N$ have c\`adl\`ag versions, $[M,N]$ has a c\`adl\`ag version as well (see \cite[Theorem I.4.47]{JS} and \cite{MetSemi}).

Let $X$ be a Banach space, $M:\mathbb R_+ \times \Omega \to X$ be a local martingale. Fix $t\geq 0$. Then $M$ is said to have a {\em covariation bilinear from} $[\![M]\!]_t$ at $t\geq 0$ if there exists a continuous bilinear form-valued random variable $[\![M]\!]_t:X^* \times X^* \times \Omega \to \mathbb R$ such that for any fixed $x^*, y^*\in X^*$ a.s.\ $[\![M]\!]_t(x^*, y^*) = [\langle M, x^*\rangle, \langle M, y^*\rangle]_t$. 

\begin{remark}\label{rem:basicpropsof[[M]]}
 Let us outline some basic properties of the covariation bilinear forms, which follow directly from \cite[Theorem 26.6]{Kal} (here we presume the existence of $[\![M]\!]_t$ and $[\![N]\!]_t$ for all $t\geq 0$)
\begin{enumerate}[(i)]
 \item $t\mapsto [\![M]\!]_t$ is nondecreasing, i.e.\ $[\![M]\!]_t(x^*, x^*) \geq [\![M]\!]_s(x^*, x^*)$ a.s.\ for all $0\leq s\leq t$ and $x^*\in X^*$,
 \item $ [\![M]\!]^{\tau} =  [\![M^{\tau}]\!]$ a.s.\ for any stopping time $\tau$,
  \item $ \Delta [\![M]\!]_{\tau}(x^*, x^*) =|  \langle \Delta M_{\tau}, x^*\rangle|^2$ a.s.\ for any stopping time $\tau$.
\end{enumerate}
\end{remark}

\begin{remark}\label{rem:existenceof[[M]]givenXisfd}
 If $X$ is finite dimensional, then it is isomorphic to a Hilbert space, and hence existence of $[\![M]\!]_t$ follows from existence of $[M]_t$ with the following estimate a.s.\
 $$
 \|[\![M]\!]_t\| = \sup_{x^*\in X^*, \|x^*\|\leq 1}[\![M]\!]_t(x^*, x^*) = \sup_{x^*\in X^*, \|x^*\|\leq 1}[\langle M, x^*\rangle, \langle M, x^*\rangle]_t \lesssim_X [M]_t.
 $$ 
 For a general infinite dimensional Banach space the existence of $[\![M]\!]_t$ remains an open problem. In Theorem \ref{thm:BDGgeneralUMD} we show that if $X$ has the UMD property, then existence of $[\![M]\!]_t$ follows automatically; moreover, in this case $ \gamma([\![M]\!]_t) <\infty$ a.s.\ (see Section \ref{sec:Gaussiancharacteristic} and Theorem~\ref{thm:BDGgeneralUMD}), which is way stronger than continuity.
\end{remark}

\section{Burkholder-Davis-Gundy inequalities: the continuous-time case}\label{sec:BDG-cont}

The following theorem is the main theorem of the paper.

\begin{theorem}\label{thm:BDGgeneralUMD}
 Let $X$ be a UMD Banach space. Then for any local martingale $M:\mathbb R_+\times \Omega \to X$ with $M_0=0$ and any $t\geq 0$ the covariation bilinear form $[\![M]\!]_t$ is well-defined and bounded almost surely, and for all $1\leq p<\infty$
 \begin{equation}\label{eq:thmBDGgeneralUMD}
   \mathbb E \sup_{0\leq s\leq t}\|M_s\|^p \eqsim_{p,X} \mathbb E  \gamma([\![M]\!]_t)^p .
 \end{equation}
\end{theorem}

\begin{proof}[Proof of Theorem \ref{thm:BDGgeneralUMD}]

{\em Step 1: finite dimensional case.}
First note that in this case $[\![M]\!]_t$ exists and bounded a.s.\ due to Remark \ref{rem:existenceof[[M]]givenXisfd}. Fix $1\leq p<\infty$. By mutlidimensional Burkholder-Davis-Gundy inequalities we may assume that both $\mathbb E \sup_{0\leq s\leq t}\|M_s\|^p$ and $\mathbb E  \gamma([\![M]\!]_t)^p$ are finite.
For each $N\geq 1$ fix a partition $0= t_1^N<\ldots < t_{n_N}^N = t$ with the mesh not exceeding $1/N$. For each $\omega\in \Omega$ and $N\geq 1$ define a bilinear form $V_N:X^*\times X^* \to \mathbb R$ as follows:
\begin{equation}\label{eq:dfofV_NforageneralBDGgenBanachspace}
 V_N(x^*, x^*) := \sum_{i=1}^{n_N} \langle M_{t_{i}^N} - M_{t_{i-1}^N}, x^*\rangle^2,\;\;\; x^*\in X^*.
\end{equation}
Note that $(M_{t_{i}^N} - M_{t_{i-1}^N})_{i=1}^{n_N}$ is a martingale difference sequence with respect to the filtration $(\mathcal F_{t_{i}^N})_{i=1}^{n_N}$, so by Theorem \ref{thm:BDGgendiscmart} 
\begin{equation}\label{eq:equivwithV_Nzaosnovu}
 \begin{split}
   \mathbb E \bigl\|\sup_{i=1}^{n_N}M_{t_{i}^N}\bigr\|^p
      &\eqsim_{p, X} \mathbb E \Bigl(\mathbb E_{\gamma}\Bigl\|  \sum_{i=1}^{n_N} \gamma_i( M_{t_{i}^N} - M_{t_{i-1}^N})\Bigr\|^2\Bigr)^{\frac p2}= \mathbb E \gamma(V_N)^p,
 \end{split}
\end{equation}
where $(\gamma_i)_{i=1}^{n_N}$ is a sequence of independent Gaussian standard random variables, and the latter equality holds due to the fact that for any fixed $\omega\in \Omega$ the random variable $ \sum_{i=1}^{n_N} \gamma_i( M_{t_{i}^N} - M_{t_{i-1}^N})(\omega)$ is Gaussian and by \eqref{eq:dfofV_NforageneralBDGgenBanachspace}
\[
 V_N(x^*, x^*) = \mathbb E_{\gamma} \Bigl\langle \sum_{i=1}^{n_N} \gamma_i( M_{t_{i}^N} - M_{t_{i-1}^N})(\omega), x^* \Bigr\rangle^2,\;\;\; x^*\in X^*.
\]
Therefore it is sufficient to show that $\gamma(V_N - [\![M]\!]_t)\to 0$ in $L^p(\Omega)$ as $N\to \infty$. Indeed, if this is the case, then by \eqref{eq:equivwithV_Nzaosnovu} and by Lemma \ref{lem:triangleineqforgammageneral}
\[
 \mathbb E \gamma([\![M]\!]_t)^p = \lim_{N\to\infty} \mathbb E \gamma(V_N)^p \eqsim_{p,X}\lim_{N\to\infty}  \mathbb E \bigl\|\sup_{i=1}^{n_N}M_{t_{i}^N}\bigr\|^p =  \mathbb E\sup_{0\leq s\leq t} \|M_s\|^p,
\]
where the latter holds by the dominated convergence theorem as any martingale has a c\`adl\`ag version (see Subsection \ref{subsec:Bsvm}).
Let us show this convergence.
Note that by Proposition \ref{prop:continuity+squreboundednessofgammaidukuryu} and Lemma \ref{lem:boundedbessofVbyagoodnorms} a.s.\
$$
\gamma(V_N - [\![M]\!]_t)^2 \lesssim_X \|{V_N - [\![M]\!]_t}\| \lesssim_X \vertiii{V_N - [\![M]\!]_t}
$$ 
(where $\vertiii{\cdot}$ is as in \eqref{eq:defofvertiiiforbilineaforms})
Therefore we need to show that $\vertiii{V_N - [\![M]\!]_t} \to 0$ in $L^{\frac p2}(\Omega)$, which follows from the fact that for any $x_i^*$ from Lemma \ref{lem:boundedbessofVbyagoodnorms}, $i=1,\ldots,n$, we have that
\[
 V_N(x_i^*, x_i^*) = \sum_{i=1}^{n_N} \langle M_{t_{i}^N} - M_{t_{i-1}^N}, x^*_i\rangle^2 \to [\langle M, x_i^*\rangle]_t
\]
in $L^{\frac p2}$-sense by \cite[Th\'eor\`eme 2]{Dol69} and \cite[Theorem 5.1]{BDG}.

\smallskip

{\em Step 2: infinite dimensional case.} First assume that $M$ is an $L^p$-bounded martingale. Without loss of generality we can assume $X$ to be separable. Since $X$ is UMD, $X$ is reflexive, so $X^*$ is separable as well. Let $Y_1 \subset Y_2 \subset \ldots\subset Y_n \subset \ldots$ be a family of finite dimensional subspaces of $X^*$ such that $\overline{\cup_n Y_n} = X^*$. For each $n\geq 1$ let $P_n:Y_n \to X^*$ be the inclusion operator. Then $\|P_n^*\|\leq 1$ and $P_n^* M$ is a well-defined $Y_n^*$-valued $L^p$-bounded martingale. By Step 1 this martingale a.s.\ has a covariation bilinear form $[\![P_n^* M]\!]_t$ acting on $Y_n \times Y_n$ and
\begin{equation}\label{eq:unifboundforP_n*nespeshit}
  \mathbb E \gamma([\![P_n^* M]\!]_t)^p \stackrel{(*)}\eqsim_{p, X} \mathbb E \sup_{0\leq s\leq t} \|P_n^* M_s\|^p \leq \mathbb E\sup_{0\leq s\leq t}  \|M_s\|^p,
\end{equation}
where $(*)$ is independent of $n$ due to \cite[Proposition 4.2.17]{HNVW1} and Remark \ref{rem:deponconstants}. Note that a.s.\ $[\![P_n^* M]\!]_t$ and $[\![P_m^* M]\!]_t$ agree for all $m\geq n\geq 1$, i.e.\ a.s.\
\begin{equation}\label{eq:P_nP_magreementpomelocham}
  [\![P_m^* M]\!]_t(x^*,y^*) = [\![P_n^* M]\!]_t(x^*,y^*) = [\langle M, x^*\rangle,\langle M,y^*\rangle]_t,\;\;\; x^*, y^*\in Y_n.
\end{equation}
Let $\Omega_0\subset \Omega$ be a subset of measure 1 such that \eqref{eq:P_nP_magreementpomelocham} holds for all $m\geq n\geq 1$. Fix $\omega \in \Omega_0$. 
Then by \eqref{eq:P_nP_magreementpomelocham} we can define a bilinear form (not necessarily continuous!) $V$ on $Y\times Y$ (where $Y:= \cup_{n}Y_n \subset X^*$) such that $V(x^*,y^*)=[\![P_n^* M]\!]_t(x^*,y^*)$ for all $x^*, y^*\in Y_n$ and $n\geq 1$. 

Let us show that $V$ is continuous (and hence has a continuous extension to $X^* \times X^*$) and $\gamma(V) <\infty$ a.s.\ on $\Omega_0$.
Notice that by Lemma \ref{lem:V_0isonsubsetofVsoitsgammaislettznakomo} the sequence $(\gamma([\![P_n^* M]\!]_t))_{n\geq 1}$ is increasing a.s.\ on $\Omega_0$. Moreover, by the monotone convergence theorem and \eqref{eq:unifboundforP_n*nespeshit} $(\gamma([\![P_n^* M]\!]_t))_{n\geq 1}$ has a limit a.s.\ on $\Omega_0$. Let ${\Omega_1}\subset \Omega_0$ be a subset of full measure such that $(\gamma([\![P_n^* M]\!]_t))_{n\geq 1}$ has a limit on ${\Omega_1}$. Then by \eqref{eq:contofVduetogammanormtikinula} $V$ is continuous on $\Omega_1$ and hence has a continuous extension to $X^* \times X^*$ (which we will denote by $V$ as well for simplicity). Then by Proposition \ref{prop:convergenceofrestrictedV_m} $\gamma(V) = \lim_{n\to \infty}\gamma([\![P_n^* M]\!]_t)$ monotonically on $\Omega_1$ and hence by monotone convergence theorem ans the fact that $\|P_n^* x\| \to \|x\|$ as $n\to \infty$ monotonically for all $x\in X$
\[
 \mathbb E\sup_{0\leq s\leq t}  \|M_s\|^p= \lim_{n\to \infty}  \mathbb E \sup_{0\leq s\leq t} \|P_n^* M_s\|^p\eqsim_{p,X} \lim_{n\to \infty}\mathbb E \gamma([\![P_n^* M]\!]_t)^p  = \mathbb E (\gamma(V))^p. 
\]
It remains to show that $V = [\![M]\!]_t$ a.s., i.e.\ $V(x^*, x^*) = [\langle M, x^*\rangle]_t$ a.s.\ for any $x^*\in X^*$. If $x^* \in Y$, then the desired follows from the construction of $V$. Fix $x^*\in X^*\setminus Y$. Since $Y$ is dense in $X^*$, there exists a Cauchy sequence $(x_n^*)_{n\geq 1}$ in $Y$ converging to $x^*$. Then since $V(x_n^*,x_n^*) = [\langle M, x^*_n\rangle]_t$ a.s.\ for all $n\geq 1$,
\begin{align*}
  \lim_{n\to \infty} |V(x_n^*,x_n^*) - [\langle M, x^*\rangle]_t|^{\frac p2}&\lesssim_p   \lim_{n\to \infty}  [\langle M, x^*-x_n^*\rangle]_t^{\frac p2}\eqsim_p \lim_{n\to \infty} \mathbb E |\langle M, x^*-x_n^*\rangle|^p\\
  &\leq  \lim_{n\to \infty} \mathbb E \| M\|^p\|x^*-x_n^*\|^p = 0,
\end{align*}
so due to a.s.\ continuity of $V$, $V(x^*, x^*)$ and $[\langle M, x^*\rangle]_t$ coincide a.s.

\smallskip

Now let $M$ be a general local martingale. By a stopping time argument we can assume that $M$ is an $L^1$-bounded martingale, and then the existence of $[\![M]\!]_t$ follows from the case $p=1$.

Let us now show \eqref{eq:thmBDGgeneralUMD}. If the left-hand side is finite then $M$ is an $L^p$-bounded martingale and the desired follows from the previous part of the proof. Let the left-hand side be infinite. Then it is sufficient to notice that by Step 1
 \[
  \mathbb E \sup_{0\leq s\leq t}\|P_n^*M_s\|^p \eqsim_{p,X} \mathbb E  \gamma([\![P_n^*M]\!]_t)^p,
 \]
 for any (finite or infinite) left-hand side, and the desired will follow as $n\to \infty$ by the fact that $\|P_n^*M_s\| \to \|M_s\|$ and $\gamma([\![P_n^*M]\!]_t) \to \gamma([\![M]\!]_t)$ monotonically a.s.\ as $n\to\infty$, and the monotone convergence theorem.
\end{proof}

\begin{remark}\label{rem:nesofUMDforgenBDGwithgamma}
Note that $X$ being a UMD Banach space is necessary in Theorem \ref{thm:BDGgeneralUMD} (see Theorem \ref{thm:necofUMDfordiscBDGtango} and \cite{NVW}).
\end{remark}

\begin{remark}
 Because of Lemma \ref{lem:gammaofV=gammaofTfits} the reader may suggest that if $X$ is a UMD Banach space, then for any $X$-valued local martingale $M$, for any $t\geq 0$, and for a.a.\ $\omega \in \Omega$ there exist a natural choice of a Hilbert space $H(\omega)$ and a natural choice of an operator $T(\omega)\in \mathcal L(H(\omega), X)$ such that for all $x^*, y^*\in X^*$ a.s.\
 $$
 [\![M]\!]_t(x^*, y^*) = \langle T^* x^*,T^*y^*\rangle.
 $$
 If this is the case, then by  Lemma \ref{lem:gammaofV=gammaofTfits} and Theorem \ref{thm:BDGgeneralUMD}
 \[
  \mathbb E \sup_{0\leq s\leq t}\|M_s\|^p \eqsim_{p,X} \mathbb E  \|T\|_{\gamma(H, X)}^p.
 \]
Such a natural pair of $H(\omega)$ and $T(\omega)$, $\omega \in \Omega$, is known for purely discontinuous local martingales (see Theorem \ref{thm:BDGUMDpdmartingales}) and for stochastic integrals (see Subsection \ref{subsec:applicItoisomgenmart} and \ref{subsec:ItoisomPrmandgrm}). Unfortunately, it remains open how such $H$ and $T$ should look like for a general local martingale $M$.
\end{remark}

\begin{remark}
 As in Remark \ref{rem:phidisccase}, by a limiting argument shown in the proof of Theorem \ref{thm:BDGgeneralUMD} one can prove that for any UMD Banach space $X$, for any martingale $M:\mathbb R_+ \times \Omega \to X$, and for any convex increasing function $\phi:\mathbb R_+ \to \mathbb R_+$ with $\phi(0)=0$ and with $\phi(2\lambda) \leq c\phi(\lambda)$ for some fixed $c>0$ for any $\lambda>0$ one has that
 \[
  \mathbb E \sup_{0\leq s\leq t}\phi(\|M_s\|) \eqsim_{\phi, X} \mathbb E \phi\bigl( \gamma([\![M]\!]_t)\bigr).
 \]
 To this end, one first needs to prove the finite-dimensional case by using the proof of \cite[Theorem 5.1]{BDG} and the fact that for any convex increasing $\psi:\mathbb R_+ \to \mathbb R_+$ with $\psi(0)=0$ and with $\psi(2\lambda) \leq c\phi(\lambda)$ one has that $\psi\circ \phi$ satisfies the same properties (perhaps with a different constant $c$), and then apply the same extending argument.
 \end{remark}

Let $X$ be a UMD Banach space, $M:\mathbb R_+ \times \Omega \to X$ be a martingale. Then by Theorem \ref{thm:BDGgeneralUMD} there exists a {\em process} $[\![M]\!]:\mathbb R_+\times \Omega \to X \otimes X$ such that for any $x^*, y^*\in X^*$ and a.e.\ $(t,\omega) \in \mathbb R_+ \times \Omega$
\begin{equation}\label{[[M]]asprocess}
  [\![M]\!]_t(x^*, y^*)(\omega) = [\langle M, x^*\rangle,\langle M, y^*\rangle]_t(\omega).
\end{equation}
In our final proposition we show that this process is adapted and has a {\em c\`adl\`ag version} (i.e.\ a version which is right-continuous with left limits).

\begin{proposition}\label{prop:cadlag[[M]]}
 Let $X$ be a UMD Banach space, $M:\mathbb R_+ \times \Omega \to X$ be a local martingale. Then there exists an increasing adapted c\`adl\`ag process $[\![M]\!]:\mathbb R_+\times \Omega \to X \otimes X$ such that \eqref{[[M]]asprocess} holds true. Moreover, if this is the case then $\gamma([\![M]\!])$ is increasing adapted c\`adl\`ag.
\end{proposition}

\begin{proof}
 Existence of such a process follows from the considerations above. Let us show that this process has an increasing, adapted, and c\`adl\`ag version. First of all by a stopping time argument assume that $M$ is a martingale (so $\mathbb E \gamma([\![M]\!]_{\infty}) < \infty$ and hence $\gamma([\![M]\!]_{\infty}) < \infty$ a.s.) and that there exists $T>0$ such that $M_t = M_T$ for any $t\geq T$. Let $(Y_n)_{n\geq 1}$ and $(P_n)_{n\geq 1}$ be as in the proof of Theorem \ref{thm:BDGgeneralUMD}. Then $P_nM$ takes values in a finite dimensional space $Y_n^*$ and hence $[\![P_nM]\!]$ has increasing, adapted, and c\`adl\`ag version. Therefore we can fix $\Omega_0\subset\Omega$ of full measure which is an intersection of the following sets:
 \begin{enumerate}[(1)]
  \item $[\![P_nM]\!]$ is increasing c\`adl\`ag for any $n\geq 1$,
  \item $[\![M]\!]_T(x^*, y^*) = [\![P_nM]\!]_T(x^*, y^*)$ for any $x^*, y^*\in Y_n$ and for any $n\geq 1$,
  \item $[\![P_mM]\!]_r(x^*, y^*) = [\![P_nM]\!]_r(x^*, y^*)$ for any $r\in \mathbb Q$, for any $x^*, y^*\in Y_{m \wedge n}$, and for any $m, n\geq 1$,
  \item $\gamma([\![M]\!]_T) = \gamma([\![M]\!]_{\infty}) <\infty$.
 \end{enumerate}
 First notice that since all $[\![P_nM]\!]$, $n\geq 1$, are increasing c\`adl\`ag on $\Omega_0$, for any $t\geq 0$ (not necessarily rational) we have that
 \[
  [\![P_mM]\!]_t(x^*, y^*) = [\![P_nM]\!]_t(x^*, y^*),\;\;\;x^*, y^* \in Y_{m \wedge n},\;\;m, n\geq 1.
 \]
Let $F:\mathbb R_+ \times \Omega \to X\otimes X$ be a bilinear form-valued process such that 
\begin{equation}\label{eq:defofFforcadlag}
 F_t(x^*, y^*) = [\![P_nM]\!]_t(x^*, y^*),\;\;\; x^*, y^*\in Y_n,\;\; t\geq 0,
\end{equation}
for any $n\geq 1$, which existence can be shown analogously proof of Theorem \ref{thm:BDGgeneralUMD}.

First note that $F$ is adapted by the definition. Let us show that $F$ is increasing c\`adl\`ag on $\Omega_0$. Fix $\omega \in \Omega_0$. Then $F_t(x^*, x^*) \geq F_s(x^*, x^*)$ for any $t\geq s\geq 0$ and any $x^* \in Y := \cup_{n}Y_n \subset X^*$, and thus we have the same for any $x^*\in X^*$ by continuity of $F_t$ and $F_s$ and the fact that $Y$ is dense in $X^*$.

Now let us show that $F$ is right-continuous. By \eqref{eq:defofFforcadlag} and the fact that $[\![P_nM]\!]$ is c\`adl\`ag we have that 
$$
(F_{t+\eps}-F_t)(x^*, x^*) \to 0,\;\;\; \eps\to 0,\;\; x^*\in Y,
$$
so by Lemma \ref{lem:V_nmonto0thenthesamefor0} and the fact that $\gamma(F_T) = \gamma([\![M]\!]_T) <\infty$ we have that $\gamma(F_{t+\eps}-F_t) \to 0$ as $\eps \to 0$, and thus the desired right continuity follows from \eqref{eq:contofVduetogammanormtikinula}.

Finally, $F$ has left-hand limits. Indeed, fix $t> 0$ and let $F_{t-}$ be a bilinear form defined by
\[
 F_{t-}(x^*, y^*) := \lim_{\eps\to 0}  F_{t-}(x^*, y^*),\;\;\; x^*, y^*\in X^*.
\]
Then $\|F_{t-}-F_{t-{\eps}}\| \to 0$ as $\eps \to 0$ by Lemma \ref{lem:V_nmonto0thenthesamefor0}, \eqref{eq:contofVduetogammanormtikinula},  and the fact that $\gamma(F_T) = \gamma([\![M]\!]_T) <\infty$, so $F$ has left-hand limits.

Now we need to conclude with the fact that $F$ is a version of $[\![M]\!]$, which follows from the fact that by \eqref{eq:defofFforcadlag} for any fixed $t\geq 0$ a.s.
\[
 F_t(x^*, y^*) = [\![M]\!]_t(x^*, y^*),\;\;\; x^*, y^*\in Y,
\]
so by a.e.\ continuity of $F_t$ and $ [\![M]\!]_t$ on $X^*\times X^*$ we have the same for all $x^*, y^*\in X^*$, and thus $F_t = [\![M]\!]_t$ a.s.

The process $\gamma(F)$ is finite a.s.\ by the fact that $\gamma(F_T)<\infty$ a.s., increasing a.s.\ by the fact that $F$ is increasing a.s.\ and by Lemma \ref{lem:VgeqWthenVgammageqWgamma}, and adapted and c\`adl\`ag as $F$ is adapted and c\`adl\`ag and by the fact that the map $V \mapsto \gamma(V)$ is continuous by\eqref{eq:contofVduetogammanormtikinula}. 
\end{proof}

\section{Ramifications of Theorem \ref{thm:BDGgeneralUMD}}\label{sec:ramifications}

Let us outline some ramifications of Theorem \ref{thm:BDGgeneralUMD}.

\subsection{Continuous and purely discontinuous martingales}\label{subsec:contandpdmart}

In the following theorems we will consider separately the cases of continuous and purely discontinuous martingales. Recall that an $X$-valued martingale is called {\em purely discontinuous} if $[\langle M, x^*\rangle]$ is a.s.\ a pure jump process for any $x^*\in X^*$ (see \cite{Kal,JS,Y17FourUMD,Y17MartDec} for details). 

First we show that if $M$ is continuous, then Theorem \ref{thm:BDGgeneralUMD} holds for the whole range $0<p<\infty$.

\begin{theorem}\label{thm:0<p<inftyforcontmart}
 Let $X$ be a UMD Banach space, $M:\mathbb R_+ \times \Omega \to X$ be a continuous local martingale. Then we have that for any $0<p<\infty$
 \begin{equation}\label{eq:BDGgenmartcontmar0<p<infty}
  \mathbb E \sup_{0\leq s\leq t}\|M_s\|^p \eqsim_{p,X} \mathbb E  \gamma([\![M]\!]_t)^p,\;\;\; t\geq 0.
 \end{equation}
\end{theorem}

For the proof we will need the following technical lemma, which extends Proposition \ref{prop:cadlag[[M]]} in the case of a continuous martingale.

\begin{lemma}\label{lem:gamma([[M]])iscontgivenMisso}
 Let $X$ be a UMD Banach space, $M:\mathbb R_+ \times \Omega \to X$ be a continuous local martingale. Then the processes $[\![M]\!]$ and $\gamma([\![M]\!])$ have continuous versions.
\end{lemma}

\begin{proof}
The proof is entirely the same as the proof of Proposition \ref{prop:cadlag[[M]]}, one needs only to use the fact that by \cite[Theorem 26.6(iv)]{Kal} we can assume that $[\![P_nM]\!]$ is increasing {\em continuous} for any $n\geq 1$ on $\Omega_0$, so both $[\![M]\!]$ and $\gamma([\![M]\!])$ will not just have left-hand limits, but be left continuous, and thus continuous (as these processes are already right continuous).
\end{proof}

\begin{proof}[Proof of Theorem \ref{thm:0<p<inftyforcontmart}] The case $p\geq 1$ follows from Theorem \ref{thm:BDGgeneralUMD}. Let us treat the case $0<p<1$. First we show that $(\gamma([\![M]\!]_t))_{t\geq 0}$ is a predictable process: $(\gamma([\![M]\!]_t))_{t\geq 0}$ is a monotone limit of  processes $(\gamma([\![P_n^*M]\!]_t))_{t\geq 0}$ (where $P_n$'s are as in the proof of Theorem \ref{thm:BDGgeneralUMD}), which are predictable due to the fact that $([\![P_n^* M]\!]_t)_{t\geq 0}$ is a $Y_n^*\otimes Y_n^*$-valued predictable process and $\gamma:Y_n^*\otimes Y_n^* \to \mathbb R_+$ is a fixed measurable function. Moreover, by Lemma \ref{lem:gamma([[M]])iscontgivenMisso} $(\gamma([\![M]\!]_t))_{t\geq 0}$ is continuous a.s., and by Remark \ref{rem:basicpropsof[[M]]} and Lemma \ref{lem:VgeqWthenVgammageqWgamma} $(\gamma([\![M]\!]_t))_{t\geq 0}$ is increasing a.s.

Now since $(\gamma([\![M]\!]_t))_{t\geq 0}$ is continuous predictable increasing, \eqref{eq:BDGgenmartcontmar0<p<infty} follows from the case $p\geq 1$ and Lenglart's inequality (see \cite{Lenglart} and \cite[Proposition IV.4.7]{RY}).
\end{proof}

\begin{theorem}\label{thm:M^nto0iff[[M^n]]tozeroforcontmart}
 Let $X$ be a UMD Banach space, $(M^n)_{n\geq 1}$ be a sequence of $X$-valued continuous local martingales such that $M^n_0=0$ for all $n\geq 1$. Then $\sup_{t\geq 0} \|M^n_t\| \to 0$ in probability as $n\to \infty$ if and only if $\gamma([\![M^n]\!]_{\infty}) \to 0$ in probability as $n\to \infty$.
\end{theorem}

\begin{proof}
The proof follows from the classical argument due do Lenglart (see \cite{Lenglart}), but we will recall this argument for the convenience of the reader.
 We will show only one direction, the other direction follows analogously. Fix $\eps, \delta>0$. For each $n\geq 1$ define a stopping time $\tau_n$ in the following way:
 \[
  \tau_n:= \inf\{t\geq 0: M^n_t > \eps\}.
 \]
Then by \eqref{eq:thmBDGgeneralUMD} and Chebyshev's inequality
\begin{align*}
 \mathbb P(\gamma([\![M^n]\!]_{\infty}) > \delta) &\leq \mathbb P(\tau_n < \infty) + \mathbb P(\gamma([\![M^n]\!]_{\tau_n}) > \delta)\\
 & \leq \mathbb P (\sup_{t\geq 0} \|M^n_t\| > \eps) + \delta^{-\frac 12} \mathbb E \gamma([\![M^n]\!]_{\tau_n})^{\frac 12}\\
 & \lesssim_X \mathbb P (\sup_{t\geq 0} \|M^n_t\| > \eps) + \delta^{-\frac 12} \mathbb E \|M^n_{\tau_n}\|\\
 &\leq\mathbb P (\sup_{t\geq 0} \|M^n_t\| > \eps) +  \delta^{-\frac 12}  \eps, 
\end{align*}
and the latter vanishes for any fixed $\delta>0$ as $\eps \to 0$ and $n\to \infty$.
\end{proof}

\begin{remark}
Note that Theorem \ref{thm:M^nto0iff[[M^n]]tozeroforcontmart} does not hold for general martingales even in the real-valued case, see \cite[Exercise 26.5]{Kal}.
\end{remark}

For the next theorem recall that $\ell^2([0, t])$ is the {\em nonseparable} Hilbert space consisting of all functions $f:[0,t] \to \mathbb R$ which support $\{s\in [0,t]:f(s)\neq 0\}$ is countable and $\|f\|_{\ell^2([0, t])} := \sum_{0\leq s\leq t} |f(s)|^2 <\infty$.

\begin{theorem}\label{thm:BDGUMDpdmartingales}
 Let $X$ be a UMD Banach space, $1\leq p<\infty$, $M:\mathbb R_+ \times \Omega \to X$ be a purely discontinuous martingale. Then for any $t\geq 0$
 \[
  \mathbb E \sup_{0\leq s\leq t} \|M_s\|^p \eqsim_{p, X} \mathbb E \|(\Delta M_s)_{0\leq s\leq t}\|_{\gamma(\ell^2([0, t]), X)}^p.
 \]
\end{theorem}

(In this case a $\gamma$-norm is well-defined. Indeed, if $H$ is nonseparable, then an operator $T:\mathcal L(H, X)$ is said to an infinite $\gamma$-norm if there exists an uncountable orthonormal system $(h_{\alpha})_{\alpha \in \Lambda}$ such that $Th_{\alpha} \neq 0$ for any $\alpha \in \Lambda$. Otherwise, there exists a separable Hilbert subspace $H_0 \subset H$ such that $T|_{H_{0}^{\perp}} = 0$, and then we set $\|T\|_{\gamma(H, X)} := \|T|_{H_0}\|_{\gamma(H_0, X)} $).

\begin{proof}
 It is sufficient to notice that for any $x^*\in X^*$ a.s.\
 \[
  [\langle M, x^*\rangle]_t = \sum_{0\leq s\leq t} |\langle \Delta M_s, x^*\rangle|^2,
 \]
 and apply Theorem \ref{thm:BDGgeneralUMD} and Lemma \ref{lem:gammaofV=gammaofTfits}.
\end{proof}

\begin{remark} Note that martingales in Theorem \ref{thm:0<p<inftyforcontmart} and \ref{thm:BDGUMDpdmartingales} cover all the martingales if $X$ is UMD.
 More specifically, if $X$ has the UMD property, then any $X$-valued local martingale $M$ has a unique decomposition $M = M^c  + M^d$ into a sum of a continuous local martingale $M^c$ and a purely discontinuous local martingale $M^d$ (such a decomposition is called the {\em Meyer-Yoeurp decomposition}, and it characterizes the UMD property, see \cite{Y17MartDec,Y17GMY}).
\end{remark}

\subsection{Martingales with independent increments}\label{martwithindepincr}

Here we show that both Theorem \ref{thm:BDGgendiscmart} and \ref{thm:BDGgeneralUMD} hold in much more general Banach spaces given the corresponding martingale has independent increments.

\begin{proposition}\label{prop:indincBDGgenBanachspaces}
 Let $X$ be a Banach space, $(d_n)_{n\geq 1}$ be an $X$-valued martingale difference sequence with independent increments. Then for any $1 < p<\infty$
 \[
  \mathbb E \sup_{m\geq 1} \Bigl\|\sum_{n=1}^m d_n\Bigr\|^p \lesssim_{p} \mathbb E \|(d_n)_{n=1}^{\infty}\|_{\gamma(\ell^2, X)}^p.
 \]
Moreover, if $X$ has a finite cotype, then
\[
  \mathbb E \sup_{m\geq 1} \Bigl\|\sum_{n=1}^m d_n\Bigr\|^p \eqsim_{p, X} \mathbb E \|(d_n)_{n=1}^{\infty}\|_{\gamma(\ell^2, X)}^p.
 \]
\end{proposition}

\begin{proof}
 Let $(r_n)_{n\geq 1}$ be a sequence of independent Rademacher random variables, $(\gamma_n)_{n\geq 1}$ be a sequence of independent standard Gaussian random variables. Then
 \begin{align*}
  \mathbb E \sup_{m\geq 1}\Bigl\|\sum_{n=1}^{m} d_n\Bigr\|^p &\stackrel{(i)}\eqsim_p \mathbb E\Bigl\|\sum_{n=1}^{\infty} d_n\Bigr\|^p  \stackrel{(ii)}\eqsim_{p}\mathbb E\mathbb E_{r} \Bigl\|\sum_{n=1}^{N} r_n d_n\Bigr\|^p\\
  &\stackrel{(iii)}\lesssim_{p}\mathbb E \mathbb E_{\gamma}\Bigl\|\sum_{n=1}^{N} \gamma_n d_n\Bigr\|^p \stackrel{(iv)}\eqsim_{p}\mathbb E \Bigl(\mathbb E_{\gamma} \Bigl\|\sum_{n=1}^{N} \gamma_n d_n\Bigr\|^2\Bigr)^{\frac p2}\\
  &= \mathbb E \|(d_n)_{n=1}^{\infty}\|_{\gamma(\ell^2, X)}^p,
 \end{align*}
 where $(i)$ follows from \eqref{eq:DoobsineqXBanach}, $(ii)$ follows from \cite[Lemma 6.3]{LT91}, $(iii)$ holds by \cite[Proposition 6.3.2]{HNVW2}, and finally $(iv)$ follows from \cite[Proposition 6.3.1]{HNVW2}.
 
 If $X$ has a finite cotype, then one has $\eqsim_{p, X}$ instead of $\lesssim_{p}$ in $(iii)$ (see \cite[Corollary 7.2.10]{HNVW2}), and the second part of the proposition follows.
\end{proof}

Based on Proposition \ref{prop:indincBDGgenBanachspaces} and the proof of Theorem \ref{thm:BDGgeneralUMD} one can show the following assertion. Notice that we presume the reflexivity of $X$ since it was assumed in the whole Section \ref{sec:Gaussiancharacteristic}.

\begin{proposition}
 Let $X$ be a reflexive Banach space, $1\leq p<\infty$, $M:\mathbb R_+\times \Omega \to X$ be an $L^p$-bounded martingale with independent increments such that $M_0=0$. Let $t\geq 0$. If $M$ has a covariation bilinear form $[\![M]\!]_t$ at $t$, then
 \[
  \mathbb E \sup_{0\leq s\leq t}\|M_s\|^p \lesssim_{p, X} \mathbb E \gamma([\![M]\!]_t)^p.
 \]
Moreover, if $X$ has a finite cotype, then the existence of $[\![M]\!]_t$ is guaranteed, and
\[
  \mathbb E \sup_{0\leq s\leq t}\|M_s\|^p \eqsim_{p, X} \mathbb E \gamma([\![M]\!]_t)^p.
 \]
\end{proposition}

\begin{proof}
 The proof coincides with the proof of Theorem \ref{thm:BDGgeneralUMD}, but one needs to use Proposition \ref{prop:indincBDGgenBanachspaces} instead of Theorem \ref{thm:BDGgendiscmart}.
\end{proof}

\subsection{One-sided estimates}\label{subsec:onesided}

In practice one often needs only the upper bound of \eqref{eq:discBDGUMDBanach}. It turns out that existence of such estimates for a fixed Banach space $X$ is equivalent to the fact that $X$ has {\em the UMD$^-$ property}.

\begin{definition}
A Banach space $X$ is called a {\em UMD$^-$ space} if for some (equivalently, for all)
$p \in (1,\infty)$ there exists a constant $\beta>0$ such that
for every $n \geq 1$, every martingale
difference sequence $(d_j)^n_{j=1}$ in $L^p(\Omega; X)$, and every sequence
$(r_j)^n_{j=1}$ of independent Rademachers
we have
\[
\Bigl(\mathbb E \Bigl\| \sum^n_{j=1} d_j\Bigr\|^p\Bigr )^{\frac 1p}
\leq \beta \Bigl(\mathbb E \Bigl \| \sum^n_{j=1}r_jd_j\Bigr\|^p\Bigr )^{\frac 1p}.
\]
The least admissible constant $\beta$ is denoted by $\beta_{p,X}^{-}$ and is called the {\em UMD$^-$ constant}.
\end{definition}

By the definition of the UMD property and a triangular inequality one can show that UMD implies UMD$^-$. Moreover, UMD$^-$ is a strictly bigger family of Banach spaces and includes nonreflexive Banach spaces such as $L^1$. The reader can find more information
on UMD$^-$ spaces in \cite{HNVW1,Ver07,CG,Gar90,Gar85,Geiss99}.

The following theorem presents the desired equivalence.

\begin{theorem}\label{thm:upperboundsUMD-}
 Let $X$ be a Banach space, $1\leq p<\infty$. Then $X$ has the UMD$^-$ property if and only if one has that for any $X$-valued martingale difference sequence $(d_n)_{n=1}^m$
 \begin{equation}\label{eq:discBDGUMD-Banach}
   \mathbb E \sup_{m\geq 1}\Bigl\|\sum_{n=1}^{m} d_n\Bigr\|^p \lesssim_{p, X} \mathbb E \|(d_n)_{n=1}^{\infty}\|_{\gamma(\ell^2, X)}^p.
 \end{equation}
\end{theorem}

\begin{proof}
Assume that $X$ has the UMD$^-$ property. Let  $(d_n)_{n=1}^m$ be an $X$-valued martingale difference sequence. Then we have that for a sequence  $(r_n)_{n\geq 1}$ of independent Rademacher random variables and for a sequence $(\gamma_n)_{n\geq 1}$ of independent standard Gaussian random variables
 \begin{align*}
  \mathbb E \sup_{m\geq 1}\Bigl\|\sum_{n=1}^{m} d_n\Bigr\|^p &\stackrel{(i)}\lesssim_p \mathbb E   \mathbb E_{r}\sup_{m\geq 1} \Bigl\|\sum_{n=1}^{N} r_n d_n\Bigr\|^p \stackrel{(ii)}\eqsim_{p,X}\mathbb E\mathbb E_{r} \Bigl\|\sum_{n=1}^{N} r_n d_n\Bigr\|^p\nonumber\\
  &\stackrel{(iii)}\eqsim_{p,X}\mathbb E \mathbb E_{\gamma}\Bigl\|\sum_{n=1}^{N} \gamma_n d_n\Bigr\|^p \stackrel{(iv)}\eqsim_{p}\mathbb E \Bigl(\mathbb E_{\gamma} \Bigl\|\sum_{n=1}^{N} \gamma_n d_n\Bigr\|^2\Bigr)^{\frac p2}\\
  &= \mathbb E \|(d_n)_{n=1}^{\infty}\|_{\gamma(\ell^2, X)}^p,\nonumber
 \end{align*}
where $(i)$ follows from \cite[(8.22)]{Burk86}, $(ii)$ holds by \cite[Proposition 6.1.12]{HNVW2}, $(iii)$ follows from \cite[Corollaries 7.2.10 and 7.3.14]{HNVW2}, and $(iv)$ follows from \cite[Proposition 6.3.1]{HNVW2}.

Let us show the converse. Assume that \eqref{eq:discBDGUMD-Banach} holds for any $X$-valued martingale difference sequence $(d_n)_{n=1}^m$. Then $X$ has a finite cotype by \cite[Corollary 7.3.14.]{HNVW2}, and the desired UMD$^-$ property follows from \cite[Corollary 7.2.10]{HNVW2}.
\end{proof}

\begin{remark}
 Unfortunately, it remains open whether one can prove the upper bound of \eqref{eq:thmBDGgeneralUMD} given $X$ has the UMD$^{-}$ property. The problem is in the approximation argument employed in the proof of \eqref{eq:thmBDGgeneralUMD}: we can not use an increasing sequence $(Y_n)_{n\geq 1}$ of finite dimensional subspaces of $X^*$ since we can not guarantee that $\beta_{p, Y_n^*}^-$ does not blow up as $n\to \infty$ (recall that $\beta_{p, Y_n^*} \leq \beta_{p, X}$ by the duality argument, see \cite[Proposition 4.2.17]{HNVW1}). Nonetheless, such an upper bound can be shown for $X=L^1$ by an {\em ad hoc} argument (by using an increasing sequence of projections onto finite-dimensional $L^1$-spaces).
\end{remark}

\section{Applications and miscellanea}\label{subsec:appandmis}

Here we provide further applications of Theorem \ref{thm:BDGgeneralUMD}.

\subsection{It\^o isomorphism: general martingales}\label{subsec:applicItoisomgenmart}

Let $H$ be a Hilbert space, $X$ be a Banach space. For each $x\in X$ and $h \in H$ we denote the linear operator $g\mapsto \langle g, h\rangle x$, $g\in H$, by $h\otimes x$. The process $\Phi: \mathbb R_+ \times \Omega \to \mathcal L(H,X)$ is called  \textit{elementary predictable}
with respect to the filtration $\mathbb F = (\mathcal F_t)_{t \geq 0}$ if it is of the form
\begin{equation}\label{eq:elprog}
 \Phi(t,\omega) = \sum_{k=1}^K\sum_{m=1}^M \mathbf 1_{(t_{k-1},t_k]\times B_{mk}}(t,\omega)
\sum_{n=1}^N h_n \otimes x_{kmn},\;\;\; t\geq 0,\;\; \omega \in \Omega,
\end{equation}
where $0 = t_0 < \ldots < t_K <\infty$, for each $k = 1,\ldots, K$ the sets
$B_{1k},\ldots,B_{Mk}$ are in $\mathcal F_{t_{k-1}}$ and the vectors $h_1,\ldots,h_N$ are in $H$.
Let $\widetilde M:\mathbb R_+ \times \Omega \to H$ be a local martingale. Then we define the {\em stochastic integral} $\Phi \cdot \widetilde M:\mathbb R_+ \times \Omega \to X$ of $\Phi$ with respect to $\widetilde M$ as follows:
\begin{equation}\label{eq:defofstochintwrtM}
 (\Phi \cdot \widetilde M)_t := \sum_{k=1}^K\sum_{m=1}^M \mathbf 1_{B_{mk}}
\sum_{n=1}^N \langle(\widetilde M(t_k\wedge t)- \widetilde M(t_{k-1}\wedge t)), h_n\rangle x_{kmn},\;\; t\geq 0.
\end{equation}

Notice that for any $t\geq 0$ the stochastic integral $\Phi \cdot \widetilde M$ obtains a covariation bilinear form $[\![\Phi \cdot \widetilde M]\!]_t$ which is a.s.\ continuous on $X^*\times X^*$ and which has the following form due to \eqref{eq:whyq_Misinportant5element} and \eqref{eq:defofstochintwrtM}
\begin{equation}\label{eq:[[M]]forstochintr}
\begin{split}
  [\![\Phi \cdot \widetilde M]\!]_t(x^*, x^*) &=   \Bigl[\Bigl\langle \int_0^{\cdot} \Phi \ud \widetilde M, x^* \Bigr\rangle\Bigr]_t = \Bigl[ \int_0^{\cdot} (\Phi^* x^*)^* \ud \widetilde M \Bigr]_t\\
  &= \int_0^t \|q_{\widetilde M}^{1/2}(s)\Phi^*(s)x^*\|^2 \ud [\widetilde M]_s,\;\;\; t\geq 0.
\end{split}
\end{equation}

\begin{remark}\label{rem:X=Rstochintextension}
 If $X = \mathbb R$, then by the real-valued Burkholder-Davis-Gundy inequality and the fact that for any elementary predictable $\Phi$
 \[
  \Bigl[ \int_0^{\cdot} \Phi \ud \widetilde M \Bigr]_t 
  = \int_0^t \|q_{\widetilde M}^{1/2}(s)\Phi^*(s)\|^2 \ud [\widetilde M]_s,\;\;\; t\geq 0,
 \]
 one has an isomorphism
 \[
  \mathbb E \sup_{t\geq 0} |(\Phi \cdot \widetilde M)_t| \eqsim \mathbb E \Bigl(\int_0^{\infty} \|q_{\widetilde M}^{1/2}(s)\Phi(s)\|^2 \ud [\widetilde M]_s\Bigr)^{\frac 12},
 \]
so
one can extend the definition of a stochastic integral to {\em all} predictable $\Phi:\mathbb R_+ \times \Omega \to H$ with
\begin{equation}\label{eq:PhiintwrtwidetildeMscalar}
  \mathbb E \Bigl(\int_0^{\infty} \|q_{\widetilde M}^{1/2}(s)\Phi(s)\|^2 \ud [\widetilde M]_s\Bigr)^{\frac 12}<\infty,
\end{equation}
by extending the stochastic integral operator from a dense subspace of all elementary predictable processes satisfying \eqref{eq:PhiintwrtwidetildeMscalar}.
We refer the reader to \cite{Mey77,MP,Kal} for details.
\end{remark}

\begin{remark}\label{rem:X=R^dstochintextension}
 Let $X = \mathbb R^d$ for some $d\geq 1$. Then analogously to Remark \ref{rem:X=Rstochintextension} one can extend the definition of a stochastic integral to all predictable processes $\Phi:\mathbb R_+ \times \Omega \to \mathcal L(H, \mathbb R^d)$ with
 \begin{align*}
  \mathbb E \Bigl(\sum_{n=1}^d \int_0^{\infty} \|q_{\widetilde M}^{1/2}(s)\Phi^*(s)e_n\|^2 \ud [\widetilde M]_s\Bigr)^{\frac 12}&= \mathbb E \|q_{\widetilde M}^{1/2}\Phi^*\|_{HS(\mathbb R^d, L^2(\mathbb R_+; [\widetilde M]))}\\
  &= \mathbb E \|\Phi q_{\widetilde M}^{1/2}\|_{HS( L^2(\mathbb R_+; [\widetilde M]),\mathbb R^d)}<\infty,
 \end{align*}
 where $(e_n)_{n=1}^d$ is a basis of $\mathbb R^d$, $\|T\|_{HS(H_1, H_2)}$ is the Hilbert-Schmidt norm of an operator $T$ acting form a Hilbert space $H_1$ to a Hilbert space $H_2$, and $ L^2(\mathbb R_+; A)$ for a given increasing $A:\mathbb R_+ \to \mathbb R$ is a Hilbert space of all functions $f:\mathbb R_+ \to \mathbb R$ such that $ \int_{\mathbb R_+} \|f(s)\|^2 \ud A(s)<\infty$.
\end{remark}

Now we present the It\^o isomorphism for vector-valued stochastic integrals with respect to general martingales, which extends \cite{VY16,NVW,Ver}.

\begin{theorem}\label{thm:stochintelempredPhi}
 Let $H$ be a Hilbert space, $X$ be a UMD Banach space, $\widetilde M:\mathbb R_+ \times \Omega \to H$ be a local martingale, $\Phi:\mathbb R_+ \times \Omega \to \mathcal L(H,X)$ be elementary predictable. Then for all $1\leq p<\infty$
 \[
  \mathbb E \sup_{0\leq s\leq t}\Bigl\|\int_0^s \Phi \ud \widetilde M\Bigr\|^p \eqsim_{p,X} \mathbb E \|\Phi q_{\widetilde M}^{1/2}\|^p_{\gamma(L^2([0,t], [\widetilde M];H),X)},\;\;\; t\geq 0,
 \]
where $[\widetilde M]$ is the quadratic variation of $\widetilde M$, $q_{\widetilde M}$ is the quadratic variation derivative (see Subsection \ref{subsec:qv+q_M}), and $\|\Phi q_{\widetilde M}^{1/2}\|^p_{\gamma(L^2([0,t], [\widetilde M];H),X)}$ is the $\gamma$-norm (see \eqref{eq:defofgammanormsnove}).
\end{theorem}

\begin{proof} Fix $t\geq 0$. Then the theorem holds by Theorem \ref{thm:BDGgeneralUMD}, Lemma \ref{lem:gammaofV=gammaofTfits}, and the fact that by \eqref{eq:[[M]]forstochintr}  for any fixed $x^*\in X^*$ a.s.\
 \begin{align*}
    \Bigl[\Bigl\langle \int_0^{\cdot} \Phi \ud \widetilde M, x^*\Bigr\rangle\Bigr]_t = \Bigl[ \int_0^{\cdot} \langle\Phi, x^* \rangle\ud \widetilde M \Bigr]_t &= \int_0^t \|q_M^{\frac 12} \Phi^*x^*\|^2 \ud [\widetilde M]_s\\
    &= \|q_M^{\frac 12} \Phi^*x^*\|^2_{L^2([0,t], [\widetilde M];H)}.
 \end{align*}
\end{proof}

Theorem \ref{thm:stochintelempredPhi} allows us to provide the following general stochastic integration result. Recall that a predictable process $\Phi:\mathbb R_+ \times \Omega \to \mathcal L(H,X)$ is called {\em strongly predictable} if there exists a sequence $(\Phi_n)_{n\geq 1}$ of elementary predictable $\mathcal L(H,X)$-valued processes such that $\Phi$ is a pointwise limit of $(\Phi_n)_{n\geq 1}$.

\begin{corollary}\label{cor:stochingenmardenfunction}
 Let $H$ be a Hilbert space, $X$ be a UMD Banach space, $\widetilde M:\mathbb R_+ \times \Omega \to H$ be a local martingale, $\Phi:\mathbb R_+ \times \Omega \to \mathcal L(H,X)$ be strongly predictable such that $\mathbb E \|\Phi q_{\widetilde M}^{1/2}\|_{\gamma(L^2(\mathbb R_+, [\widetilde M];H),X)}<\infty$. Then there exists a martingale $\Phi \cdot \widetilde M$ which coincides with the stochastic integral given $\Phi$ is elementary predictable such that
 \begin{equation}\label{eq:stochintwrtgenPhiweaksence}
  \langle \Phi \cdot \widetilde M, x^*\rangle = (\Phi^* x^*)\cdot \widetilde M,\;\;\; x^*\in X^*,
 \end{equation}
where the latter integral is defined as in Remark \ref{rem:X=Rstochintextension}. Moreover,  then we have that for any $1\leq p<\infty$
 \begin{equation}\label{eq:corgenstochintegration}
  \mathbb E \sup_{t\geq 0}\|(\Phi \cdot \widetilde M)_t\|^p \eqsim_{p, X} \mathbb E \|\Phi q_{\widetilde M}^{1/2}\|^p_{\gamma(L^2(\mathbb R_+, [\widetilde M];H),X)}.
 \end{equation}
\end{corollary}

For the proof we will need the following technical lemma.

\begin{lemma}\label{lem:convergenceofY_n^*toXvaluedguy}
 Let $X$ be a reflexive separable Banach space, $Y_1 \subset Y_2 \subset \ldots \subset Y_n \subset \ldots \subset X^*$ be finite dimensional subspaces such that $\overline{\cup_n Y_n} = X^*$. Let $P_n : Y_n\hookrightarrow X^*$, $n\geq 1$, and $P_{n,m}: Y_n\hookrightarrow Y_m$, $m\geq n\geq 1$, be the inclusion operators. For each $n\geq 1$ let $x_n\in Y_n^*$ be such that $P_{n,m}^* x_m = x_n$ for all $m\geq n\geq 1$. Assume also that $\sup_n \|x_n\| <\infty$. Then there exists $x\in X$ such that $P_{n}^* x = x_n$ for all $n\geq 1$ and $\|x\| = \lim_{n\to \infty} \|x_n\|$ monotonically.
\end{lemma}

\begin{proof}
 Set $C = \sup_{n}\|x_n\|$. First notice that $(x_n)_{n\geq 1}$ defines a bounded linear functional on $Y = \cup_n Y_n$. Indeed, fix $y\in Y_n$ for some fixed $n\geq 1$ (then automatically $y\in Y_m$ for any $m\geq n$). Define $\ell(y) = \langle x_n, y\rangle$. Then this definition of $\ell$ agrees for different $n$'s since for any $m\geq n$ we have that \
 $$
 \langle x_m, y_n\rangle = \langle x_m, P_{n,m}y_n\rangle = \langle P_{n,m}^* x_m, y_n\rangle = \langle x_n, y_n\rangle.
 $$
 Moreover, this linear functional is bounded since $|\langle x_n, y_n\rangle| \leq \|x_n\| \|y_n\| \leq C \|y_n\|$. So, it can be continuously extended to the whole space $X^*$. Since $X$ is reflexive, there exists $x\in X$ such that $\ell(x^*) = \langle x^*, x\rangle$ for any $x^*\in X^*$. Then for any fixed $n\geq 1$ and for any $y\in Y_n$ we have that
 \[
 \langle x_n, y \rangle = \ell(y) =  \langle x, y\rangle = \langle x, P_ny\rangle  = \langle P_n^*x, y\rangle ,
 \]
so $P_n^*x = x_n$. The latter follows from the fact that $\|P_n^* x\| \to \|x\|$ monotonically as $n\to \infty$ for any $x\in X$.
\end{proof}

\begin{proof}[Proof of Corollary \ref{cor:stochingenmardenfunction}]
 We will first consider the finite dimensional case and then deduce the infinite dimensional case.
 
 {\em Finite dimensional case.} Since $X$ is finite dimensional, it is isomorphic to a finite dimensional Euclidean space, and so the $\gamma$-norm is equivalent to the Hilbert-Schmidt norm (see e.g.\ \cite[Proposition 9.1.9]{HNVW2}). Then $\Phi$ is stochastically integrable with respect to $\widetilde M$ due to Remark \ref{rem:X=R^dstochintextension}, so \eqref{eq:stochintwrtgenPhiweaksence} clearly holds and we have that for any $x^* \in X^*$ a.s.\
 \[
  [\langle \Phi \cdot \widetilde M, x^*\rangle]_t = [(\Phi^*x^*)\cdot \widetilde M]_t = \int_0^t \|q_{\widetilde M}^{1/2}(s)\Phi^*(s) x^*\|^2 \ud [\widetilde M]_s,\;\;\; t\geq 0,
 \]
thus \eqref{eq:corgenstochintegration} follows from Theorem \ref{thm:BDGgeneralUMD} and Lemma \ref{lem:gammaofV=gammaofTfits}.
 
 {\em Infinite dimensional case.} Let now $X$ be general. Since $\Phi$ is strongly predictable, it takes values in a separable subspace of $X$, so we may assume that $X$ is separable. Since $X$ is UMD, it is reflexive, so $X^*$ is separable as well, and there exists a sequence $Y_1\subset Y_2 \subset \ldots\subset Y_n \subset\ldots \subset X^*$ of finite dimensional subsets of $X^*$ such that $\overline{\cup_n Y_n} = X^*$. For each $m\geq n\geq 1$ define inclusion operators $P_n:Y_n \hookrightarrow X^*$ and $P_{n,m}:Y_n \hookrightarrow Y_m$. Notice that by the ideal property \cite[Theorem 9.1.10]{HNVW2} $\mathbb E \|P_n^* \Phi q_{\widetilde M}^{1/2}\|_{\gamma(L^2(\mathbb R_+, [\widetilde M];H),Y_n^*)}<\infty$ for any $n\geq 1$, so since $Y_n^*$ is finite dimensional, the stochastic integral $(P_n^* \Phi) \cdot \widetilde M$ is well-defined by the case above and
 \begin{equation}\label{eq:Y_n^*valuedrestrictionofZmart}
    \mathbb E\sup_{t\geq 0} \bigl\|\bigl((P_n^* \Phi) \cdot \widetilde M\bigr)_t\bigr\| \eqsim_{X}\mathbb E \|P_n^* \Phi q_{\widetilde M}^{1/2}\|_{\gamma(L^2(\mathbb R_+, [\widetilde M];H),Y_n^*)},
 \end{equation}
 where the equivalence is independent of $n$ since $Y_n \subset X^*$ for all $n\geq 1$ and due to \cite[Proposition 4.2.17]{HNVW1} and Theorem \ref{thm:BDGgeneralUMD}.
Denote the stochastic integral $(P_n^* \Phi) \cdot \widetilde M$ by $Z^n$. Note that $Z^n$ is $Y_n^*$-valued, and since $P_{n,m}^* P_m^*\Phi = P_n^* \Phi$ for all $m\geq n\geq 1$, $P_{m,n}^* Z^m_t = Z^n_t$ a.s.\ for any $t\geq 0$. Therefore by Lemma \ref{lem:convergenceofY_n^*toXvaluedguy} there exists a process $Z:\mathbb R_+ \times \Omega \to X$ such that $P_n^*Z = Z^n$ for all $n\geq 1$. Let us show that $Z$ is integrable. Fix $t\geq 1$. Notice that by Lemma \ref{lem:convergenceofY_n^*toXvaluedguy} the limit $\|Z_t\| = \lim_{n\to \infty} \|P_n^* Z_t\| = \lim_{n\to \infty} \|Z^n_t\|$ is monotone, so by the monotone convergence theorem, \eqref{eq:Y_n^*valuedrestrictionofZmart}, and the ideal property \cite[Theorem 9.1.10]{HNVW2}
\begin{align*}
 \mathbb E \|Z_t\| = \lim_{n\to \infty} \mathbb E \|Z^n_t\|&\lesssim_X \limsup_{n\to \infty} \mathbb E \|P_n^* \Phi q_{\widetilde M}^{1/2}\|_{\gamma(L^2(\mathbb R_+, [\widetilde M];H),Y_n^*)}\\
 &\leq \mathbb E \|\Phi q_{\widetilde M}^{1/2}\|_{\gamma(L^2(\mathbb R_+, [\widetilde M];H),X)}.
\end{align*}
Now let us show that $Z$ is a martingale. Since $Z$ is integrable, due to \cite[Section 2.6]{HNVW1} it is sufficient to show that $\mathbb E (\langle Z_t, x^*\rangle|\mathcal F_s) = \langle Z_s, x^*\rangle$ for all $0\leq s\leq t$ for all $x^*$ from some dense subspace $Y$ of $X^*$. Set $Y = \cup_n Y_n$ and $x^* \in Y_n$ for some $n\geq 1$. Then for all $0\leq s\leq t$
\begin{align*}
  \mathbb E (\langle Z_t, x^*\rangle|\mathcal F_s)  &=\mathbb E (\langle Z_t, P_n x^*\rangle|\mathcal F_s)=\mathbb E (\langle P_n^* Z_t, x^*\rangle|\mathcal F_s)\\
  &=\mathbb E (\langle  Z^n_t, x^*\rangle|\mathcal F_s) = \langle Z^n_s, x^*\rangle= \langle Z_s, x^*\rangle,
\end{align*}
so $Z$ is a martingale. Finally, let us show \eqref{eq:corgenstochintegration}. First notice that for any $n\geq 1$ and $x^* \in Y_n \subset X^*$ a.s.\
\[
 [\langle Z, x^*\rangle]_t = [\langle Z^n, x^*\rangle]_t = \int_0^t \|q_{\widetilde M}^{1/2}(s)\Phi^*(s) x^*\|^2 \ud [\widetilde M]_s,\;\;\; t\geq 0;
\]
the same holds for a general $x^*\in X^*$ by a density argument. Then \eqref{eq:corgenstochintegration} follows from Theorem \ref{thm:BDGgeneralUMD} and Lemma \ref{lem:gammaofV=gammaofTfits}.
\end{proof}

\begin{remark}
As the reader can judge, the basic assumptions on $\Phi$ in Corollary \ref{cor:stochingenmardenfunction} can be weakened by a stopping time argument. Namely, one can assume that $\Phi q_{\widetilde M}^{1/2}$ is {\em locally} in $L^1(\Omega, \gamma(L^2(\mathbb R_+, [\widetilde M];H),X))$ (i.e.\ there exists an increasing sequence $(\tau_n)_{n\geq 1}$ of stopping times such that $\tau_n \to \infty$ a.s.\ as $n\to \infty$ and $\Phi q_{\widetilde M}^{1/2} \mathbf 1_{[0,\tau_n]}$ is in $L^1(\Omega, \gamma(L^2(\mathbb R_+, [\widetilde M];H),X))$ for all $n\geq 1$). Notice that such an assumption is a natural generalization of classical assumptions for stochastic integration in the real-valued case (see e.g.\ \cite[p.\ 526]{Kal}). 

In the case when $\widetilde M$ is continuous by a standard localization argument (since $t\mapsto \| \Phi q_{\widetilde M}^{1/2}\|_{\gamma(L^2([0, t], [\widetilde M];H),X)}$ is continuous) one can assume even a weaker assumption, namely that $\Phi q_{\widetilde M}^{1/2}$ is locally in $\gamma(L^2(\mathbb R_+, [\widetilde M];H),X)$,  see e.g.\ \cite{NVW,VY16,Kal}.
\end{remark}

In the theory of stochastic integration one might be interested in one-sided estimates. In the following proposition we show that such type of estimates is possible if $X$ satisfies the UMD$^-$ property (see Subsection \ref{subsec:onesided}).

\begin{proposition}
  Let $H$ be a Hilbert space, $X$ be a UMD$^-$ Banach space, $\widetilde M:\mathbb R_+ \times \Omega \to H$ be a local martingale, $1 \leq p<\infty$, $\Phi:\mathbb R_+ \times \Omega \to \mathcal L(H,X)$ be strongly predictable such that $\mathbb E \|\Phi q_{\widetilde M}^{1/2}\|_{\gamma(L^2(\mathbb R_+, [\widetilde M];H),X)}^p<\infty$ and such that there exists a sequence $(\Phi)_{n\geq 1}$ of elementary predictable $\mathcal L(H, X)$-valued processes such that
  \[
   \mathbb E \|(\Phi - \Phi_n) q_{\widetilde M}^{1/2}\|_{\gamma(L^2(\mathbb R_+, [\widetilde M];H),X)}^p \to 0,\;\;\; n\to \infty.
  \]
Then there exists an $L^p$-bounded martingale $\Phi \cdot \widetilde M$ as a strong $L^p$-limit of $(\Phi_n \cdot \widetilde M)_{n\geq 1}$, and we have that for any $1\leq p<\infty$
 \begin{equation}\label{eq:corgenstochintegrationUMD-}
  \mathbb E \sup_{t\geq 0}\|(\Phi \cdot \widetilde M)_t\|^p \lesssim_{p, X} \mathbb E \|\Phi q_{\widetilde M}^{1/2}\|^p_{\gamma(L^2(\mathbb R_+, [\widetilde M];H),X)}.
 \end{equation}
\end{proposition}

\begin{proof}
 Inequality \eqref{eq:corgenstochintegrationUMD-} for $\Phi = \Phi_n$ follows from Theorem \ref{thm:upperboundsUMD-}, while the proposition together with \eqref{eq:corgenstochintegrationUMD-} for a general $\Phi$ follows from a simple limiting argument.
\end{proof}

\subsection{It\^o isomorphism: Poisson and general random measures}\label{subsec:ItoisomPrmandgrm}

Let $(J, \mathcal J)$ be a measurable space, $N$ be a Poisson random measure on $J \times \mathbb R_+$, $\widetilde N$ be the corresponding compensated Poisson random measure (see e.g.\ \cite{Dirk14,Kal,Sato,KingPois,GY19} for details). Then by Theorem \ref{thm:BDGUMDpdmartingales} for any UMD Banach space $X$, for any $1\leq p<\infty$, and for any elementary predictable $F:J\times R_+ \times \Omega \to X$ we have that
\begin{equation}\label{eq:stochintwrttildeNofXvaluedguy}
  \mathbb E \sup_{0\leq s\leq t} \Bigl\|\int_{J\times [0,s]} F\ud \widetilde N\Bigr\|^p \eqsim_{p, X} \mathbb E \|F\|^p_{\gamma(L^2(J\times [0,t]; N), X)},\;\;\; t\geq 0.
\end{equation}

The same holds for a general quasi-left continuous random measure (see \cite{JS,DY17,Nov75,Mar13,Grig71,KalRM} for the definition and the details): if $\mu$ is a general quasi-left continuous random measure on $J \times \mathbb R_+$, $\nu$ is its compensator, and $\bar{\mu} := \mu-\nu$, then for any $1\leq p<\infty$
\begin{equation}\label{eq:stochintwrtbarmuofXvaluedguy}
  \mathbb E \sup_{0\leq s\leq t}\Bigl\|\int_{J\times [0,t]} F\ud  \bar{\mu}\Bigr\|^p \eqsim_{p, X} \mathbb E \|F\|^p_{\gamma(L^2(J\times [0,t]; \mu), X)},\;\;\; t\geq 0.
\end{equation}

The disadvantage of right-hand sides of \eqref{eq:stochintwrttildeNofXvaluedguy} and \eqref{eq:stochintwrtbarmuofXvaluedguy} is that both of them are not predictable and do not depend continuously on time a.s.\ on $\Omega$ (therefore they seem not to be useful from the SPDE's point of view since one may not produce a fixed point argument). For example, if $X = L^q$ for some $1<q<\infty$, then such predictable a.s.\ continuous in time right-hand sides do exist (see \cite{DY17,Dirk14}). For general UMD Banach spaces those can be provided so far only by decoupled tangent martingales, see \cite{Y19}.

\subsection{Necessity of the UMD property}\label{subsec:necofUMDgeneralrhs}

As it follows from Remark \ref{rem:nesofUMDforgenBDGwithgamma}, Theorem \ref{thm:BDGgeneralUMD} holds only in the UMD setting. The natural question is whether there exists an appropriate right-hand side of \eqref{eq:thmBDGgeneralUMD} in terms of $([\langle M, x^*\rangle,\langle M, y^*\rangle])_{x^*, y^*\in X^*}$ for some non-UMD Banach space $X$ and some $1\leq p<\infty$. Here we show that this is impossible.

Assume that for some Banach space $X$ and some $1\leq p<\infty$ there exists a function $G$ acting on families of stochastic processes parametrized by $X^*\times X^*$ (i.e.\ each family has the form $V = (V_{x^*, y^*})_{x^*, y^*\in X^*}$) taking values in $\mathbb R$ such that for any $X$-valued local martingale $M$ starting in zero we have that
\begin{equation}\label{eq:necofUMDgenLHS}
\mathbb E \sup_{t\geq 0} \|M_t\|^p \eqsim_{p, X} G([\![M]\!]),
\end{equation}
where we denote $[\![M]\!] = ([\langle M, x^*\rangle,\langle M, y^*\rangle])_{x^*, y^*\in X^*}$ for simplicity (note that the latter might not have a proper bilinear structure). Let us show that then $X$ must have the UMD property.

Fix any $X$-valued $L^p$-bounded martingale difference sequence $(d_n)_{n= 1}^N$ and any $\{-1,1\}$-valued sequence $(\eps_n)_{n= 1}^N$. Let $e_n:= \eps_nd_n$ for all $n=1,\ldots,N$. For every $x^*, y^*\in X^*$ define a stochastic process $V_{x^*,y^*}:\mathbb R_+ \times \Omega \to \mathbb R$ as
\[
V_{x^*}(t) = \sum_{n=1}^{N \wedge [t]} \langle d_n, x^*\rangle \cdot  \langle d_n, y^*\rangle =  \sum_{n=1}^{N \wedge [t]} \langle e_n, x^*\rangle \cdot  \langle e_n, y^*\rangle ,\;\;\; t\geq 0
\]
(recall that $[t]$ is the integer part of $t$). Let $V := (V_{x^*, y^*})_{x^*,y^*\in X^*}$. Then by \eqref{eq:necofUMDgenLHS}
\begin{equation}\label{eq:necofUMDgenLHSequivalenceoftransform}
\mathbb E \sup_{k\geq 0} \Bigl\|\sum_{n=1}^k d_n\Bigr\|^p\eqsim_{p, X} G(V) \eqsim_{p, X} \mathbb E \sup_{k\geq 0} \Bigl\|\sum_{n=1}^k e_n\Bigr\|^p.
\end{equation}
Since $N$,  $(d_n)_{n= 1}^N$, and $(\eps_n)_{n= 1}^N$ are general, \eqref{eq:necofUMDgenLHSequivalenceoftransform} implies that $X$ is a UMD Banach space (see the proof of Theorem \ref{thm:necofUMDfordiscBDGtango}).

\subsection{Martingale domination}\label{sec:martdom}

The next theorem shows that under some natural domination assumptions on martingales one gets $L^p$-estimates.

\begin{theorem}\label{thm:applicmartingdomina}
 Let $X$ be a UMD Banach space, $M,N:\mathbb R_+ \times \Omega \to X$ be local martingales such that $\|N_0\| \leq \|M_0\|$ a.s.\ and $[\langle N, x^*\rangle]_{\infty} \leq [\langle M, x^*\rangle]_{\infty}$ a.s.\ for all $x^*\in X^*$. Then for all $1\leq p<\infty$
 \begin{equation}\label{eq:morethanwds12stul'ev}
  \mathbb E \sup_{t\geq 0} \|N_{t}\|^p\lesssim_{p,X} \mathbb E \sup_{t\geq 0} \|M_{t}\|^p.
 \end{equation}
\end{theorem}

 Note that the assumptions in Theorem \ref{thm:applicmartingdomina} are a way more general than the 
{\em weak differential subordination} assumptions (recall that $N$ is weakly differentially subordinate to $M$ if $[\langle M, x^*\rangle] - [\langle N, x^*\rangle]$ is nondecreasing a.s.\ for any $x^*\in X^*$, see \cite{Y17FourUMD,Y17MartDec,OY18}), so Theorem \ref{thm:applicmartingdomina} significantly improves the $L^p$-bounds obtained previously for weakly differentially subordinated martingales in \cite{Y17FourUMD,Y17MartDec} and extends the results to the case $p=1$ as well.

\begin{proof}[Proof of Theorem \ref{thm:applicmartingdomina}]
First notice that by a triangular inequality
$$
\mathbb E \sup_{t\geq 0} \|M_{t}\|^p \eqsim_p \mathbb E \|M_0\|^p + \mathbb E \sup_{t\geq 0} \|M_{t} - M_0\|^p,
$$
$$
\mathbb E \sup_{t\geq 0} \|N_{t}\|^p \eqsim_p \mathbb E \|N_0\|^p + \mathbb E \sup_{t\geq 0} \|N_{t} - N_0\|^p.
$$
Consequently we can reduce the statement to the case $M_0=N_0=0$ a.s.\ (by setting $M := M-M_0$, $N:= N-N_0$), and
then the proof follows directly from Theorem \ref{thm:BDGgeneralUMD} and Lemma \ref{lem:VgeqWthenVgammageqWgamma}.
\end{proof}

\begin{remark}
 It is not known what the sharp constant is in \eqref{eq:morethanwds12stul'ev}. Nevertheless, sharp inequalities of such type have been discovered in the scalar case by Os\k{e}kowski in \cite{Os11b}. It was shown there that if $M$ and $N$ are real-valued $L^p$-bounded martingales such that a.s.
 \[
  [N]_t\leq [M]_t,\;\;\;t\geq 0, \;\;\;\; \text{if}\;1< p\leq 2,
 \]
  \[
  [N]_{\infty} - [N]_{t-}\leq [M]_{\infty} - [M]_{t-},\;\;\;t\geq 0, \;\;\;\; \text{if}\;2 \leq p <\infty,
 \]
 then
 \[
  (\mathbb E |N_{\infty}|^p)^{\frac 1p} \leq (p^*-1)(\mathbb E |M_{\infty}|^p)^{\frac 1p},\;\;\; 1<p<\infty,
 \]
where $p^* := \max\{p, \tfrac{p}{p-1}\}$.
\end{remark}

\subsection{Martingale approximations}

The current subsection is devoted to approximation of martingales. Namely, we will extend the following lemma by Weisz (see \cite[Theorem 6]{Weisz92}) to general UMD Banach space-valued martingales.

 \begin{lemma}\label{lem:approxofMbyL^2M^ninsuosense}
  Let $X$ be a finite dimensional Banach space, $M:\mathbb R_+ \times \Omega \to X$ be a martingale such that $\mathbb E \sup_{t\geq 0} \|M_t\|<\infty$. Then there exists a sequence $(M^n)_{n\geq 1}$ of $X$-valued uniformly bounded martingales such that $\mathbb E \sup_{t\geq 0}\|M_t-M^n_t\| \to 0$ as $n\to \infty$.
 \end{lemma}
 
  Here is the main theorem of the current subsection.

\begin{theorem}\label{thm:WeisztogenUMD}
 Let $X$ be a UMD Banach space, $1\leq p<\infty$, $M:\mathbb R_+ \times \Omega \to X$ be a martingale such that $\mathbb E \sup_{t\geq 0} \|M_t\|^p <\infty$. Then there exists a sequence $(M^n)_{n\geq 1}$ of $X$-valued $L^{\infty}$-bounded martingales such that $\mathbb E \sup_{t\geq 0}\|M_t-M_t^n\|^p \to 0$ as $n\to \infty$.
\end{theorem}

Though this theorem easily follows from Doob's maximal inequality \eqref{eq:DoobsineqXBanach} in the case $p>1$, the case $p=1$ (which is the most important one for the main application of Theorem \ref{thm:WeisztogenUMD}, Theorem \ref{thm:BDGforBFSpgeq1}) remains problematic and requires some work.

For the proof of the theorem we will need to find its analogues for purely discontinuous martingales. Let us first recall some definitions.

A random variable $\tau:\Omega \to \mathbb R_+$ is called an {\em optional stopping time} (or just a {\em stopping time}) if $\{\tau\leq t\} \in \mathcal F_t$ for each $t\geq 0$. With an optional stopping time $\tau$ we associate a $\sigma$-field $\mathcal F_{\tau} = \{A\in \mathcal F_{\infty}: A\cap \{\tau\leq t\}\in \mathcal F_{t}, t\in\mathbb R_+\}$. Note that $M_{\tau}$ is strongly $\mathcal F_{\tau}$-measurable for any local martingale $M$. We refer to \cite[Chapter 7]{Kal} for details.

Recall that due to the existence of a c\`adl\`ag version of a martingale $M:\mathbb R_+ \times \Omega \to X$, we can define an $X$-valued random variables $M_{\tau-}$ and $\Delta M_{\tau}$ for any stopping time $\tau$ in the following way: $M_{\tau-} = \lim_{\eps \to 0}M_{(\tau - \eps)\vee 0}$, $\Delta M_{\tau} = M_{\tau} -  M_{\tau-}$.

A stopping time $\tau$ is called {\em predictable} if there exists a sequence of stopping times $(\tau_n)_{n\geq 1}$ such that $\tau_n<\tau$ a.s.\ on $\{\tau>0\}$ for each $n\geq 1$ and $\tau_n \to \tau$ monotonically a.s.
A stopping time $\tau$ is called {\em totally inaccessible} if $\mathbb P\{\tau = \sigma < \infty\} =0$ for each predictable stopping time $\sigma$.

\begin{definition}
 Let $X$ be a Banach space, $M:\mathbb R_+ \times \Omega \to X$ be a local martingale. Then $M$ is called {\em quasi-left continuous} if $\Delta M_{\tau} = 0$ a.s.\ for any predictable stopping time $\tau$. $M$ is called to have {\em accessible jumps} if $\Delta M_{\tau}=0$ a.s.\ for any totally inaccessible stopping time.
\end{definition}

The reader can find more information on quasi-left continuous martingales and martingales with accessible jumps in \cite{Kal,Y17MartDec,DY17,JS,Y17GMY}. 

In order to prove Theorem \ref{thm:WeisztogenUMD} we will need to show similar approximation results for quasi-left continuous purely discontinuous martingales and purely discontinuous martingales with accessible jumps. Both cases will be considered separately.

\subsubsection{Quasi-left continuous purely discontinuous martingales}

Before stating the corresponding approximation theorem let us show the following proposition.

\begin{proposition}\label{prop:approxofpdqlcmartbyjumpsrukarukakakaka}
 Let $X$ be a Banach space, $M:\mathbb R_+ \times \Omega \to X$ be a purely discontinuous quasi-left continuous martingale. Then there exist sequences of positive numbers $(a_n)_{n\geq 1}$, $(b_n)_{n\geq 1}$, and a sequence of $X$-valued purely discontinuous quasi-left continuous martingales $(M^n)_{n\geq 1}$ such that
 $$
 \sup_t \|\Delta M^n_t\| \leq a_n,\;\;\;\#\{t\geq 0: \Delta M^n_t \neq 0\} \leq b_n \;\; \text{a.s.}\;\; \forall n\geq 1,
 $$
 \begin{equation}\label{eq:approxofpdqlcMnhasthesamejumpsasMm}
    \{t\geq 0: \Delta M^n_t \neq 0\} \subset  \{t\geq 0: \Delta M^m_t \neq 0\} \;\; \text{a.s.}\;\; \forall m\geq n\geq 1,
 \end{equation}
 \begin{equation}\label{eq:forproponapproxofpdqlcmartDelmaM^nsubsetDeltaM}
   \Delta M^n_t = \Delta M_t\;\;\; \forall t\geq 0 \;\; \text{s.t.}\;\; \Delta M^n_t\neq 0 \;\; \text{a.s.}\;\; \forall n\geq 1,
 \end{equation}
 and
 \begin{equation}\label{eq:approxpdqlcMncovesjumpspfM}
  \cup_{n\geq 1}\{t\geq 0: \Delta M^n_t \neq 0\} = \{t\geq 0: \Delta M_t \neq 0\}\;\;\text{a.s.}
 \end{equation}
\end{proposition}

\begin{proof}[Sketch of the proof]
 The construction of such a family of martingales was essentially provided in the proof of \cite[Lemma 5.20]{DY17}. We will recall the construction here for the convenience of the reader. First of all we refer the reader to \cite{DY17,Kal,JS,KalRM} for the basic definitions and facts on {\em random measures}, which presenting we will omit here for the brevity and simplification of the proof. Let $\mu^M$ be a random measure defined on $(\mathbb R_+ \times X, \mathcal B(\mathbb R_+) \otimes \mathcal B(X))$ by
 \[
  \mu^M(A\times B) = \sum_{t\in A} \mathbf 1_{\Delta M_t \in B\setminus\{0\}},\;\;\; A\in \mathcal B(\mathbb R_+), B\in \mathcal B(X).
 \]
Let $\nu^M$ be the corresponding compensator, $\bar{\mu}^M := \mu^M - \nu^M$. Due to the proof of \cite[Lemma 5.20]{DY17} there exists an a.s.\ increasing sequence $(\tau_n)_{n\geq 1}$ of stopping times such $\tau_n \to \infty$ a.s.\ as $n\to \infty$, and such = that there exist positive sequences $(a_n)_{n\geq 1}$, $(b_n)_{n\geq 1}$ with $(a_n)_{n\geq 1}$ being increasing natural and with
\[
 \#\{t\geq 0: \|\Delta M^{\tau_n}_t\| \in [1/a_n, a_n]\} \leq b_n.
\]
Define a predictable set $A_n := [0, \tau_n]\times B_n \subset \mathbb R_+ \times X$, where $B_n := \{x\in X: \|x\|\in [1/a_n, a_n]\}$. Then the desired $M^n$ equals the stochastic integral 
$$
M^n_t := \int_{[0,t]\times X}\mathbf 1_{A_n}(s, x) x \ud \bar{\mu}^M(\dd s, \dd x),\;\;\;\; t\geq 0,
$$
where the latter is a well-defined martingale since by \cite[Subsection 5.4]{DY17} it is sufficient to check that for any $t\geq 0$
\begin{multline*}
  \int_{[0,t]\times X}\|\mathbf 1_{A_n}(s, x) x \|\ud {\mu}^M(\dd s, \dd x) = \int_{A_n \cap[0,t]\times X}\| x \|\ud {\mu}^M(\dd s, \dd x)\\
  = \sum_{t\in[0,\tau_n\wedge t]} \|\Delta M^{\tau_n}_t\| \mathbf 1_{\Delta M^{\tau_n}_t \in [1/a_n, a_n]} \leq a_nb_n<\infty.
\end{multline*}
All the properties of the sequence $(M^n)_{n\geq 1}$ then follow from the construction, namely from the fact that $A_n$ are a.s.\ increasing with $\cup_n A_n = \mathbb R_+ \times X \setminus \{0\}$ a.s., and the fact that $\nu^M$ is non-atomic in time since $M$ is quasi-left continuous (see \cite[Subsection 5.4]{DY17}). 
\end{proof}

In the next theorem we show that the martingales obtained in Proposition \ref{prop:approxofpdqlcmartbyjumpsrukarukakakaka} approximate $M$ in the strong $L^p$-sense.

\begin{theorem}\label{thm:approxofpdqlcmartbywhatyouneed}
 Let $X$ be a UMD Banach space, $M$ be an $X$-valued martingale, $(M^n)_{n\geq 1}$ be a sequence of $X$-valued martingales constructed in Proposition \ref{prop:approxofpdqlcmartbyjumpsrukarukakakaka}. Assume that for some fixed $1\leq p<\infty$, $\mathbb E \sup_{t\geq 0} \|M_t\|^p <\infty$. Then $\mathbb E \sup_{t\geq 0} \|M^n_t\|^p <\infty$ for all $n\geq 1$ and
 \[
  \mathbb E \sup_{t\geq 0} \|M_t- M^n_t\|^p \to 0,\;\; \; n\to \infty.
 \]
\end{theorem}

\begin{proof}
 First of all notice that by Theorem \ref{thm:BDGUMDpdmartingales}, \eqref{eq:forproponapproxofpdqlcmartDelmaM^nsubsetDeltaM}, and \cite[Proposition 6.1.5]{HNVW2} for any $n\geq 1$
 \begin{multline*}
    \mathbb E \sup_{t\geq 0} \|M^n_t\|^p \eqsim_{p, X} \mathbb E \Bigl(\mathbb E_{\gamma}\Bigl\|\sum_{t\geq 0} \gamma_s \Delta M^n_s\Bigr\|^2\Bigr)^{\frac p2}\\
    \leq \mathbb E \Bigl(\mathbb E_{\gamma}\Bigl\|\sum_{t\geq 0} \gamma_s \Delta M_s\Bigr\|^2\Bigr)^{\frac p2} \eqsim_{p, X}\mathbb E \sup_{t\geq 0} \|M_t\|^p.
 \end{multline*}
Let us show the second part of the theorem. Note that by \eqref{eq:forproponapproxofpdqlcmartDelmaM^nsubsetDeltaM} a.s.\ for all $x^*\in X^*$
\[
 [\![M - M^n]\!]_{\infty}(x^*, x^*) = \sum_{t\geq 0} \langle \Delta M_t, x^*\rangle^2 \mathbf 1_{\Delta M_t\neq \Delta M^n_t},
\]
which monotonically vanishes as $n\to \infty$ by \eqref{eq:approxofpdqlcMnhasthesamejumpsasMm} and \eqref{eq:approxpdqlcMncovesjumpspfM}.
Consequently, the desired follows form Theorem \ref{thm:BDGgeneralUMD}, Lemma \ref{lem:V_nmonto0thenthesamefor0}, and the monotone convenience theorem.
\end{proof}

\subsubsection{Purely discontinuous martingales with accessible jumps}

Now let us turn to purely discontinuous martingales with accessible jumps. First notice that by \cite[Proposition 25.4]{Kal}, \cite[Theorem 25.14]{Kal}, and by \cite[Subsection 5.3]{DY17} the following lemmas hold.

\begin{lemma}
 Let $X$ be a Banach space, $M:\mathbb R_+ \times \Omega \to X$ be a local martingale with accessible jumps. Then there exists a set $(\tau_n)_{n\geq 0}$ of predictable stopping times with disjoint graphs (i.e.\ $\tau_n\neq \tau_m$ a.s.\ for all $m\neq n$) such that a.s.\
 \begin{equation}\label{eq:DeltaMareintaunonly}
   \{t\geq 0: \Delta M_{t} \neq 0\} \subset \{\tau_1, \tau_2, \ldots, \tau_n,\ldots\}.
 \end{equation}
\end{lemma}

\begin{lemma}\label{lem:DeltaMtauisamartiftauispredictable}
 Let $X$ be a Banach, $M:\mathbb R_+ \times \Omega \to X$ be an $L^1$-bounded martingale, $\tau$ be a predictable stopping time. Then
 \[
  N_t:= \Delta M_{\tau} \mathbf 1_{[0,t]}(\tau),\;\;\; t\geq 0,
 \]
defines an $L^1$-bounded martingale.
\end{lemma}

Let $X$ be a Banach space, $M:\mathbb R_+ \times \Omega \to X$ be a purely discontinuous martingale with accessible jumps, $(\tau_n)_{n\geq 0}$ be a set of predictable stopping times with disjoint graphs such that \eqref{eq:DeltaMareintaunonly} holds. Thanks to Lemma \ref{lem:DeltaMtauisamartiftauispredictable} for each $n\geq 1$ we can define a martingale
\begin{equation}\label{eq:defofMnforMpdwithaj3topora}
 M^n_t = \sum_{i=1}^n \Delta M_{\tau_i} \mathbf 1_{[0,t]}(\tau_i),\;\;\; t\geq 0.
\end{equation}
{\em Does $(M^n)_{n\geq 1}$ converge to $M$ in strong $L^p$-sense?} The following theorem answers this question in the UMD case.

\begin{theorem}\label{thm:approxofpdmartwajsig}
 Let $X$ be a UMD Banach space, $M:\mathbb R_+ \times \Omega \to X$ be a martingale with accessible jumps, $(M^n)_{n\geq 1}$ be as in \eqref{eq:defofMnforMpdwithaj3topora}. Assume that $\mathbb E \sup_{t\geq 0} \|M_t\|^p <\infty$ for some fixed $1\leq p<\infty$. Then $\mathbb E \sup_{t\geq 0} \|M^n_t\|^p <\infty$ for all $n\geq 1$ and
 \[
  \mathbb E \sup_{t\geq 0} \|M_t- M^n_t\|^p \to 0,\;\; \; n\to \infty.
 \]
\end{theorem}

\begin{proof}
 The proof is fully analogous to the proof of Theorem \ref{thm:approxofpdqlcmartbywhatyouneed}.
\end{proof}

\subsubsection{Proof of Theorem \ref{thm:WeisztogenUMD}}

Let us now prove Theorem \ref{thm:WeisztogenUMD}. Since $X$ is a UMD Banach space, $M$ has the {\em canonical decomposition}, i.e.\ there exist an $X$-valued continuous local martingale $M^c$, an $X$-valued purely discontinuous quasi-left continuous local martingale $M^q$, and an $X$-valued purely discontinuous local martingale $M^a$ with accessible jumps such that $M^c_0=M^q_0=0$ and $M=M^c+M^q+M^a$ (see Subsection \ref{subsec:candec} for details). Moreover, by \eqref{eq:camdecLpbounforpgeq1} and a triangle inequality
\begin{equation*}
 \mathbb E \sup_{t\geq 0} \Bigl(\|M^c_t\|^p + \|M^q_t\|^p + \|M^a_t\|^p\Bigr) \eqsim_{p,X} \mathbb E \sup_{t\geq 0} \|M_t\|^p,
\end{equation*}
so it is sufficient to show Theorem \ref{thm:WeisztogenUMD} for each of these three cases separately. By \cite[Theorem 1.3.2 and 3.3.16]{HNVW1} $M$ converges a.s., so we can assume that there exists $T>0$ such that $M_t = M_T$ a.s.\ for all $t\geq T$.

{\em Case 1: $M$ is continuous.} The theorem follows from the fact that every continuous martingale is locally bounded and the fact that $M_t= M_T$ for all $t\geq T$.

{\em Case 2: $M$ is purely discontinuous quasi-left continuous.} By Theorem \ref{thm:approxofpdqlcmartbywhatyouneed} one can assume that $M$ has uniformly bounded jumps. Then the theorem follows from the fact that any adapted c\`adl\`ag process with uniformly bounded jumps is local uniformly bounded and the fact that $M_t= M_T$ for all $t\geq T$.
 
{\em Case 3: $M$ is purely discontinuous with accessible jumps.} By Theorem \ref{thm:approxofpdmartwajsig} we can assume that there exist predictable stopping times $(\tau_n)_{n=1}^N$ with disjoint graphs such that
\[
 M_t = \sum_{n=1}^N \Delta M_{\tau_n} \mathbf 1_{[0,t]}(\tau_n),\;\;\; t\geq 0.
\]
Fix $\eps>0$.
Without loss of generality we may assume that the stopping times $(\tau_n)_{n=1}^N$ are bounded a.s.
Due to \cite[Subsection 5.3]{DY17} we may additionally assume that $(\tau_n)_{n=1}^N$ is a.s.\ increasing. 
Then by \cite[Subsection 5.3]{DY17} (or \cite[Lemma 26.18]{Kal} in the real-valued case) the sequence $(0,\Delta M_{\tau_1}, 0, \Delta M_{\tau_2},\ldots, 0,\Delta M_{\tau_N})$ is a martingale difference sequence with respect to the filtration 
$$
\mathbb G := (\mathcal F_{\tau_1-}, \mathcal F_{\tau_1},\mathcal F_{\tau_2-}, \mathcal F_{\tau_2},\ldots, \mathcal F_{\tau_N-}, \mathcal F_{\tau_N})
$$ 
(see \cite[Lemma 25.2]{Kal} for the definition of $\mathcal F_{\tau-}$). As any discrete $L^p$-bounded martingale difference sequence, $(0,\Delta M_{\tau_1}, 0, \Delta M_{\tau_2},\ldots, 0,\Delta M_{\tau_N})$  can be approximated in a strong $L^p$-sense by a uniformly bounded $X$-valued $\mathbb G$-martingale difference sequence $(0, d_1^{\eps}, 0, d_2^{\eps}, \ldots, 0, d_N^{\eps})$ such that 
\[
\mathbb E\sup_{n=1}^N \Bigl\| \sum_{i=1}^n \Delta M_{\tau_i} - d_i^{\eps}\Bigr\|^p<\eps.
\]
The martingale difference sequence $(0, d_1^{\eps}, 0, d_2^{\eps}, \ldots, 0, d_N^{\eps})$ can be translated back to a martingale on $\mathbb R_+$ in the same way as it was shown in \cite[Subsection 5.3]{DY17}, i.e.\ one can define a process $N^{\eps}:\mathbb R_+ \times \Omega \to X$ such that
\[
 N^{\eps}_t := \sum_{n=1}^N d_n \mathbf 1_{[0,t]}(\tau_n),\;\;\; t\geq 0,
\]
which is a martingale by \cite[Subsection 5.3]{DY17} (or see \cite[Lemma 26.18]{Kal} for the real valued version) with
\[
 \mathbb E \sup_{t\geq 0} \|M_t - N^{\eps}_t\|^p  = \mathbb E \sup_{t\geq 0}\Bigl\|\sum_{0\leq s\leq t} \Delta M_s - \Delta N^{\eps}_s\Bigr\|^p = \mathbb E\sup_{n=1}^N \Bigl\| \sum_{i=1}^n \Delta M_{\tau_i} - d_i\Bigr\|^p<\eps,
\]
which terminates the proof.

\begin{remark}
 Clearly Theorem \ref{thm:WeisztogenUMD} holds true if $X$ has a Schauder basis. Therefore it remain open for whether Theorem \ref{thm:WeisztogenUMD} holds true for a general Banach space.
\end{remark}

\subsection{The canonical decomposition}\label{subsec:candec}

Let $X$ be a Banach space. Then $X$ has the UMD property if and only if any $X$-valued local martingale $M$ has the so-called {\em canonical decomposition}, i.e.\ there exist an $X$-valued continuous local martingale $M^c$ (a Wiener-like part), an $X$-valued purely discontinuous quasi-left continuous local martingale $M^q$ (a compensated Poisson-like part), and an $X$-valued purely discontinuous local martingale $M^a$ with accessible jumps (a discrete-like part) such that $M^c_0=M^q_0=0$ and $M=M^c+M^q+M^a$. We refer the reader to \cite{Kal,DY17,Y17GMY,Y17MartDec} for the details on the canonical decomposition. 

As it was shown in \cite{DY17,Y17GMY,Y17MartDec}, the canonical decomposition is unique, and by \cite[Section 3]{Y17MartDec} together with \eqref{eq:DoobsineqXBanach} we have that for any $1<p<\infty$ and for any $i=c,q,a$
\begin{equation}\label{eq:camdecLpbounforp>1}
 \mathbb E \sup_{t\geq 0} \|M^i_t\|^p \lesssim_{p,X} \mathbb E \sup_{t\geq 0} \|M_t\|^p.
\end{equation}
Theorem \ref{thm:applicmartingdomina} allows us to extend \eqref{eq:camdecLpbounforp>1} to the case $p=1$. Indeed, it is known due to \cite{Y17GMY,Y17MartDec} that for any $x^*\in X^*$ a.s.
\[
[\langle M, x^*\rangle]_t =  [\langle M^c, x^*\rangle]_t + [\langle M^q, x^*\rangle]_t + [\langle M^a, x^*\rangle]_t,\;\;\;\; t\geq 0,
\]
so by Theorem \ref{thm:applicmartingdomina}
\begin{equation}\label{eq:camdecLpbounforpgeq1}
 \mathbb E \sup_{t\geq 0} \|M^i_t\|^p \lesssim_{p,X} \mathbb E \sup_{t\geq 0} \|M_t\|^p,
\end{equation}
for all $1\leq p<\infty$ and any $i=c,q,a$.

\subsection{Covariation bilinear forms for pairs of martingales}

Let $X$ be a UMD Banach space, $M, N:\mathbb R_+ \times \Omega\to X$ be local martingales. Then for any fixed $t\geq 0$ and any $x^*, y^*\in X^*$ we have that by \cite[Theorem 26.6(iii)]{Kal} a.s.\
$$
[\langle M, x^*\rangle, \langle N, y^*\rangle]_t \leq [\![M]\!]_t(x^*, x^*)[\![N]\!]_t(y^*, y^*).
$$
Thus analogously the proof of Theorem \ref{thm:BDGgeneralUMD} (by exploiting a subspace $Y$ of $X^*$ that is a linear span of a countable subset of $X^*$) there exists a bounded bilinear form-valued random variable $[\![M, N]\!]_t:\Omega\to X\otimes X$ such that $[\langle M, x^*\rangle, \langle N, y^*\rangle]_t =[\![M, N]\!]_t(x^*, y^*) $ for any $x^*, y^*\in X^*$ a.s.

Now let $X$ and $Y$ be UMD Banach spaces (perhaps different), $M:\mathbb R_+ \times \Omega \to X$, $N:\mathbb R_+ \times \Omega \to Y$ be local martingales. Then we can show that for any $t\geq 0$ there exists a bilinear form-valued process $[\![M,N]\!]_t:\Omega\to X \otimes Y$ such that $[\![M,N]\!]_t = [\langle M, x^*\rangle, \langle N, y^*\rangle]_t$ a.s.\ for any $x^*\in X^*$ and $y^*\in Y^*$. Indeed, one can presume the Banach space to be $X\times Y$ and extend both $M$ and $N$ to take values in this Banach space. Then by the first part of the present subsection there exists a bilinear form $[\![M,N]\!]_t$ acting on $(X \times Y)^* \times (X \times Y)^*$ such that for any $x^*\in X^*$ and $y^*\in Y^*$ a.s.\
\begin{equation}\label{eq:covbilineatformdiffXY}
 \begin{split}
  [\![M,N]\!]_t\bigl((x^*, y^*), (x^*, y^*)\bigr) &= [\langle M, (x^*, y^*)\rangle, \langle N, (x^*, y^*)\rangle]_t\\
  &=[\langle M, x^*\rangle, \langle N, y^*\rangle]_t.
 \end{split}
\end{equation}
It remains to restrict $[\![M,N]\!]_t$ back to $X\otimes Y$ from $(X\times Y)\otimes (X\times Y)$ which is possible by \eqref{eq:covbilineatformdiffXY}.

Interesting things happen given $Y=\mathbb R$. In this case $[\![M, N]\!]_t$ takes values in $X\otimes \mathbb R \simeq X$, so $[\![M, N]\!]_t$ is simply $X$-valued, and it is easy to see that
\begin{equation}\label{eq:[[M,N]]givenX=XY=R}
  [\![M, N]\!]_t = \mathbb P-\lim_{\text{mesh}\to 0}\sum_{i=1}^n (M(t_n)-M(t_{n-1}))(N(t_n)-N(t_{n-1})),
\end{equation}
where the limit in probability is taken over partitions $0= t_0 < \ldots < t_n = t$, and it is taken in a {\em weak} sense (i.e.\ \eqref{eq:[[M,N]]givenX=XY=R} holds under action of any linear functional $x^*\in X^*$). It remains open whether \eqref{eq:[[M,N]]givenX=XY=R} holds in a strong sense.

\section{UMD Banach function spaces}\label{subsec:UMDBanachfs}

Here we are going to extend \eqref{eq:BDGBFSINTRO} to the case $p=1$.
Let us first recall some basic definitions on Banach function spaces. 
For a given measure space $(S,\Sigma,\mu)$, the linear space of all real-valued measurable functions is denoted by $L^0(S)$. We endow $L^0(S)$ with the local convergence in measure topology.

\begin{definition}\label{def:Bfs}
Let $(S,\Sigma,\mu)$ be a measure space. Let $n:L^0(S)\to [0,\infty]$ be a function which satisfies the following properties:
  \begin{enumerate}[(i)]
\item $n(x) = 0$ if and only if $x=0$,
\item for all $x,y\in L^0(S)$ and $\lambda\in \R$, $n(\lambda x) = |\lambda| n(x)$ and $n(x+y)\leq n(x)+n(y)$,
    \item if $x \in L^0(S), y \in L^0(S)$, and $|x| \leq |y|$, then $n(x) \leq n(y)$,
    \item there exists $\zeta \in L^0(S)$ with $\zeta>0$ and $n(\zeta)<\infty$,
    \item if $0 \leq x_n \uparrow x$ with $(x_n)_{n=1}^\infty$ a sequence in $L^0(S)$ and $x \in L^0(S)$, then  $n(x) = \sup_{n \in \N}n(x_n)$.
  \end{enumerate}
Let $X$ denote the space of all $x\in L^0(S)$ for which $\|x\|:=n(x)<\infty$. Then $X$ is called a {\em normed function space associated} to $n$. It is called a {\em Banach function space} when $(X,\|\cdot\|_X)$ is complete. We will additionally assume the following natural property of $X$:
\begin{enumerate}[(i)]
 \item[(vi)] $X$ is continuously embedded into $L^0(S)$ with the local convergence in measure topology.
\end{enumerate}
\end{definition}

Notice that the condition $(vi)$ holds automatically if one changes the measure on $(S, \Sigma)$ in an appropriate way (see \cite[Theorem 1.b.14]{LTCBsII}). 
We refer the reader to \cite{VY18,Rubio86,M-N91,LTCBsII,Zaa} for the details on Banach function spaces.

Given a Banach function space $X$ over a measure space $S$ and Banach space $E$, let $X(E)$ denote the space of all strongly measurable functions $f:S\to E$ with $\|f\|_E \in X$. The space $X(E)$ becomes a Banach space when equipped with the norm $\|f\|_{X(E)} = \big\|\sigma\mapsto \|f(\sigma)\|_E\big\|_X$.

Let $X$ be a UMD Banach function space over a $\sigma$-finite measure space $(S, \Sigma,\mu)$. According to \cite{VY18} any $X$-valued $L^p$-bounded martingale $M$, $1<p<\infty$, has a pointwise martingale version, i.e.\ there exists a process $N:\mathbb R_+ \times \Omega \times S \to \mathbb R$ such that
\begin{enumerate}[(i)]
 \item $N|_{[0,t]\times \Omega \times S}$ is $\mathcal B([0, t])\otimes \mathcal F_t \otimes \Sigma$-measurable for all $t\geq 0$,
 \item $N(\cdot, \cdot, \sigma)$ is a local martingale for a.e.\ $\sigma\in S$,
 \item $N(\omega, t,\cdot) = M_t(\omega)$ for any $t\geq 0$ for a.a.\ $\omega \in \Omega$.
\end{enumerate}
A process $N$ satisfying (i) and (ii) is called a {\em local martingale field}.
Moreover, it was shown in \cite{VY18} that for any $1<p<\infty$
\begin{equation}\label{eq:BDGforBFSp>1}
 \mathbb E \sup_{t\geq 0} \|M_t\|^p \eqsim_{p, X} \mathbb E \|[N]_{\infty}^{1/2}\|^p,
\end{equation}
where $\sigma\mapsto [N(\cdot, \cdot, \sigma)]_{\infty}^{1/2}$, $\sigma\in S$, defines an element of $X$ a.s.
The goal of the present subsection is to show that \eqref{eq:BDGforBFSp>1} holds for $p=1$.

\begin{theorem}\label{thm:BDGforBFSpgeq1}
 Let $X$ be a UMD Banach function space over a $\sigma$-finite measure space $(S, \Sigma,\mu)$, $M:\mathbb R_+ \times \Omega \to X$ be a local martingale. Then there exists a local martingale field $N:\mathbb R_+ \times \Omega \times S \to \mathbb R$ such that $N(\omega, t,\cdot) = M_t(\omega)$ for all $t\geq 0$ for a.a.\ $\omega \in \Omega$, and for all $1\leq p<\infty$
 \begin{equation}\label{eq:BDGforBFSpgeq1}
 \mathbb E \sup_{t\geq 0} \|M_t\|^p \eqsim_{p, X} \mathbb E \|[N]_{\infty}^{1/2}\|^p,
\end{equation}
\end{theorem}

Let us first show the discrete version of Theorem \ref{thm:BDGforBFSpgeq1}, which was shown in \cite[Theorem 3]{Rubio86} for the case $p\in (1,\infty)$.

\begin{proposition}\label{prop:BDGforBFSdiscretecase}
 Let $X$ be a UMD Banach function space over a measure space $(S, \Sigma,\mu)$, $(d_n)_{n\geq 1}$ be an $X$-valued martingale difference sequence. Then for all $1\leq p<\infty$
 \[
  \mathbb E \sup_{N\geq 1} \Bigl\| \sum_{n=1}^N d_n \Bigr\|^p \eqsim_{p, X} \mathbb E \Bigl\|\Bigl(\sum_{n=1}^{\infty}|d_n|^2\Bigr)^{\frac 12}\Bigr\|^{p}.
 \]
\end{proposition}

\begin{proof}
 The proof follows from Theorem \ref{thm:BDGgendiscmart} and the equivalence \cite[(9.26)]{HNVW2} between the $\gamma$-norm and the square function.
\end{proof}

\begin{remark}\label{rem:controlofconstantsBDGforBFSdis}
 By Remark \ref{rem:c_ppXandC_p'Xareindepofp} and \cite[(9.26)]{HNVW2} one has that for any $r\in (1,\infty)$  there exist positive $C_{r, X}$ and $c_{r,X}$ such that for any $1\leq p\leq r$
 \[
   c_{r, X} \mathbb E \Bigl\|\Bigl(\sum_{n=1}^{\infty}|d_n|^2\Bigr)^{\frac 12}\Bigr\|^{p} \leq \mathbb E \sup_{N\geq 1}\Bigl\|\sum_{n=1}^{N} d_n\Bigr\|^p \leq C_{r, X} \mathbb E \Bigl\|\Bigl(\sum_{n=1}^{\infty}|d_n|^2\Bigr)^{\frac 12}\Bigr\|^{p}.
 \]
\end{remark}

We will also need the following technical lemma proved in \cite[Section 4]{VY18}. Recall that $\mathcal D_b([0,\infty);X)$ is the Banach space of all bounded $X$-valued c\`adl\`ag functions on $\mathbb R_+$, which is also known as a {\em Skorohod space}.

\begin{lemma}\label{lem:MQ1(X)isBanach}
 Let $X$ be a Banach function space over a $\sigma$-finite measure space $(S, \Sigma,\mu)$. Let
\begin{multline*}
 \textnormal{MQ}^1(X) := \{N:\R_+\times \Omega\times S\to\R: N \ \text{is a local martingale field,}\\
 N_0(\sigma) = 0\; \forall \sigma\in S, \  \text{and} \ \|N\|_{\textnormal{MQ}^1(X)} <\infty\}, 
\end{multline*}
where
\begin{equation}\label{eq:norminMQ1(X)}
 \|N\|_{\textnormal{MQ}^1(X)} := \|[N]_{\infty}^{1/2}\|_{L^1(\Omega;X)}.
\end{equation}
Then $(\textnormal{MQ}^1(X), \|\cdot\|_{\textnormal{MQ}^1(X)})$ is a Banach space. 
Moreover, if $N^n\to N$ in $\textnormal{MQ}^1$, then there exists a subsequence $(N^{n_k})_{k\geq 1}$ such that pointwise a.e.\ in $S$, we have $N^{n_k}\to N$ in $L^1(\Omega;\mathcal D_b([0,\infty)))$.
\end{lemma}

\begin{proof}[Proof of Theorem \ref{thm:BDGforBFSpgeq1}] We will consider separately the cases $p>1$ and $p=1$.

{\em Case $p>1$.} This case was covered in \cite{VY18}. Nevertheless, we wish to notice that by modifying the proof from \cite{VY18} by using Proposition \ref{prop:BDGforBFSdiscretecase} one can obtain better behavior of the equivalence constants in \eqref{eq:BDGforBFSpgeq1}. Namely, by exploiting the same proof together with  Proposition \ref{prop:BDGforBFSdiscretecase} and Remark \ref{rem:controlofconstantsBDGforBFSdis} one obtains that for any $p'\in (1,\infty)$  there exist positive $C_{p', X}$ and $c_{p',X}$ (the same as in Remark \ref{rem:controlofconstantsBDGforBFSdis}) such that for any $1< p\leq p'$
\begin{equation}\label{eq:controlofconstantsBDGforBFScont}
  c_{p', X} \mathbb E \|[N]_{\infty}^{1/2}\|^p \leq \mathbb E \sup_{t\geq 0} \|M_t\|^p \leq C_{p', X} \mathbb E \|[N]_{\infty}^{1/2}\|^p.
\end{equation}

{\em Case $p=1$.} By Theorem \ref{thm:WeisztogenUMD} there exists a sequence $(M^n)_{n\geq 1}$ of uniformly bounded $X$-valued martingales such that 
\begin{equation}\label{eq:proofofthmDGforBFSpgeq1MntoM}
 \mathbb E \sup_{t\geq 0} \|M_t-M^n_t\| \to 0,\;\;\; n\to \infty.
\end{equation}
Since $M^n$ is uniformly bounded for any $n\geq 1$, $\mathbb E \sup_{t\geq 0} \|M^n_t\|^2 <\infty$, so by Case $p>1$ there exists a local martingale field $N^n$ such that $N^n(\omega, t,\cdot) = M^n_t(\omega)$ for all $t\geq 0$ for a.a.\ $\omega \in \Omega$. By \eqref{eq:controlofconstantsBDGforBFScont} one has that there exist positive constants $C_X$ and $c_X$ such that for all $m, n\geq 1$
\[
  c_{X} \mathbb E \|[N^n-N^m]_{\infty}^{1/2}\| \leq \mathbb E \sup_{t\geq 0} \|M^n_t- M^m_t\| \leq C_{X} \mathbb E \|[N^n-N^m]_{\infty}^{1/2}\|,
\]
hence due to \eqref{eq:proofofthmDGforBFSpgeq1MntoM} $(N^n)_{n\geq 1}$ is a Cauchy sequence in $\textnormal{MQ}^1(X)$. Since by Lemma \ref{lem:MQ1(X)isBanach} the linear space $\textnormal{MQ}^1(X)$ endowed with the norm \eqref{eq:norminMQ1(X)} is Banach, there exists a limit $N$ of $(N^n)_{n\geq 1}$ in $\textnormal{MQ}^1(X)$.

Let us show that $N$ is the desired local martingale field. Fix $t\geq 0$. We need to who that $N(\cdot,t,\cdot) = M_t$ a.s.\ on $\Omega$. First notice that by the last part of Lemma \ref{lem:MQ1(X)isBanach} there exists a subsequence of $(N^n)_{n\geq 1}$ which we will denote by $(N^n)_{n\geq 1}$ as well such that $N^n(\cdot, t, \sigma) \to N(\cdot, t, \sigma)$ in $L^1(\Omega)$ for a.e.\ $\sigma\in S$. On the other hand by Jensen's inequality
\[
 \bigl\|\mathbb E |N^n(\cdot,t,\cdot) - M_t|\bigr\| =  \bigl\|\mathbb E |M^n_t - M_t|\bigr\| \leq \mathbb E \|M^n_t- M_t\| \to 0,\;\;\;\; n\to \infty.
\]
Hence $N^n(\cdot,t,\cdot)\to M_t$ in $X(L^1(\Omega))$, and thus by Definition \ref{def:Bfs}$(vi)$ in $L^0(S;L^1(\Omega))$. Therefore we can find a subsequence of $(N^n)_{n\geq 1}$ (which we will again denote by $(N^n)_{n\geq 1}$) such that $N^n(\cdot,t,\sigma)\to M_t(\sigma)$ in $L^1(\Omega)$ for a.e.\ $\sigma\in S$ (here we use that fact that $\mu$ is $\sigma$-finite), so $N(\cdot, t, \cdot) = M_t$ a.s.\ on $\Omega \times S$, and consequently by Definition \ref{def:Bfs}$(iii)$, $N(\omega, t, \cdot) = M_t(\omega)$ for a.a.\ $\omega \in \Omega$. 

Let us finally show \eqref{eq:BDGforBFSpgeq1}. Since $N^n \to N$ in $\textnormal{MQ}^1(X)$ and by \eqref{eq:proofofthmDGforBFSpgeq1MntoM}
\[
  \mathbb E \|[N]_{\infty}^{1/2}\| = \lim_{n\to \infty} \mathbb E \|[N^n]_{\infty}^{1/2}\| \eqsim_X \lim_{n\to \infty}\mathbb E \sup_{t\geq 0} \|M^n_t\| = \mathbb E \sup_{t\geq 0} \|M_t\|,
\]
which terminates the proof.
\end{proof}

\begin{remark}
 It was shown in \cite{VY18} that in the case $p>1$ the equivalence \eqref{eq:BDGforBFSpgeq1} can be strengthen. Namely, in this case one can show that
 \begin{equation}\label{eq:RdFsupinside}
  \mathbb E  \bigl\|\sup_{t\geq 0} |M_t|\bigl\|^p \eqsim_{p, X} \mathbb E \|[N]_{\infty}^{1/2}\|^p,
 \end{equation}
 i.e.\ one has the same equivalence with a pointwise supremum in $S$.
The techniques that provide such an improvement were discovered by Rubio de Francia in \cite{Rubio86}. Unfortunately, it remains open whether \eqref{eq:RdFsupinside} holds for $p=1$. Surprisingly, \eqref{eq:RdFsupinside} holds for $p=1$ and for $X=L^1(S)$ by a simple Fubini-type argument, so it might be that \eqref{eq:RdFsupinside} holds for $p=1$ even for other nonreflexive Banach spaces.
\end{remark}

\section*{Acknowledgment}

The author would like to thank Mark Veraar and Jan van Neerven for inspiring conversations, careful reading of the paper, and useful suggestions. The author thanks Adam Os\k{e}kowski for  discussing inequality \eqref{eq:morethanwds12stul'ev}. We also thank Stanis\l{}aw Kwapie\'n for his hint on Proposition \ref{prop:Gaustype2=1=Hilbert} and Emiel Lorist for his helpful comments on Banach function spaces (see Section \ref{subsec:UMDBanachfs}). The author thanks Ben Goldys, Carlo Marinelli, and the anonymous referees for their useful suggestions.

\bibliographystyle{plain}

\begin{thebibliography}{10}

\bibitem{BChD11}
H.~Bahouri, J.-Y. Chemin, and R.~Danchin.
\newblock {\em Fourier analysis and nonlinear partial differential equations},
  volume 343 of {\em Grundlehren der Mathematischen Wissenschaften [Fundamental
  Principles of Mathematical Sciences]}.
\newblock Springer, Heidelberg, 2011.

\bibitem{BDMKR}
T.~Bj\"{o}rk, G.~Di~Masi, Y.~Kabanov, and W.~Runggaldier.
\newblock Towards a general theory of bond markets.
\newblock {\em Finance Stoch.}, 1(2):141--174, 1997.

\bibitem{BogGaus}
V.I. Bogachev.
\newblock {\em Gaussian measures}, volume~62 of {\em Mathematical Surveys and
  Monographs}.
\newblock American Mathematical Society, Providence, RI, 1998.

\bibitem{Bour83}
J.~Bourgain.
\newblock Some remarks on {B}anach spaces in which martingale difference
  sequences are unconditional.
\newblock {\em Ark. Mat.}, 21(2):163--168, 1983.

\bibitem{BN}
Z.~Brze{\'z}niak and J.M.A.M. van Neerven.
\newblock Stochastic convolution in separable {B}anach spaces and the
  stochastic linear {C}auchy problem.
\newblock {\em Studia Math.}, 143(1):43--74, 2000.

\bibitem{Bur73}
D.L. Burkholder.
\newblock Distribution function inequalities for martingales.
\newblock {\em Ann. Probability}, 1:19--42, 1973.

\bibitem{Burk81}
D.L. Burkholder.
\newblock A geometrical characterization of {B}anach spaces in which martingale
  difference sequences are unconditional.
\newblock {\em Ann. Probab.}, 9(6):997--1011, 1981.

\bibitem{Burk86}
D.L. Burkholder.
\newblock Martingales and {F}ourier analysis in {B}anach spaces.
\newblock In {\em Probability and analysis ({V}arenna, 1985)}, volume 1206 of
  {\em Lecture Notes in Math.}, pages 61--108. Springer, Berlin, 1986.

\bibitem{Burk87}
D.L. Burkholder.
\newblock Sharp inequalities for martingales and stochastic integrals.
\newblock {\em Ast\'erisque}, (157-158):75--94, 1988.
\newblock Colloque Paul L\'evy sur les Processus Stochastiques (Palaiseau,
  1987).

\bibitem{Burk91}
D.L. Burkholder.
\newblock Explorations in martingale theory and its applications.
\newblock In {\em \'Ecole d'\'Et\'e de {P}robabilit\'es de {S}aint-{F}lour
  {XIX}---1989}, volume 1464 of {\em Lecture Notes in Math.}, pages 1--66.
  Springer, Berlin, 1991.

\bibitem{Burk01}
D.L. Burkholder.
\newblock Martingales and singular integrals in {B}anach spaces.
\newblock In {\em Handbook of the geometry of {B}anach spaces, {V}ol. {I}},
  pages 233--269. North-Holland, Amsterdam, 2001.

\bibitem{BDG}
D.L. Burkholder, B.J. Davis, and R.F. Gundy.
\newblock Integral inequalities for convex functions of operators on
  martingales.
\newblock In {\em Proceedings of the {S}ixth {B}erkeley {S}ymposium on
  {M}athematical {S}tatistics and {P}robability ({U}niv. {C}alifornia,
  {B}erkeley, {C}alif., 1970/1971), {V}ol. {II}: {P}robability theory}, pages
  223--240. Univ. California Press, Berkeley, Calif., 1972.

\bibitem{CG}
S.G. Cox and S.~Geiss.
\newblock On decoupling in {B}anach spaces.
\newblock {\em arXiv:1805.12377}, 2018.

\bibitem{DPZ}
G.~Da~Prato and J.~Zabczyk.
\newblock {\em Stochastic equations in infinite dimensions}, volume 152 of {\em
  Encyclopedia of Mathematics and its Applications}.
\newblock Cambridge University Press, Cambridge, second edition, 2014.

\bibitem{Day47}
M.M. Day.
\newblock Some characterizations of inner-product spaces.
\newblock {\em Trans. Amer. Math. Soc.}, 62:320--337, 1947.

\bibitem{dlPG}
V.H. de~la Pe\~{n}a and E.~Gin\'{e}.
\newblock {\em Decoupling}.
\newblock Probability and its Applications (New York). Springer-Verlag, New
  York, 1999.
\newblock From dependence to independence, Randomly stopped processes.
  $U$-statistics and processes. Martingales and beyond.

\bibitem{Dirk14}
S.~Dirksen.
\newblock It\^o isomorphisms for {$L^p$}-valued {P}oisson stochastic integrals.
\newblock {\em Ann. Probab.}, 42(6):2595--2643, 2014.

\bibitem{DMY18}
S.~Dirksen, C.~Marinelli, and I.S. Yaroslavtsev.
\newblock Stochastic evolution equations in ${L}^p$-spaces driven by jump
  noise.
\newblock {\em In preparation}.

\bibitem{DY17}
S.~Dirksen and I.S. Yaroslavtsev.
\newblock ${L}^{q}$-valued {B}urkholder-{R}osenthal inequalities and sharp
  estimates for stochastic integrals.
\newblock {\em Proc. Lond. Math. Soc. (3)}, 119(6):1633--1693, 2019.

\bibitem{Dol69}
C.~Dol\'eans.
\newblock Variation quadratique des martingales continues \`a droite.
\newblock {\em Ann. Math. Statist}, 40:284--289, 1969.

\bibitem{Fa63}
E.F. Fama.
\newblock Mandelbrot and the stable {P}aretian hypothesis.
\newblock {\em The journal of business}, 36(4):420--429, 1963.

\bibitem{Gar85}
D.J.H. Garling.
\newblock Brownian motion and {UMD}-spaces.
\newblock In {\em Probability and {B}anach spaces ({Z}aragoza, 1985)}, volume
  1221 of {\em Lecture Notes in Math.}, pages 36--49. Springer, Berlin, 1986.

\bibitem{Gar90}
D.J.H. Garling.
\newblock Random martingale transform inequalities.
\newblock In {\em Probability in {B}anach spaces 6 ({S}andbjerg, 1986)},
  volume~20 of {\em Progr. Probab.}, pages 101--119. Birkh\"auser Boston,
  Boston, MA, 1990.

\bibitem{Geiss99}
S.~Geiss.
\newblock A counterexample concerning the relation between decoupling constants
  and {UMD}-constants.
\newblock {\em Trans. Amer. Math. Soc.}, 351(4):1355--1375, 1999.

\bibitem{GM-SS}
S.~Geiss, S.~Montgomery-Smith, and E.~Saksman.
\newblock On singular integral and martingale transforms.
\newblock {\em Trans. Amer. Math. Soc.}, 362(2):553--575, 2010.

\bibitem{GY19}
S.~Geiss and I.S. Yaroslavtsev.
\newblock Dyadic and stochastic shifts and {V}olterra-type operators.
\newblock {\em In preparation}.

\bibitem{GvN}
B.~Goldys and J.M.A.M. van Neerven.
\newblock Transition semigroups of {B}anach space-valued {O}rnstein-{U}hlenbeck
  processes.
\newblock {\em Acta Appl. Math.}, 76(3):283--330, 2003.

\bibitem{GP74}
J.B. Gravereaux and J.~Pellaumail.
\newblock Formule de {I}to pour des processus non continus \`a valeurs dans des
  espaces de {B}anach.
\newblock {\em Ann. Inst. H. Poincar\'{e} Sect. B (N.S.)}, 10:399--422 (1975),
  1974.

\bibitem{Grig71}
B.~Grigelionis.
\newblock The representation of integer-valued random measures as stochastic
  integrals over the {P}oisson measure.
\newblock {\em Litovsk. Mat. Sb.}, 11:93--108, 1971.

\bibitem{GK1}
I.~Gy{\"o}ngy and N.V. Krylov.
\newblock On stochastic equations with respect to semimartingales. {I}.
\newblock {\em Stochastics}, 4(1):1--21, 1980/81.

\bibitem{HPZh95}
T.~Halpin-Healy and Y.-Ch. Zhang.
\newblock Kinetic roughening phenomena, stochastic growth, directed polymers
  and all that. {A}spects of multidisciplinary statistical mechanics.
\newblock {\em Physics reports}, 254(4-6):215--414, 1995.

\bibitem{HNVW1}
T.P. Hyt\"{o}nen, J.M.A.M. van Neerven, M.C. Veraar, and L.~Weis.
\newblock {\em Analysis in {B}anach spaces. {V}ol. {I}. {M}artingales and
  {L}ittlewood-{P}aley theory}, volume~63 of {\em Ergebnisse der Mathematik und
  ihrer Grenzgebiete.}
\newblock Springer, 2016.

\bibitem{HNVW2}
T.P. Hyt\"{o}nen, J.M.A.M. van Neerven, M.C. Veraar, and L.~Weis.
\newblock {\em Analysis in {B}anach spaces. {V}ol. {II}. {P}robabilistic
  methods and operator theory}, volume~67 of {\em Ergebnisse der Mathematik und
  ihrer Grenzgebiete.}
\newblock Springer, Cham, 2017.

\bibitem{HNVW3}
T.P. Hyt\"{o}nen, J.M.A.M. van Neerven, M.C. Veraar, and L.~Weis.
\newblock {\em Analysis in {B}anach spaces. {V}ol. {III}. {H}armonic and
  Spectral Theory}.
\newblock Ergebnisse der Mathematik und ihrer Grenzgebiete. Springer, in
  preparation.

\bibitem{JS}
J.~Jacod and A.N. Shiryaev.
\newblock {\em Limit theorems for stochastic processes}, volume 288 of {\em
  Grundlehren der Mathematischen Wissenschaften}.
\newblock Springer-Verlag, Berlin, second edition, 2003.

\bibitem{Kak70}
S.C. Kak.
\newblock The discrete {H}ilbert transform.
\newblock {\em Proceedings of the IEEE}, 58(4):585--586, 1970.

\bibitem{Kal}
O.~Kallenberg.
\newblock {\em Foundations of modern probability}.
\newblock Probability and its Applications (New York). Springer-Verlag, New
  York, second edition, 2002.

\bibitem{KalRM}
O.~Kallenberg.
\newblock {\em Random measures, theory and applications}, volume~77 of {\em
  Probability Theory and Stochastic Modelling}.
\newblock Springer, Cham, 2017.

\bibitem{KalS91}
O.~Kallenberg and R.~Sztencel.
\newblock Some dimension-free features of vector-valued martingales.
\newblock {\em Probab. Theory Related Fields}, 88(2):215--247, 1991.

\bibitem{KS}
I.~Karatzas and S.E. Shreve.
\newblock {\em Brownian motion and stochastic calculus}, volume 113 of {\em
  Graduate Texts in Mathematics}.
\newblock Springer-Verlag, New York, second edition, 1991.

\bibitem{KL04}
V.~Keyantuo and C.~Lizama.
\newblock Fourier multipliers and integro-differential equations in {B}anach
  spaces.
\newblock {\em J. London Math. Soc. (2)}, 69(3):737--750, 2004.

\bibitem{KingPois}
J.F.C. Kingman.
\newblock {\em Poisson processes}, volume~3 of {\em Oxford Studies in
  Probability}.
\newblock The Clarendon Press, Oxford University Press, New York, 1993.
\newblock Oxford Science Publications.

\bibitem{Kr97}
N.V. Krylov.
\newblock On {SPDE}'s and superdiffusions.
\newblock {\em Ann. Probab.}, 25(4):1789--1809, 1997.

\bibitem{Kry}
N.V. Krylov.
\newblock An analytic approach to {SPDE}s.
\newblock In {\em Stochastic partial differential equations: six perspectives},
  volume~64 of {\em Math. Surveys Monogr.}, pages 185--242. Amer. Math. Soc.,
  Providence, RI, 1999.

\bibitem{Kuo}
H.H. Kuo.
\newblock {\em Gaussian measures in {B}anach spaces}.
\newblock Lecture Notes in Mathematics, Vol. 463. Springer-Verlag, Berlin-New
  York, 1975.

\bibitem{LT91}
M.~Ledoux and M.~Talagrand.
\newblock {\em Probability in {B}anach spaces}, volume~23 of {\em Ergebnisse
  der Mathematik und ihrer Grenzgebiete (3)}.
\newblock Springer-Verlag, Berlin, 1991.
\newblock Isoperimetry and processes.

\bibitem{Lenglart}
E.~Lenglart.
\newblock Relation de domination entre deux processus.
\newblock {\em Ann. Inst. H. Poincar\'e Sect. B (N.S.)}, 13(2):171--179, 1977.

\bibitem{LinPhD}
N.~Lindemulder.
\newblock {\em Weighted Function Spaces with Applications to Boundary Value
  Problems}.
\newblock PhD thesis, Delft University of Technology, 2019.

\bibitem{LVY18}
N.~Lindemulder, M.C. Veraar, and I.S. Yaroslavtsev.
\newblock The {UMD} property for {M}usielak-{O}rlicz spaces.
\newblock In {\em Positivity and noncommutative analysis}, Trends Math., pages
  349--363. Birkh\"{a}user/Springer, Cham, [2019] \copyright 2019.

\bibitem{LTCBsII}
J.~Lindenstrauss and L.~Tzafriri.
\newblock {\em Classical {B}anach spaces. {II}}, volume~97 of {\em Ergebnisse
  der Mathematik und ihrer Grenzgebiete}.
\newblock Springer-Verlag, Berlin-New York, 1979.
\newblock Function spaces.

\bibitem{Mar13}
C.~Marinelli.
\newblock {On maximal inequalities for purely discontinuous $L_q$-valued
  martingales}.
\newblock {\em arXiv:1311.7120}, 2013.

\bibitem{MarRo16}
C.~Marinelli and M.~R\"ockner.
\newblock On the maximal inequalities of {B}urkholder, {D}avis and {G}undy.
\newblock {\em Expo. Math.}, 34(1):1--26, 2016.

\bibitem{MP76}
B.~Maurey and G.~Pisier.
\newblock S\'eries de variables al\'eatoires vectorielles ind\'ependantes et
  propri\'et\'es g\'eom\'etriques des espaces de {B}anach.
\newblock {\em Studia Math.}, 58(1):45--90, 1976.

\bibitem{MetSemi}
M.~M\'etivier.
\newblock {\em Semimartingales}, volume~2 of {\em de Gruyter Studies in
  Mathematics}.
\newblock Walter de Gruyter \&\ Co., Berlin-New York, 1982.
\newblock A course on stochastic processes.

\bibitem{MP}
M.~M{\'e}tivier and J.~Pellaumail.
\newblock {\em Stochastic integration}.
\newblock Academic Press [Harcourt Brace Jovanovich, Publishers], New
  York-London-Toronto, Ont., 1980.
\newblock Probability and Mathematical Statistics.

\bibitem{Mey77}
P.A. Meyer.
\newblock Notes sur les int\'egrales stochastiques. {I}. {I}nt\'egrales
  hilbertiennes.
\newblock pages 446--462. Lecture Notes in Math., Vol. 581, 1977.

\bibitem{M-N91}
P.~Meyer-Nieberg.
\newblock {\em Banach lattices}.
\newblock Universitext. Springer-Verlag, Berlin, 1991.

\bibitem{Ngamma}
J.M.A.M.~van Neerven.
\newblock {$\gamma$}-radonifying operators---a survey.
\newblock In {\em The {AMSI}-{ANU} {W}orkshop on {S}pectral {T}heory and
  {H}armonic {A}nalysis}, volume~44 of {\em Proc. Centre Math. Appl. Austral.
  Nat. Univ.}, pages 1--61. Austral. Nat. Univ., Canberra, 2010.

\bibitem{NVW}
J.M.A.M.~van Neerven, M.~C. Veraar, and L.W. Weis.
\newblock Stochastic integration in {UMD} {B}anach spaces.
\newblock {\em Ann. Probab.}, 35(4):1438--1478, 2007.

\bibitem{NVW1}
J.M.A.M.~van Neerven, M.~C. Veraar, and L.W. Weis.
\newblock Stochastic evolution equations in {UMD} {B}anach spaces.
\newblock {\em J. Funct. Anal.}, 255(4):940--993, 2008.

\bibitem{NW1}
J.M.A.M.~van Neerven and L.W. Weis.
\newblock Stochastic integration of functions with values in a {B}anach space.
\newblock {\em Studia Math.}, 166(2):131--170, 2005.

\bibitem{Nov75}
A.A. Novikov.
\newblock Discontinuous martingales.
\newblock {\em Teor. Verojatnost. i Primemen.}, 20:13--28, 1975.

\bibitem{Os11b}
A.~Os{\c{e}}kowski.
\newblock On relaxing the assumption of differential subordination in some
  martingale inequalities.
\newblock {\em Electron. Commun. Probab.}, 16:9--21, 2011.

\bibitem{Os12}
A.~Os{\polhk{e}}kowski.
\newblock {\em Sharp martingale and semimartingale inequalities}, volume~72 of
  {\em Instytut Matematyczny Polskiej Akademii Nauk. Monografie Matematyczne
  (New Series)}.
\newblock Birkh\"auser/Springer Basel AG, Basel, 2012.

\bibitem{OY18}
A.~Os{\c{e}}kowski and I.S. Yaroslavtsev.
\newblock The {H}ilbert transform and orthogonal martingales in {B}anach
  spaces.
\newblock {\em Int. Math. Res. Not. IMRN}, in press.

\bibitem{Pis16}
G.~Pisier.
\newblock {\em Martingales in Banach spaces}, volume 155.
\newblock Cambridge University Press, 2016.

\bibitem{Prot}
P.E. Protter.
\newblock {\em Stochastic integration and differential equations}, volume~21 of
  {\em Stochastic Modelling and Applied Probability}.
\newblock Springer-Verlag, Berlin, 2005.
\newblock Second edition. Version 2.1, Corrected third printing.

\bibitem{RY}
D.~Revuz and M.~Yor.
\newblock {\em Continuous martingales and {B}rownian motion}, volume 293 of
  {\em Grundlehren der Mathematischen Wissenschaften}.
\newblock Springer-Verlag, Berlin, third edition, 1999.

\bibitem{Roz90}
B.L. Rozovski\u\i.
\newblock {\em Stochastic evolution systems}, volume~35 of {\em Mathematics and
  its Applications (Soviet Series)}.
\newblock Kluwer Academic Publishers Group, Dordrecht, 1990.
\newblock Linear theory and applications to nonlinear filtering, Translated
  from the Russian by A. Yarkho.

\bibitem{Rubio86}
J.L. Rubio~de Francia.
\newblock Martingale and integral transforms of {B}anach space valued
  functions.
\newblock In {\em Probability and {B}anach spaces ({Z}aragoza, 1985)}, volume
  1221 of {\em Lecture Notes in Math.}, pages 195--222. Springer, Berlin, 1986.

\bibitem{San96}
F.~Santos.
\newblock Inscribing a symmetric body in an ellipse.
\newblock {\em Inform. Process. Lett.}, 59(4):175--178, 1996.

\bibitem{Sato}
K.-i. Sato.
\newblock {\em L\'evy processes and infinitely divisible distributions},
  volume~68 of {\em Cambridge Studies in Advanced Mathematics}.
\newblock Cambridge University Press, Cambridge, 2013.
\newblock Translated from the 1990 Japanese original, Revised edition of the
  1999 English translation.

\bibitem{SR13}
I.R. Shafarevich and A.O. Remizov.
\newblock {\em Linear algebra and geometry}.
\newblock Springer, Heidelberg, 2013.
\newblock Translated from the 2009 Russian original by David Kramer and Lena
  Nekludova.

\bibitem{SC02}
A.N. Shiryaev and A.S. Cherny\u{\i}.
\newblock A vector stochastic integral and the fundamental theorem of asset
  pricing.
\newblock {\em Tr. Mat. Inst. Steklova}, 237(Stokhast. Finans. Mat.):12--56,
  2002.

\bibitem{vN98}
J.M.A.M. van Neerven.
\newblock Nonsymmetric {O}rnstein-{U}hlenbeck semigroups in {B}anach spaces.
\newblock {\em J. Funct. Anal.}, 155(2):495--535, 1998.

\bibitem{VerPhD}
M.C. Veraar.
\newblock {\em Stochastic integration in Banach spaces and applications to
  parabolic evolution equations}.
\newblock PhD thesis, TU Delft, Delft University of Technology, 2006.

\bibitem{Ver}
M.C. Veraar.
\newblock Continuous local martingales and stochastic integration in {UMD}
  {B}anach spaces.
\newblock {\em Stochastics}, 79(6):601--618, 2007.

\bibitem{Ver07}
M.C. Veraar.
\newblock Randomized {UMD} {B}anach spaces and decoupling inequalities for
  stochastic integrals.
\newblock {\em Proc. Amer. Math. Soc.}, 135(5):1477--1486, 2007.

\bibitem{VY16}
M.C. Veraar and I.S. Yaroslavtsev.
\newblock Cylindrical continuous martingales and stochastic integration in
  infinite dimensions.
\newblock {\em Electron. J. Probab.}, 21:Paper No. 59, 53, 2016.

\bibitem{VY18}
M.C. Veraar and I.S. Yaroslavtsev.
\newblock Pointwise properties of martingales with values in {B}anach function
  spaces.
\newblock In {\em High Dimensional Probability VIII}, pages 321--340. Springer,
  2019.

\bibitem{Weisz92}
F.~Weisz.
\newblock Martingale {H}ardy spaces with continuous time.
\newblock In {\em Probability theory and applications}, volume~80 of {\em Math.
  Appl.}, pages 47--75. Kluwer Acad. Publ., Dordrecht, 1992.

\bibitem{Y17FourUMD}
I.S. Yaroslavtsev.
\newblock Fourier multipliers and weak differential subordination of
  martingales in {UMD} {B}anach spaces.
\newblock {\em Studia Math.}, 243(3):269--301, 2018.

\bibitem{Y19}
I.S. Yaroslavtsev.
\newblock Local characteristics and tangency of vector-valued martingales.
\newblock {\em arXiv:1907.11588}, 2019.

\bibitem{Y17MartDec}
I.S. Yaroslavtsev.
\newblock Martingale decompositions and weak differential subordination in
  {UMD} {B}anach spaces.
\newblock {\em Bernoulli}, 25(3):1659--1689, 2019.

\bibitem{Y17GMY}
I.S. Yaroslavtsev.
\newblock On the martingale decompositions of {G}undy, {M}eyer, and {Y}oeurp in
  infinite dimensions.
\newblock {\em Ann. Inst. Henri Poincar\'{e} Probab. Stat.}, 55(4):1988--2018,
  2019.

\bibitem{Zaa}
A.C. Zaanen.
\newblock {\em Integration}.
\newblock North-Holland Publishing Co., Amsterdam; Interscience Publishers John
  Wiley \& Sons, Inc., New York, 1967.
\newblock Completely revised edition of An introduction to the theory of
  integration.

\end{thebibliography}

\def\cprime{$'$} \def\polhk#1{\setbox0=\hbox{#1}{\ooalign{\hidewidth
  \lower1.5ex\hbox{`}\hidewidth\crcr\unhbox0}}}
  \def\polhk#1{\setbox0=\hbox{#1}{\ooalign{\hidewidth
  \lower1.5ex\hbox{`}\hidewidth\crcr\unhbox0}}} \def\cprime{$'$}

\end{document}